\documentclass[10pt]{article}
\usepackage{amsfonts,amssymb,amsmath,amsthm,mathtools}
\usepackage{array}
\usepackage[margin=2cm]{geometry}
\usepackage{cite}
\usepackage[english]{babel}
\usepackage[T1]{fontenc}
\usepackage{microtype}
\usepackage[hidelinks]{hyperref}
\hypersetup{
   colorlinks=true,
  linkcolor=blue,
  citecolor=blue,
  urlcolor=blue,
  pdftitle={Quantitative mean-field limits for repulsive Coulomb flows at bounded density and Riesz weak--strong stability},
  pdfauthor={Ning Jiang, Zhengyang Qiao, Juntao Wu, Jiangwei Zhang},
  pdfsubject={Bounded-density Coulomb mean-field limits and Riesz weak--strong stability},
  pdfkeywords={mean-field limit, propagation of chaos, bounded density, modulated energy, Coulomb interaction, Riesz interaction, Wasserstein distance, dissipation}
}

\theoremstyle{plain}
\newtheorem{theorem}{Theorem}[section]
\newtheorem{corollary}[theorem]{Corollary}
\newtheorem{proposition}[theorem]{Proposition}
\newtheorem{lemma}[theorem]{Lemma}
\theoremstyle{definition}
\newtheorem{definition}[theorem]{Definition}
\theoremstyle{remark}
\newtheorem{remark}[theorem]{Remark}

\numberwithin{equation}{section}

\def\R{\mathbb{R}}
\newcommand{\dd}{\,\mathrm{d}}
\newcommand{\eps}{\varepsilon}

\begin{document}

\title{\Large\bfseries Quantitative mean-field limits for repulsive Coulomb flows at bounded density and Riesz weak--strong stability}
\author{Ning Jiang$^\mathrm{a}$,\quad Zhengyang Qiao$^\mathrm{b}$,\quad Juntao Wu$^\mathrm{a}$\thanks{Corresponding author. Email: 00036371@whu.edu.cn},\quad Jiangwei Zhang$^\mathrm{c}$\\
  {\small\itshape $^\mathrm{a}$ School of Mathematics and Statistics, Wuhan University, Wuhan, Hubei 430072, P. R. China}\\
  {\small\itshape $^\mathrm{b}$ College of Intelligence Science and Technology, National University of Defense Technology,}\\
  {\small\itshape Changsha 410073, P. R. China}\\
  {\small\itshape $^\mathrm{c}$ Institute of Applied Physics and Computational Mathematics, Beijing 100088, P. R. China}}
\date{}
\maketitle

\begin{abstract}
We establish quantitative mean-field convergence and propagation of chaos for
repulsive Coulomb gradient flows at the bounded-density regularity of the
limiting equation.  The argument combines the dissipative modulated energy
identity with the normalized quadratic transport cost of the full
$N$-particle law.  The coupled differential identities retain a negative term given by the mean-square
force error, which enters an exact completion of squares after mollification of
the reference field.  The resulting identity controls the commutator of the
non-Lipschitz remainder by this negative quadratic term and lower-order errors,
while the sharp first-order commutator estimate is applied to the mollified
Lipschitz field.
For the Coulomb equation, the sharp $L^\infty$ decay yields the deterministic
envelope
\[
 m(t)=\frac{\|\rho_0\|_{L^\infty}}{1+t\|\rho_0\|_{L^\infty}},
 \qquad m'=-m^2.
\]
Weighting the transport cost by this envelope and mollifying at the associated
density-dependent scale $m(t)^{-1/d}$ lead to an Osgood comparison with coefficient proportional
to $m(t)$.  Consequently, for every $d\ge2$ and
$\rho_0\in\mathcal P_2(\mathbb R^d)\cap L^\infty(\mathbb R^d)$, we obtain a
quantitative comparison with the global bounded-density Coulomb solution on
every prescribed finite interval.  For tensorized initial data, the normalized
squared Wasserstein distance of the full $N$-particle law, the expected
modulated energy with its finite-$N$ correction, and the time integral of the
mean-square force error are bounded by $N^{-2\gamma_{T,d}/d}$ for $d\ge3$ and
by $((1+\log N)/N)^{\gamma_{T,2}}$ for $d=2$, where
\[
 \gamma_{T,d}=\bigl(1+T\|\rho_0\|_{L^\infty}\bigr)^{-c_d}.
\]
For $d-2<s<d$, we also prove Riesz weak--strong stability for prescribed
reference solutions in $L^\infty(0,T;B^{s-d+2}_{\infty,q})$.  The cases
$q=1$, $1<q<\infty$, and $q=\infty$ yield Gronwall, Bihari, and Osgood
comparisons, respectively, together with uniqueness in the stated Besov
class.  Finally, an outlier construction on $\mathbb R^d$ separates
modulated energy convergence and Kac chaos from normalized Wasserstein
convergence of the full $N$-particle law.
\end{abstract}



\section{Introduction}

\subsection{Model and motivation}
The derivation of effective continuum equations from large systems of
interacting particles is a central problem in kinetic theory, statistical
mechanics, and the analysis of collective systems.  In the mean-field regime,
each particle interacts weakly with all the others, while the cumulative
interaction remains of order one.  The limiting description replaces the
empirical field generated by finitely many particles with the self-consistent
field generated by a one-particle density.  This mean-field description goes
back to Kac and McKean and underlies the classical works of Braun--Hepp,
Dobrushin, and Sznitman; see
\cite{Kac1956,McKean1967,Braun,Dobrushin1,Sznitman1991} and the surveys
\cite{Jabin2014,Golse2016,ChaintronDiezI2022,ChaintronDiezII2022,CarrilloChoi2021}.

In this paper, we consider $N$ indistinguishable particles in $\R^d$
interacting through an even repulsive potential $g$,
\begin{equation}\label{eq:particle-system}
 \dot x_i(t)=-\frac1N\sum_{j\ne i}\nabla g(x_i(t)-x_j(t)),
 \qquad 1\le i\le N,
\end{equation}
and the corresponding mean-field equation
\begin{equation}\label{eq:mean-field-equation}
 \partial_t\rho=\nabla\!\cdot(\rho u),
 \qquad u=\nabla g*\rho.
\end{equation}
The characteristic velocity in \eqref{eq:mean-field-equation} is $-u$.
Formally, both equations are gradient flows of the same interaction energy at
the discrete and continuum levels.  This variational structure places the
problem within the broader theory of Wasserstein gradient flows and nonlocal
aggregation equations \cite{AmbrosioAGS,Serfaty111,CarrilloChoi2021}.
For Coulomb and Riesz interactions, the limiting equation is a singular
nonlocal Wasserstein gradient flow.  In several regimes it is closely related
to nonlocal and fractional porous medium type equations
\cite{Serfaty111,ChoiJeong}.

The principal examples are the Newtonian and logarithmic Coulomb kernels
\[
 g(x)=\begin{cases}
 -(2\pi)^{-1}\log|x|,&d=2,\\
 [(d-2)|\mathbb S^{d-1}|]^{-1}|x|^{2-d},&d\ge3,
 \end{cases}
 \qquad -\Delta g=\delta_0,
\]
and the Riesz family
\[
 g_s(x)=\frac1s|x|^{-s},\qquad 0<s<d.
\]
The singularity of $\nabla g$ at the origin is the central difficulty.  For
globally Lipschitz forces, one can compare characteristics directly and close
a Gronwall estimate in Wasserstein distance \cite{Dobrushin1}.  Mildly
singular forces can still be treated by refined trajectory and Wasserstein
arguments
\cite{HaurayJabin2007,HaurayJabin2015,Hauray2009,BermanOnnheim2019}.
At the Coulomb scale and beyond, however, the derivative of the force is too
singular for a direct Dobrushin estimate at the natural regularity of the
limiting density.  Moreover, the empirical measure is atomic, so the
interaction energy must be renormalized by removing the diagonal
\cite{Duerinckx2016,Serfaty1}.  These two features explain why Coulomb and
super-Coulomb mean-field limits require methods adapted to the energetic
structure of the kernel.

\subsection{Mean-field convergence and propagation of chaos}
For a configuration $X_N=(x_1,\ldots,x_N)$, let
\[
 \mu_N=\frac1N\sum_{i=1}^N\delta_{x_i}
\]
be its empirical measure.  A deterministic mean-field result concerns the
convergence of $\mu_N(t)$ to $\rho_t$ for a sequence of configurations satisfying the corresponding quantitative
initial bounds.  Propagation of chaos is the corresponding statement for
symmetric $N$-particle laws.  Following Kac, a sequence $(\rho_N(t))_{N\ge1}$
is called $\rho_t$-chaotic if, for every fixed $k$,
\[
 \rho_{N,k}(t)\longrightarrow \rho_t^{\otimes k}
 \qquad\text{as }N\to\infty.
\]
Relations among convergence of fixed marginals, empirical measure convergence,
entropic chaos, and Wasserstein formulations are discussed in
\cite{Sznitman1991,HaurayMischler2014,ChaintronDiezI2022}.

The estimates below are formulated at the level of the $N$-particle law.  They
control the normalized squared Wasserstein distance to $\rho_t^{\otimes N}$
together with the expected modulated energy.  Tensorized initial data yield
quantitative propagation of chaos, while correlated initial laws are covered
under corresponding quantitative bounds on these two quantities.  Their distinct
roles are described in Subsection~\ref{subsec:wasserstein-modulated}.

\subsection{Related work}
The classical mean-field theory for regular interactions is based on stability
of characteristics and Wasserstein distances
\cite{Braun,Dobrushin1,Sznitman1991}.  For singular forces, trajectory methods
remain effective below the Coulomb threshold
\cite{HaurayJabin2007,Hauray2009,HaurayJabin2015}; recent quantitative
$W_p$ estimates for non-attractive first-order systems in this range were
obtained in \cite{HoferSchubert2025}.

For Coulomb and Riesz systems, Duerinckx developed modulated energy arguments
for dissipative Riesz flows, and Serfaty established the method for Coulomb and
super-Coulomb interactions in arbitrary dimension
\cite{Duerinckx2016,Serfaty2017,Serfaty1}.  Nguyen, Rosenzweig, and Serfaty
extended the framework to the full potential Riesz range and more general
Riesz-type interactions \cite{NguyenRosenzweigSerfaty2022}.  The global
bounded-density theory and the sharp $L^\infty$ decay for the limiting Coulomb
equation are due to Serfaty and V\'azquez \cite{Serfaty111}.  At this
regularity the Coulomb field is Zygmund and log-Lipschitz, whereas the standard
multiplicative commutator estimate is naturally formulated for Lipschitz
comparison fields.  The Coulomb argument developed below combines quadratic
transport with the dissipative modulated energy identity to bridge this
regularity gap.

Sharp finite-$N$ commutator estimates were developed further by Rosenzweig and
Serfaty and by Hess-Childs, Rosenzweig, and Serfaty
\cite{Rosenzweig11222,HessChildsRosenzweigSerfaty2025}.  Almost-Lipschitz
transport fields under scaling-critical Sobolev assumptions were subsequently
studied in \cite{HessChildsRosenzweigSerfaty2026}.  In the present argument the
sharp first-order estimate is applied to a mollified Lipschitz field, while the
non-Lipschitz remainder is treated together with the remaining negative
quadratic term involving the force error.  The finite-$N$ scales in the main estimates are
inherited from the sharp lower-bound and first-order commutator estimates; the
time-dependent loss is generated by the
Bihari or Osgood comparison.

At scaling-critical regularity, Rosenzweig obtained quantitative convergence
for two-dimensional point vortices with bounded vorticity and for
higher-dimensional conservative Coulomb flows with bounded density
\cite{Rosenzweig2022Vortices,Rosenzweig2022Coulomb}.  The latter convergence
result is formulated on a short time interval.  For the repulsive dissipative
Coulomb flow considered here, the density decay and the remaining force
dissipation yield quantitative comparison on every prescribed finite interval.
Uniform-in-time estimates are available in other settings, including diffusive
Riesz flows and periodic singular Riesz flows
\cite{RosenzweigSerfaty2023,ChodronRosenzweigSerfaty2025}; the attractive
logarithmic gas was treated in \cite{ChodronRosenzweigSerfatyLog2025}.
Nguyen and Serfaty introduced a multiscale mollification metric for a broader
class of singular first-order systems, including interactions that need not
arise from a potential \cite{NguyenSerfaty2026}.  Their results cover the
maximal smooth lifespan in the sub-Coulomb regime and low-dimensional Coulomb
cases, with shorter time scales in higher-dimensional Coulomb and
super-Coulomb regimes.  The present Coulomb theorem instead exploits the
repulsive gradient-flow sign and the bounded-density decay to obtain, on arbitrary
prescribed finite intervals, quantitative estimates for the Wasserstein distance
on the full $N$-particle law, the modulated energy, and the time-integrated
mean-square force error.

Ben-Porat, Carrillo, and Jabin studied singular mean-field dynamics with
time-dependent particle weights \cite{BenPoratCarrilloJabin2026}, where the
weight evolution leads to an additional source term and a corresponding
functional inequality.  Relative entropy and modulated free energy methods
provide a complementary route to quantitative propagation of chaos.  Jabin and
Wang developed such estimates for Vlasov and stochastic particle systems with
rough interactions, and subsequent works extended the approach to broader
classes of singular kernels, including attractive interactions and
Patlak--Keller--Segel type models
\cite{JabinWang2016,JabinWang2018,BreschJabinWang2019,BreschJabinWang2020,
BreschJabinWang2023}.  For diffusive systems, entropy and hierarchy methods
have led to quantitative, and in some regimes uniform-in-time, propagation of
chaos
\cite{FournierHaurayMischler2014,Lacker2023,LackerLeFlem2023,
BreschJabinSoler2025}; related relative-entropy techniques have also been
adapted to Landau particle approximations
\cite{CarrilloFengGuoJabinWang2026}.  A complementary line of work quantifies
higher-order correlations more directly through tools such as Glauber calculus,
correlation hierarchies, and duality
\cite{Duerinckx2021,BejarLopezBlausteinJabinSoler2026,KhouryJabin2026}.  In the
second-order kinetic setting, Duerinckx and Jabin recently derived the two-dimensional
Vlasov--Poisson limit for classical Coulomb particles without a microscopic
cutoff by a dual hierarchy and refined BBGKY correlation analysis
\cite{DuerinckxJabin2026}.  Wasserstein methods for gradient flows under
convexity assumptions are discussed in \cite{BermanOnnheim2019}.

\subsection{Main contributions}
The main results consist of a bounded-density Coulomb theorem, a Riesz
weak--strong stability theorem, and a counterexample concerning the
Wasserstein distance of the full $N$-particle law.
\begin{enumerate}
\item \emph{Coulomb flows at bounded density.}
For every $d\ge2$ and $\rho_0\in\mathcal P_2(\mathbb R^d)\cap L^\infty$, we
obtain a quantitative comparison between the $N$-particle law and the global
Coulomb solution at the natural bounded-density regularity on every prescribed
finite interval.  The estimate simultaneously controls the normalized
Wasserstein distance of the full $N$-particle law, the expected modulated
energy with its finite-$N$ correction, and the time integral of the mean-square
force error.  The proof retains a negative term given by the mean-square force error in the combined
transport--energy inequality, uses it through an exact completion of squares,
and combines the result with a time-weighted quadratic transport estimate based
on the Riccati decay of $m(t)$ and mollification at the scale $m(t)^{-1/d}$.  For
tensorized data the convergence rates are stated in
Theorem~\ref{thm:bounded-density-coulomb-stability}; the time exponent may be
chosen as $(1+T\|\rho_0\|_\infty)^{-c_d}$ with $c_d$ depending only on the
dimension.

\item \emph{A super-Coulomb Riesz extension.}
For $d-2<s<d$, the same combination of quadratic transport and
modulated energy dissipation yields weak--strong stability for prescribed reference solutions in
$B^{s-d+2}_{\infty,q}$, $1\le q\le\infty$.  The cases $q=1$,
$1<q<\infty$, and $q=\infty$ lead to Gronwall, Bihari, and Osgood comparisons,
respectively.  The particle approximation also gives uniqueness within the
stated Besov class.

\item \emph{The Wasserstein distance for the full $N$-particle law.}
On $\mathbb R^d$ we construct symmetric Kac-chaotic laws for which the modulated
energy with its finite-$N$ correction tends to zero while the normalized
$N$-particle Wasserstein distance from $\rho^{\otimes N}$ stays bounded away
from zero.
\end{enumerate}

The two main results rely on the coupled estimates described next.

\subsection{Coupling the quadratic transport cost and the modulated energy}
\label{subsec:dissipative-transport-energy-closure}
\paragraph{Combined transport--energy estimate.}
The quadratic transport estimate and the dissipative modulated energy identity
along the same trajectories take the form
\begin{equation}\label{eq:intro-dissipative-closure-pair}
 Q_N'\le r_N+\omega(Q_N),\qquad F_N'=-2r_N-C_N[u].
\end{equation}
Their sum retains the negative contribution $-r_N$.  Writing
$u=u_\eps+w_\eps$, with $u_\eps$ a mollification of the reference field, the
exact identity
\begin{equation}\label{eq:intro-dissipative-closure-square}
 -r_N-C_N[w_\eps]
 =-r_N[u_\eps]+\|w_\eps\|_{L^2(\mu_N)}^2
 +2\int w_\eps\cdot(K_N-u)\rho
\end{equation}
incorporates the commutator of the non-Lipschitz remainder into the remaining
negative quadratic term involving the force error.  The sharp first-order commutator estimate is then applied to
$u_\eps$, while the last two terms are controlled by the mollification bounds
and estimates of the continuous pairing term.  Together these estimates close
the coupled differential inequality used below.

\paragraph{Weighted quadratic transport and choice of mollification scale.}
For the Coulomb theorem the sharp $L^\infty$ decay estimate supplies a deterministic envelope with $m'=-m^2$.  The
weighted quadratic transport cost $\widehat mQ_N$ cancels the quadratic
contribution generated by Young's inequality.  The same density bound determines
the scale $R=m^{-1/d}$ and the mollification radius
$\eps_\vartheta=R\vartheta$.  After the inequalities at a fixed scale are
averaged, a deterministic dyadic parameter is chosen from the averaged error.
The resulting Osgood term has coefficient $C_d m(t)$, whereas
$m(t)^{2-2/d}$ appears only in a linear term handled by an integrating factor.

\subsection{Wasserstein distance and modulated energy}\label{subsec:wasserstein-modulated}
The quadratic transport cost and the modulated energy measure different aspects
of convergence on $\mathbb R^d$.  For a configuration $X_N$ with empirical
measure $\mu_N$, the modulated energy is
\[
 F_N(X_N,\rho)
 = {\iint}_{(\R^d)^2\setminus\Delta_2}
 g(x-y)\,\dd(\mu_N-\rho)(x)\dd(\mu_N-\rho)(y).
\]
It is adapted to the singular interaction after removal of the atomic
self-interaction.  The normalized Wasserstein distance additionally detects
the quadratic cost of transporting small amounts of mass far from the bulk:

\[
 \frac1N W_2^2(\rho_N(t),\rho_t^{\otimes N})
\]
is the mean quadratic transport cost per particle.  Symmetry gives
\begin{equation}\label{eq:intro-W2-marginal-control}
 \frac1kW_2^2(\rho_{N,k}(t),\rho_t^{\otimes k})
 \le \frac1N W_2^2(\rho_N(t),\rho_t^{\otimes N}),
 \qquad 1\le k\le N,
\end{equation}
so a bound for the full $N$-particle law immediately controls every fixed
marginal.

The distinction is genuine on $\mathbb R^d$.
Proposition~\ref{prop:whole-space-modulated energy-obstruction} constructs
symmetric Kac-chaotic laws for which
\[
 \int(F_N+\eta_N)\,\dd\rho_N\longrightarrow0,
 \qquad
 \liminf_{N\to\infty}\frac1N
 W_2^2(\rho_N,\rho^{\otimes N})>0.
\]
The example places mass $1/N$ at distance of order $\sqrt N$.  This mass is
invisible in every fixed marginal in the limit and contributes negligibly to
the modulated energy, but it produces a contribution of order one to the
normalized quadratic transport cost.  The logarithmic analogue and the
compact case are discussed in Remark~\ref{rem:whole-space-vs-compact-W2}.

\subsection{Outline of the proof and organization}
After stating the main results, Section~\ref{sec:abstract-framework} derives the
transport and modulated energy identities and the completion of squares
identity.  Section~\ref{sec:particle-law-comparison} records the averaging and
comparison tools.  We then turn immediately to the primary application: the
global Coulomb theory for bounded densities is recalled in
Section~\ref{sec:coulomb-reference-solutions}, and the Coulomb mean-field theorem
is proved in the following section using the weighted transport estimate and
mollification at the scale $R(t)=m(t)^{-1/d}$.  Section~\ref{sec:critical-riesz-proof} develops the super-Coulomb Riesz
extension under the stated Besov regularity.  The counterexample and the consistency estimate for the empirical continuity
equation follow afterwards.

The appendices contain technical proofs for the Riesz extension, the principal
value argument, Coulomb particle dynamics, measurability and integrability,
the Riesz singular chain rule and the normalization of the imported finite-$N$
inequalities.

\section{Main results}\label{sec:main-results}
All constants are independent of $N$.  Generic constants $C$ may change from
line to line, with subscripts indicating their allowed dependence; the constants
$A_{d,s}$, $A_d$, and $c_d$ in the finite-$N$ corrections and Coulomb exponent
are fixed once and for all.  Unless stated otherwise, the $N$-particle laws are
symmetric, have finite second moment, and give no mass to the collision set.
The finite-$N$ lower bounds
and commutator estimates used below, in the normalization adopted here, are
collected in Appendix~\ref{sec:static-inputs}.  For the tensorized initial law $\rho_0^{\otimes N}$, the
estimates give explicit rates of propagation of chaos.  The comparison estimates
also apply to correlated initial laws satisfying the corresponding
quantitative initial assumptions.

\subsection{Notation for the theorem statements}
For $N\ge2$ define the collision set and collision-free configuration space by
\[
 \Delta_N:=\{X_N=(x_1,\ldots,x_N)\in(\R^d)^N:
 x_i=x_j\ \text{for some }i\ne j\},\qquad
 \Omega_N:=(\R^d)^N\setminus\Delta_N,
\]
and write $\Delta_2:=\{(x,x):x\in\R^d\}$.  We use the homogeneous Zygmund
seminorm
\[
 [f]_{\Lambda_*}:=
 \sup_{x\in\R^d,\,h\ne0}
 \frac{|f(x+h)+f(x-h)-2f(x)|}{|h|},
\]
and the convention $\log_+ r:=\max\{\log r,0\}$ for $r>0$.
For $P,Q\in\mathcal P_2((\R^d)^N)$, $W_2(P,Q)$ denotes the quadratic
Wasserstein distance.  We use the normalized squared distance
\[
 \frac1N W_2^2(P,Q)
 =\inf_{\pi\text{ coupling }P\text{ and }Q}
 \int\frac1N\sum_{i=1}^N|x_i-y_i|^2\,\dd\pi(X_N,Y_N).
\]
Section~\ref{sec:abstract-framework} defines, for a translation-invariant
potential $W$, the quantities $K_i^W$, $u^W$, $F_N^W$, $r_N^W$,
$C_N^W(X_N,\rho;v)$, and $Q_N$.  Once the kernel is fixed, we omit the
superscript.  At the level of the full $N$-particle law we write
\begin{equation}\label{eq:common-law-observables}
 \overline r_N(t):=\int r_N(t,X_N)\,\dd\rho_N(t)(X_N).
\end{equation}
We write $\rho_t:=\rho(t,\cdot)$ and $u_t:=u(t,\cdot)$ for time slices.
The following notation will be used throughout the theorem statements.
\begin{itemize}
\item $\rho_t$ is the limiting one-particle density, whereas $\rho_N(t)$ is the $N$-particle law on $(\R^d)^N$; $\mu_N(X_N)$ is the empirical measure of one configuration sampled from that law.
\item $K_i(X_N)$ is the discrete interaction field evaluated at particle $i$; when needed, $K_N(x)=N^{-1}\sum_j\nabla g(x-x_j)$ denotes the empirical field away from the particle positions.
\item $Q_N$ denotes the normalized quadratic transportation cost along a coupled trajectory.  Its expectation under the coupling bounds the normalized squared Wasserstein distance between the $N$-particle law and the tensorized law defined above.
\item $F_N$ is the modulated energy,
$r_N=N^{-1}\sum_i|K_i-u(x_i)|^2$, $\overline r_N$ its expectation with respect
to the $N$-particle law, and $C_N[v]$ the commutator associated with a comparison
vector field $v$.
\end{itemize}

When $d-1\le s<d$, the commutator is defined by the absolutely convergent formula obtained by symmetrization.  Appendix~\ref{sec:principal-value-representation}
shows that centered cutoffs and even regularizations converge to the same
quantity.

\subsection{Coulomb interactions with bounded limiting density}
\label{subsec:natural-coulomb-main-results}
For every $d\ge2$, let $g_d$ be the normalized Coulomb fundamental solution
\begin{equation}\label{eq:unified-coulomb-kernel}
 g_d(x):=
 \begin{cases}
  -(2\pi)^{-1}\log|x|,&d=2,\\[1mm]
  \bigl((d-2)|\mathbb S^{d-1}|\bigr)^{-1}|x|^{2-d},&d\ge3,
 \end{cases}
 \qquad -\Delta g_d=\delta_0.
\end{equation}
Proposition~\ref{prop:bounded-density-flow} gives a unique global limiting
solution for every
$\rho_0\in\mathcal P_2(\R^d)\cap L^\infty(\R^d)$, and the corresponding
interaction field is bounded, Zygmund, and log-Lipschitz.

For a probability density $\mu\in\mathcal P_2(\R^d)\cap L^\infty(\R^d)$,
define the dimension-dependent Coulomb finite-$N$ error term
\begin{equation}\label{eq:coulomb-finite-N-correction}
 \eta_{N,d}(\mu):=A_d
 \begin{cases}
  \dfrac{1+\log(1+N\|\mu\|_{L^\infty})}{N},&d=2,\\[3mm]
  \|\mu\|_{L^\infty}^{1-2/d}N^{-2/d},&d\ge3,
 \end{cases}
 \qquad
 \eta_{N,d}:=\eta_{N,d}(\rho_0),
\end{equation}
where $A_d>0$ is fixed sufficiently large for the finite-$N$ lower bound and the first-order commutator estimate in
Lemma~\ref{lem:unified-coulomb-static}.  The sharp $L^\infty$ decay estimate of
Serfaty--V\'azquez \cite[Theorem~5.1]{Serfaty111} implies
$\eta_{N,d}(\rho_t)\le\eta_{N,d}$ for every $t\ge0$.
Writing $M_0:=\|\rho_0\|_{L^\infty}>0$, the same estimate gives the
deterministic density envelope
\begin{equation}\label{eq:coulomb-density-envelope}
 m(t):=\frac{M_0}{1+tM_0},
 \qquad m'(t)=-m(t)^2,
 \qquad \|\rho_t\|_{L^\infty}\le m(t).
\end{equation}
We also introduce the time-dependent weight
\begin{equation}\label{eq:coulomb-normalized-transport-weight}
 \widehat m(t):=\frac{m(t)}{1+M_0}.
\end{equation}
Then $0<\widehat m(t)<1$ and $\widehat m'=-m\widehat m$.  This normalization
makes the constant in the logarithmic transport estimate independent of $M_0$.
Choose a dimensional constant $c_d>0$ sufficiently large to
dominate the Osgood coefficient in the comparison and define
\begin{equation}\label{eq:coulomb-osgood-exponent}
 \gamma_{T,d}:=
 \exp\!\left[-c_d\int_0^T m(t)\,\dd t\right]
 =\bigl(1+TM_0\bigr)^{-c_d}.
\end{equation}
Thus $\gamma_{T,d}>0$ for every finite $T$, and its polynomial power depends
only on the dimension.  The additional linear term generated by the
mollification remainder affects only the multiplicative constant in the main
estimate, not the exponent $\gamma_{T,d}$.

\begin{theorem}[Quantitative mean-field estimate for Coulomb flows with bounded density]
\label{thm:bounded-density-coulomb-stability}
Let $d\ge2$, $T>0$, and let $\rho$ be the global solution with bounded density
generated by
$\rho_0\in\mathcal P_2(\R^d)\cap L^\infty(\R^d)$ for the kernel
$g_d$.  Let $N\ge2$ and let
$\rho_N^0\in\mathcal P_2((\R^d)^N)$ be symmetric, give no mass to the
collision set, and satisfy
\[
 \int\bigl(F_N(X_N,\rho_0)+\eta_{N,d}\bigr)\,\dd\rho_N^0(X_N)<\infty.
\]
We set
\begin{equation}\label{eq:coulomb-unified-initial-error}
 a_{N,d}:=\frac1N W_2^2(\rho_N^0,\rho_0^{\otimes N})
 +\int\bigl(F_N(X_N,\rho_0)+\eta_{N,d}\bigr)\,\dd\rho_N^0
 +\begin{cases}
  \dfrac{1+\log N}{N},&d=2,\\[2mm]
  N^{-2/d},&d\ge3.
 \end{cases}
\end{equation}
Then there exists $C_{T,d,\rho_0}>0$, independent of $N$, such that
\begin{equation}\label{eq:coulomb-main-estimate}
\begin{aligned}
&\sup_{0\le t\le T}\left[
 \frac1N W_2^2(\rho_N(t),\rho_t^{\otimes N})
 +\int\bigl(F_N(X_N,\rho_t)+\eta_{N,d}\bigr)
 \,\dd\rho_N(t)(X_N)\right]
 +\int_0^T\overline r_N(t)\,\dd t\\
&\hspace{4cm}\le C_{T,d,\rho_0}
 \bigl(a_{N,d}^{\gamma_{T,d}}+a_{N,d}\bigr).
\end{aligned}
\end{equation}
For the tensorized initial law $\rho_N^0=\rho_0^{\otimes N}$,
\begin{equation}\label{eq:coulomb-product-rate}
\begin{aligned}
 &\sup_{0\le t\le T}\left[\frac1N W_2^2(\rho_N(t),\rho_t^{\otimes N})
 +\int\bigl(F_N(X_N,\rho_t)+\eta_{N,d}\bigr)\,\dd\rho_N(t)(X_N)\right]\\
 &\qquad+\int_0^T\overline r_N(t)\,\dd t
 \le C_{T,d,\rho_0}
 \begin{cases}
  \left(\dfrac{1+\log N}{N}\right)^{\gamma_{T,2}},&d=2,\\[3mm]
  N^{-2\gamma_{T,d}/d},&d\ge3.
 \end{cases}
\end{aligned}
\end{equation}
\end{theorem}

\begin{remark}
\label{rem:coulomb-fixed-time-comparison}
The reference solution in Theorem~\ref{thm:bounded-density-coulomb-stability}
is the global bounded-density Coulomb solution.  The time dependence of the
estimate is encoded in
$\gamma_{T,d}=(1+T\|\rho_0\|_\infty)^{-c_d}$, which remains positive on every
finite interval.  The finite-$N$ scales $N^{-2/d}$ for $d\ge3$ and
$(1+\log N)/N$ for $d=2$ arise from the sharp lower bound and first-order
commutator estimate, while the Osgood comparison produces the additional
loss in the convergence exponent.
\end{remark}

\subsection{Weak--strong stability under \texorpdfstring{$B^{s-d+2}_{\infty,q}$}{B(s-d+2; infinity,q)} regularity}
\label{subsec:critical-main-result}
The case $q=1$ is the Lipschitz endpoint, while $1<q\le\infty$ yields
almost-Lipschitz or Zygmund velocity fields.  We use the inhomogeneous
Littlewood--Paley norm
\[\|f\|_{B^\alpha_{\infty,q}}=\bigl\|(2^{j\alpha}\|\Delta_jf\|_{L^\infty})_{j\ge-1}\bigr\|_{\ell^q},
 \qquad 1\le q\le\infty,
\]
and, for $d\ge2$ and $d-2<s<d$, set
\[
 \alpha:=s-d+2\in(0,2).
\]
The exponent $\alpha=s-d+2$ is chosen so that the Riesz interaction
operator maps $B^\alpha_{\infty,q}$ to $B^1_{\infty,q}$ at high
frequencies.  Thus $q=1$ gives the Lipschitz endpoint and $q=\infty$ the
Zygmund endpoint.  This Besov scale is the one used in the velocity estimates
of the weak--strong argument.
\begin{definition}[Reference solution with $B^{s-d+2}_{\infty,\infty}$ regularity]
\label{def:critical-riesz-reference-solution}
A curve $\rho$ is a limiting solution with the prescribed Besov regularity on $[0,T]$ if it is a
nonnegative distributional solution of unit mass of
\[
 \partial_t\rho=\nabla\!\cdot(\rho u),
 \qquad u=\nabla g_s*\rho,
\]
and satisfies $\rho\in C([0,T];\mathcal P_2(\R^d))$ and
\[\rho\in L^\infty\!\left(0,T;L^1\cap L^\infty
 \cap B^\alpha_{\infty,\infty}(\R^d)\right).
\]
The convolution is understood in the sense of distributions.  The representatives
of $\rho$ and $u$ used below are those constructed in
Proposition~\ref{prop:critical-riesz-borel-reference}.  That proposition also shows
that the essential $L^\infty$ and Besov bounds in this definition extend to every
time $t\in[0,T]$.  In particular, each $\rho_t$ admits an $L^\infty$ density,
including at $t=0$, so the norms at individual time slices used in the statements below
are well defined.
\end{definition}

We set
\[
 M_T:=\max\left\{\|\rho_0\|_{L^\infty},
 \operatorname*{ess\,sup}_{0<t<T}\|\rho_t\|_{L^\infty}\right\},
 \qquad
 \eta_N:=A_{d,s}M_T^{s/d}N^{s/d-1}.
\]
Choose $A_{d,s}>0$ large enough that $\eta_N$ dominates the additive finite-$N$
terms in both the lower bound and the first-order commutator estimate of
Lemma~\ref{lem:static-riesz-transfer}; in particular, $F_N+\eta_N\ge0$.
Once $T$ and the reference solution are fixed, so is $\eta_N$ through $M_T$.
Throughout this subsection, $\rho$ is a solution in the sense of
Definition~\ref{def:critical-riesz-reference-solution}.  When $d-1\le s<d$,
the commutator is understood through the absolutely convergent symmetrized
formula of Appendix~\ref{sec:principal-value-representation}.

\begin{theorem}[Weak--strong stability under $B^{s-d+2}_{\infty,q}$ regularity]
\label{thm:critical-riesz-propagation}
Fix $1\le q\le\infty$.  Let $d\ge2$, $d-2<s<d$, $T>0$, $N\ge2$, and let $\rho$ be a limiting solution
satisfying Definition~\ref{def:critical-riesz-reference-solution}.  Let
$\rho_N^0\in\mathcal P_2((\R^d)^N)$ be symmetric, satisfy
$\rho_N^0(\Delta_N)=0$, and assume
\[\int_{(\R^d)^N}\left(F_N(X_N,\rho_0)+\eta_N\right)
 \,\dd \rho_N^0(X_N)<\infty.
\]
Let $\rho_N(t)$ be the pushforward of $\rho_N^0$ under the particle flow, and put
\begin{equation*}
 a_N:=\frac1N W_2^2(\rho_N^0,\rho_0^{\otimes N})
 +\int\bigl(F_N(X_N,\rho_0)+\eta_N\bigr)\,\dd\rho_N^0(X_N)
 +N^{s/d-1}.
\end{equation*}

If $q=\infty$, define
\[\gamma_T:=\exp\!\left[-C_{d,s}\int_0^T
 \left(1+M_T+\|\rho_t\|_{B^\alpha_{\infty,\infty}}\right)^2\,\dd t\right].
\]
Then
\begin{equation}\label{eq:critical-riesz-main-estimate}
\begin{aligned}
&\sup_{0\le t\le T}\left[\frac1N W_2^2(\rho_N(t),\rho_t^{\otimes N})
 +\int\bigl(F_N(X_N,\rho_t)+\eta_N\bigr)\,\dd\rho_N(t)(X_N)\right]
 +\int_0^T \overline r_N(t)\,\dd t\\
&\hspace{4cm}\le C_{T,d,s,\rho}\bigl(a_N^{\gamma_T}+a_N\bigr).
\end{aligned}
\end{equation}

If $1\le q<\infty$, assume in addition
\begin{equation}\label{eq:critical-riesz-finite-q-bound}
 \rho\in L^\infty\!\left(0,T;B^\alpha_{\infty,q}(\R^d)\right).
\end{equation}
Then
\begin{equation}\label{eq:critical-riesz-finite-q-main-estimate}
\begin{aligned}
&\sup_{0\le t\le T}\left[\frac1N W_2^2(\rho_N(t),\rho_t^{\otimes N})
 +\int\bigl(F_N(X_N,\rho_t)+\eta_N\bigr)\,\dd\rho_N(t)(X_N)\right]
 +\int_0^T \overline r_N(t)\,\dd t\\
&\qquad\le C_{T,d,s,q,\rho}\,a_N
 \exp\!\left(C_{T,d,s,q,\rho}
 \bigl(1+\log_+(1/a_N)\bigr)^{1-1/q}\right).
\end{aligned}
\end{equation}
For $q=1$, the exponential factor is absorbed into the constant, so the
right-hand side is simply $C_{T,d,s,\rho}a_N$.

For the tensorized initial law $\rho_N^0=\rho_0^{\otimes N}$, the estimate yields the
explicit rates
\begin{equation}\label{eq:critical-product-rates}
\begin{aligned}
 &\sup_{0\le t\le T}\left[\frac1N W_2^2(\rho_N(t),\rho_t^{\otimes N})
 +\int\bigl(F_N(X_N,\rho_t)+\eta_N\bigr)\,\dd\rho_N(t)(X_N)\right]\\
 &\qquad+\int_0^T\overline r_N(t)\,\dd t
 \le C_{T,d,s,q,\rho}
 \begin{cases}
 N^{-(1-s/d)\gamma_T},&q=\infty,\\
 N^{s/d-1}\exp\!\bigl(C(\log N)^{1-1/q}\bigr),&1<q<\infty,\\
 N^{s/d-1},&q=1.
 \end{cases}
\end{aligned}
\end{equation}
\end{theorem}

\begin{remark}
\label{rem:critical-besov-comparison}
Since $B^1_{\infty,1}\hookrightarrow W^{1,\infty}$, the case $q=1$ gives a
Gronwall estimate.  At the borderline $s=d-1$, the error
$\eps\ell(\eps)$ is balanced by the scale $\eps\simeq z/\ell(z)$ and again
yields a linear comparison.  For $1<q<\infty$, the modulus
$r(1+\log(1/r))^{1-1/q}$ gives the Bihari estimate in
\eqref{eq:critical-product-rates}, while for $q=\infty$ the Zygmund modulus is treated by the Osgood
comparison.  The reference solution is prescribed by
Definition~\ref{def:critical-riesz-reference-solution};
Corollary~\ref{cor:critical-riesz-uniqueness} gives uniqueness within this
class.  The scale $N^{s/d-1}$ is the finite-$N$ correction inherited from the
sharp lower bound and commutator estimate.
\end{remark}

\begin{corollary}[Uniqueness in the $B^{s-d+2}_{\infty,\infty}$ class]
\label{cor:critical-riesz-uniqueness}
Let $d\ge2$ and $d-2<s<d$.  Let $(\rho,u)$ and $(\widetilde\rho,\widetilde u)$ be two solutions satisfying Definition~\ref{def:critical-riesz-reference-solution}, with the same initial density.  Then
\[
 \rho_t=\widetilde\rho_t
 \qquad\text{in }\mathcal P_2(\R^d)
 \quad\text{for every }t\in[0,T].
\]
The limiting equation is unique within the class specified in
Definition~\ref{def:critical-riesz-reference-solution}.
\end{corollary}

\subsection{Tensorized initial laws and marginal estimates}
\begin{corollary}
\label{cor:product-data-and-marginals}
For the tensorized initial law $\rho_N^0=\rho_0^{\otimes N}$, Lemma~\ref{lem:unified-product-initial-data}
gives, for the interaction kernel $g$ of the applicable theorem,
\[
 \mathbb E_{\rho_0^{\otimes N}}F_N(X_N,\rho_0)
 =-\frac1N\iint g(x-y)\rho_0(x)\rho_0(y)\,\dd x\dd y.
\]
For the tensorized initial law, the principal rates are
\eqref{eq:coulomb-product-rate} and \eqref{eq:critical-product-rates}.
For every $1\le k\le N$, the corresponding estimate for the marginal of order $k$ follows
directly from the projection inequality
\eqref{eq:intro-W2-marginal-control}.  In particular, in the Coulomb case with bounded density,
\begin{equation}\label{eq:coulomb-product-marginal-rate}
 \sup_{0\le t\le T}\frac1k
 W_2^2(\rho_{N,k}(t),\rho_t^{\otimes k})
 \le C_{T,d,\rho_0}
 \begin{cases}
  \left(\dfrac{1+\log N}{N}\right)^{\gamma_{T,2}},&d=2,\\[3mm]
  N^{-2\gamma_{T,d}/d},&d\ge3,
 \end{cases}
 \qquad 1\le k\le N.
\end{equation}
\end{corollary}

The proof is given at the end of Section~\ref{sec:particle-law-comparison}, after
the estimate for the $N$-particle law.  A further consequence for the empirical continuity
equation is recorded in Section~\ref{sec:dynamical-consequences}.

\section{Basic identities and estimates}
\label{sec:abstract-framework}
This section establishes the identities used for all kernels considered below.
We define the modulated energy, the force error $r_N$, the commutator, and the
quadratic transportation cost $Q_N$, and derive the completion-of-squares and
chain-rule formulas.  The discrete energy dissipation first gives
$r_N\in L^1$, which provides the integrability required for the subsequent
singular chain rule.  Kernel-dependent finite-$N$ bounds are introduced in the
corresponding applications.

\subsection{Translation-invariant notation}
\label{subsec:translation-invariant-specialization}

All singular applications below are translation invariant, with identity
mobility and no external potential.  Let
$W:\R^d\setminus\{0\}\to\R$ be an even potential, smooth away from the
origin, and set $G=\nabla W$, which is odd.  For a collision-free
configuration $X_N=(x_1,\ldots,x_N)$ define
\begin{equation}\label{eq:unified-force-energy-moment}
 \begin{aligned}
 K_i^W(X_N):=\frac1N\sum_{j\ne i}G(x_i-x_j),\qquad 
 H_N^W(X_N):=\frac1{2N^2}\sum_{i\ne j}W(x_i-x_j),\qquad 
 m_{2,N}(X_N):=\frac1N\sum_{i=1}^N|x_i|^2.
 \end{aligned}
\end{equation}
The particle system and its continuum equation are
\begin{equation}\label{eq:unified-particle-continuum}
 \dot x_i=-K_i^W(X_N),
 \qquad
 \partial_t\rho=\nabla\cdot(\rho u^W),
 \qquad
 u^W=G*\rho.
\end{equation}
When the singular convolution is not absolutely convergent, $u^W$ denotes the
bounded representative required by the applicable limiting solution
hypothesis.  Once $W$ is fixed we suppress the superscript and write $u=u^W$.
For the empirical measure $\mu_N=N^{-1}\sum_i\delta_{x_i}$, define
\begin{equation}\label{eq:unified-FN}
 F_N^W(X_N,\rho)
 :=\iint_{(\R^d)^2\setminus\Delta_2}
 W(x-y)\,\dd(\mu_N-\rho)(x)\dd(\mu_N-\rho)(y).
\end{equation}
For a vector field $v$, define the commutator with all arguments visible by
\[
 C_N^W(X_N,\rho;v)
 :=\iint_{(\R^d)^2\setminus\Delta_2}
 (v(x)-v(y))\cdot G(x-y)\,
 \dd(\mu_N-\rho)(x)\dd(\mu_N-\rho)(y).
\]
When $X_N$ and $\rho$ are fixed we abbreviate this as $C_N^W[v]$; after the
kernel is fixed we also write $C_N(X_N,\rho;v)$ or simply $C_N[v]$.
Away from the atoms of $\mu_N$, write
\[K_N^W(x):=\frac1N\sum_{j=1}^NG(x-x_j).
\]
Here $K_i^W$ is the discrete field at particle $i$ with the self-interaction
removed, whereas $K_N^W(x)$ is the empirical field at a generic point.  We
suppress the superscript $W$ when there is no ambiguity.  For the limiting interaction field $u^W$, set
\begin{equation}\label{eq:unified-force-error}
 r_N^W(X_N,\rho):=\frac1N\sum_{i=1}^N|K_i^W-u^W(x_i)|^2.
\end{equation}
Once the kernel is fixed, we suppress $W$ and write $r_N=r_N^W$.  Thus
$r_N$ is the mean-square error between the discrete force at the particle
positions and the limiting interaction field.
For the Riesz kernel $W=g_s$, one has
\[
 -z\cdot\nabla g_s(z)=s g_s(z)=|z|^{-s},
\]
and the finite-$N$ error term is of order
$\|\rho\|_{L^\infty}^{s/d}N^{s/d-1}$.  For the planar logarithmic kernel,
$-z\cdot\nabla g(z)=(2\pi)^{-1}$ and the correction is of order
$N^{-1}[1+\log(1+N\|\rho\|_{L^\infty})]$.  These scales determine the finite-$N$ error terms that appear in the
application sections; the regularity of the limiting field and the interpretation of the singular integrals are supplied there.

\subsection{Completion of squares}

\begin{proposition}[Completion of squares]
\label{prop:completion-of-squares}
Let $v$ be a bounded vector field, put $w:=u-v$, and define
\[r_N[v]:=\frac1N\sum_{i=1}^N|K_i^W-v(x_i)|^2,
 \qquad
 \|w\|_{L^2(\mu_N)}^2:=\frac1N\sum_{i=1}^N|w(x_i)|^2.
\]
Assume that the quantities below are well defined, either absolutely or by applying the same symmetric regularization to all terms.  Then
\begin{equation}\label{eq:completion-of-squares}
 -r_N-C_N^W[w]
 =-r_N[v]+\|w\|_{L^2(\mu_N)}^2
 +2\int_{\R^d}w(x)\cdot(K_N^W(x)-u(x))\rho(x)\,\dd x.
\end{equation}
In addition,
\begin{equation}\label{eq:force-error-comparison}
 r_N\le2r_N[v]+2\|w\|_{L^2(\mu_N)}^2.
\end{equation}
For $d-1\le s<d$, the common regularization and its limit are described in
Sections~\ref{sec:critical-riesz-proof} and~\ref{sec:principal-value-representation}.
\end{proposition}

\begin{proof}
Oddness of $G$ gives
\begin{equation}\label{eq:abs-rough-identity}
C_N^W[w]=\frac2N\sum_iw(x_i)\cdot(K_i^W-u(x_i))-2\int w\cdot(K_N^W-u)\rho.
\end{equation}
Let $e_i=K_i^W-u(x_i)$.  Using \eqref{eq:abs-rough-identity},
\begin{align*}
-r_N-C_N^W[w]
&=-\frac1N\sum_i\bigl(|e_i|^2+2w(x_i)\cdot e_i\bigr)
 +2\int w\cdot(K_N^W-u)\rho\\
&=-\frac1N\sum_i|e_i+w(x_i)|^2
 +\frac1N\sum_i|w(x_i)|^2
 +2\int w\cdot(K_N^W-u)\rho,
\end{align*}
which proves \eqref{eq:completion-of-squares}.  In addition,
\[
 |e_i|^2\le2|e_i+w(x_i)|^2+2|w(x_i)|^2,
\]
which yields \eqref{eq:force-error-comparison}.  For singular kernels, the same calculation is first performed at fixed symmetric
regularization and then passed to the $L^1(0,T)$ limit.
\end{proof}

We repeatedly use the following standard estimate.
\begin{lemma}
\label{lem:near-far-locally-integrable-field}
Let $G:\R^d\setminus\{0\}\to\R^d$ satisfy, for some $0\le\beta<d$,
\[
 |G(z)|\le C_G\bigl(|z|^{-\beta}\mathbf1_{\{|z|<1\}}
 +\mathbf1_{\{|z|\ge1\}}\bigr).
\]
For $X_N=(x_1,\ldots,x_N)$ set
$K_N(x)=N^{-1}\sum_jG(x-x_j)$ away from the atoms, and let
$u=G*\rho$, where $\rho\in\mathcal P(\R^d)\cap L^\infty(\R^d)$.  Then
\[
 \int_{\R^d}(|K_N(x)|+|u(x)|)\rho(x)\,\dd x
 \le C_{d,\beta}C_G(1+\|\rho\|_{L^\infty}),
\]
with a constant independent of $N$ and $X_N$.
\end{lemma}

\begin{proof}
Fix a center $a$.  Splitting the convolution at unit distance, the near-field
and far-field contributions are bounded by
\[
 C_G\|\rho\|_\infty|\mathbb S^{d-1}|
 \int_0^1r^{d-\beta-1}\,\dd r
 \qquad\text{and}\qquad C_G\|\rho\|_1,
\]
respectively.  Averaging the resulting bound over the particle positions
controls $K_N$; a further integration against $\rho(a)\,\dd a$ controls $u$.
\end{proof}

\subsection{Characteristic flows and tensorized laws}

We also use the projection inequality \eqref{eq:intro-W2-marginal-control}.
Indeed, if $\rho_N,\widetilde\rho_N\in\mathcal P_2((\R^d)^N)$ are symmetric,
symmetrizing an optimal coupling under simultaneous coordinate permutations does
not change its cost.  Projection onto the first $k$ coordinate pairs then gives
\eqref{eq:intro-W2-marginal-control}.

\begin{proposition}
\label{prop:unified-lagrangian-representation}
Let $b:[0,T]\times\R^d\to\R^d$ be jointly Borel and continuous in
space for almost every time.  Assume that there exist
nonnegative functions $a,m\in L^1(0,T)$ and a continuous, nondecreasing
modulus $\omega$ such that, for almost every $t$,
\[\|b_t\|_{L^\infty}\le a(t),
 \qquad
 |b_t(x)-b_t(y)|\le m(t)\omega(|x-y|),
 \qquad
 \int_{0^+}\frac{\dd r}{\omega(r)}=\infty.
\]
Then the ODE $\dot\Psi^t=b_t\circ\Psi^t$, $\Psi^0=\mathrm{Id}$, generates a unique
global flow on $[0,T]$.  If $\rho\in C([0,T];\mathcal P(\R^d))$ solves
$\partial_t\rho+\nabla\cdot(b\rho)=0$, then
\begin{equation}\label{eq:unified-lagrangian-pushforward}
 \rho_t=(\Psi^t)_\#\rho_0\qquad(0\le t\le T).
\end{equation}
If moreover $b\in L^1(0,T;W^{1,\infty})$, the flow is bi-Lipschitz and, for
$\rho_0\in L^\infty$,
\begin{equation}\label{eq:unified-lagrangian-density-bound}
 \|\rho_t\|_{L^\infty}
 \le \|\rho_0\|_{L^\infty}
 \exp\!\left(\int_0^t\|\nabla\cdot b_\tau\|_{L^\infty}\,\dd\tau\right).
\end{equation}
\end{proposition}

\begin{proof}
This is the standard Osgood-flow consequence of the superposition principle; see
\cite[Theorem~8.2.1]{AmbrosioAGS}.  The Osgood condition gives uniqueness of
integral curves, while $a\in L^1(0,T)$ prevents finite-time escape.  The
superposition principle then identifies every narrowly continuous solution with
the pushforward by the unique flow, which proves
\eqref{eq:unified-lagrangian-pushforward}.  In the Lipschitz case, the usual
forward and backward Gronwall estimates and the Jacobian formula give
\eqref{eq:unified-lagrangian-density-bound}.
\end{proof}

\begin{lemma}
\label{lem:unified-product-initial-data}
Let $W$ be even, let $\rho$ be a probability density, and assume
\[
 E_W(\rho):=\iint W(x-y)\rho(x)\rho(y)\,\dd x\dd y
\]
is absolutely convergent.  Then $\rho^{\otimes N}(\Delta_N)=0$ and
\[
 \mathbb E_{\rho^{\otimes N}}F_N^W(X_N,\rho)
 =-\frac1NE_W(\rho).
\]
If a deterministic correction $\eta_N^W(\rho)$ satisfies
$F_N^W+\eta_N^W(\rho)\ge0$, then
\[0\le\mathbb E_{\rho^{\otimes N}}\bigl(F_N^W+\eta_N^W(\rho)\bigr)
 =\eta_N^W(\rho)-\frac1NE_W(\rho).
\]
\end{lemma}

\begin{proof}
Expanding the off-diagonal functional, applying Fubini, and
counting the $N(N-1)$ ordered pairs gives
\[
 \frac{N(N-1)}{N^2}E_W(\rho)
 -2E_W(\rho)+E_W(\rho)
 =-\frac1NE_W(\rho).
\]
The second identity follows by adding the finite-$N$ error term and
integrating.
\end{proof}

\subsection{Energy dissipation and moment identities}

\begin{proposition}
\label{prop:unified-particle-identities}
Let $X_N\in C^1([0,T);\Omega_N)$ solve the particle equation in
\eqref{eq:unified-particle-continuum}.  Then
\begin{equation}\label{eq:unified-microscopic-dissipation}
 \frac{\dd}{\dd t}H_N^W(X_N(t))
 =-\frac1N\sum_{i=1}^N|K_i^W(X_N(t))|^2.
\end{equation}
Moreover,
\begin{equation}\label{eq:unified-discrete-virial}
 \frac{\dd}{\dd t}m_{2,N}(X_N(t))
 =-\frac1{N^2}\sum_{i\ne j}(x_i(t)-x_j(t))\cdot G(x_i(t)-x_j(t)).
\end{equation}
For sufficiently integrable Lagrangian solutions of the continuum equation,
\begin{equation}\label{eq:unified-continuum-virial}
 \frac{\dd}{\dd t}\int|x|^2\rho_t(x)\,\dd x
 =-\iint (x-y)\cdot G(x-y)\rho_t(x)\rho_t(y)\,\dd x\dd y.
\end{equation}
\end{proposition}

\begin{proof}
Evenness of $W$ gives $\nabla_{x_i}H_N^W=N^{-1}K_i^W$, which yields
\eqref{eq:unified-microscopic-dissipation}.  In addition,
\[
 \frac{\dd}{\dd t}m_{2,N}
 =-\frac2N\sum_i x_i\cdot K_i^W
 =-\frac1{N^2}\sum_{i\ne j}(x_i-x_j)\cdot G(x_i-x_j),
\]
which is \eqref{eq:unified-discrete-virial}.  Testing the continuum equation
with $|x|^2$ and symmetrizing gives \eqref{eq:unified-continuum-virial}.
\end{proof}

Collision exclusion and global continuation are verified in the applications using the energy and virial bounds available for each kernel.

\subsection{Quadratic transportation cost along coupled trajectories}

Let $X_N(t)$ solve the particle system and let
$y_i(t)$ solve $\dot y_i=-u_t(y_i)$.  Set
\[q_i=x_i-y_i,
 \qquad
 Q_N=\frac1N\sum_i|q_i|^2,
\]
and use the quantity $r_N$ from \eqref{eq:unified-force-error} with the
time-dependent interaction field $u_t$.
\begin{proposition}
\label{prop:unified-transport-cost}
Along every trajectory in the coupling,
\[Q_N'
 =-\frac2N\sum_iq_i\cdot(K_i^W-u_t(x_i))
 -\frac2N\sum_iq_i\cdot(u_t(x_i)-u_t(y_i)).
\]
Assume that $\omega_t:[0,\infty)\to[0,\infty)$ is nondecreasing and concave
and that
\[2|x-y|\,|u_t(x)-u_t(y)|
 \le \omega_t(|x-y|^2)
 \qquad\text{for all }x,y.
\]
Then, for almost every time,
\begin{equation}\label{eq:unified-Q-inequality}
 Q_N'\le r_N+Q_N+\omega_t(Q_N).
\end{equation}
In the Lipschitz case this reads
\[
 Q_N'\le r_N+\bigl(1+2\|\nabla u_t\|_\infty\bigr)Q_N.
\]
For a Zygmund or log-Lipschitz field, the sum of the linear term and the
modulus term is absorbed into a concave Osgood modulus of the form
$C_t r(1+\log_+(1/r))$, after a concave extension for large $r$.
\end{proposition}

\begin{proof}
This is the standard Dobrushin coupling computation \cite{Dobrushin1}, written
so as to isolate the quantity $r_N$.  With
$e_i=K_i^W(X_N)-u_t(x_i)$ one has
\[
 \dot q_i=-e_i-\bigl(u_t(x_i)-u_t(y_i)\bigr),
\]
and differentiation gives the stated identity.  Young's inequality and the
assumed modulus yield
\[
 -\frac2N\sum_i q_i\cdot e_i\le Q_N+r_N,
 \qquad
 -\frac2N\sum_iq_i\cdot\bigl(u_t(x_i)-u_t(y_i)\bigr)
 \le \frac1N\sum_i\omega_t(|q_i|^2).
\]
Concavity and Jensen's inequality bound the last average by $\omega_t(Q_N)$,
which proves \eqref{eq:unified-Q-inequality}.  If $u_t$ is Lipschitz, take
$\omega_t(r)=2\|\nabla u_t\|_{L^\infty}r$.
\end{proof}

All subsequent estimates for $Q_N$ follow from
\eqref{eq:unified-Q-inequality} by inserting the corresponding modulus of
continuity of the interaction field; when convenient, the linear term
$Q_N$ is absorbed into the same comparison modulus.

\subsection{Regularization and singular chain rule}

Let $W^\kappa$ be a smooth even regularization of $W$, put
$G^\kappa=\nabla W^\kappa$, and define
\[
 K_i^\kappa=\frac1N\sum_{j\ne i}G^\kappa(x_i-x_j),
 \qquad
 u^\kappa=G^\kappa*\rho.
\]
Define $F_N^\kappa$ by replacing $W$ with $W^\kappa$ in
\eqref{eq:unified-FN}, and set
\[D_N^\kappa
 :=\frac1N\sum_i(K_i^W-u(x_i))\cdot(K_i^\kappa-u^\kappa(x_i)).
\]

\begin{proposition}
\label{prop:chain-rule-admissibility}
Fix $T>0$ and a collision-free particle trajectory.  Suppose that
\[
 \int_0^T\frac1N\sum_{i=1}^N |K_i(X_N(t))|^2\,\dd t<\infty,
 \qquad u\in L^2(0,T;L^\infty(\R^d)).
\]
Then
\[
 r_N(t):=\frac1N\sum_{i=1}^N|K_i(X_N(t))-u_t(x_i(t))|^2\in L^1(0,T).
\]
\end{proposition}
\begin{proof}
The conclusion follows from
\[
 r_N(t)\le \frac2N\sum_i|K_i(X_N(t))|^2+2\|u(t)\|_{L^\infty}^2
\]
after integration over $[0,T]$.
\end{proof}

This estimate is used below in the singular chain rule.

\begin{proposition}[Singular chain rule for the modulated energy]
\label{prop:unified-regularized-chain-rule}
\begin{enumerate}
\item[(i)] \emph{Regularized identity.} Assume that $\rho_t=(\Psi^t)_\#\rho_0$ with
$\partial_t\Psi^t=-u_t\circ\Psi^t$, and that the differentiated Lagrangian
integrals are integrable.  Then
\begin{equation}\label{eq:unified-regularized-chain-rule}
 \frac{\dd}{\dd t}F_N^\kappa
 =-2D_N^\kappa-C_N^{W^\kappa}[u_t].
\end{equation}
\item[(ii)] \emph{Singular limit.} Assume that $r_N\in L^1(0,T)$ is known independently of the modulated energy
identity and that, as $\kappa\downarrow0$,
\begin{equation}\label{eq:unified-chain-rule-limits}
 F_N^\kappa\to F_N^W\ \text{in }C([0,T]),
 \qquad
 D_N^\kappa\to r_N\ \text{in }L^1(0,T),
 \qquad
 C_N^{W^\kappa}[u]\to C_N^W[u]\ \text{in }L^1(0,T).
\end{equation}
Then $F_N^W\in AC([0,T])$ and
\[\frac{\dd}{\dd t}F_N^W=-2r_N-C_N^W[u]
 \qquad\text{for a.e. }t\in(0,T).
\]
\end{enumerate}
\end{proposition}

\begin{proof}
\emph{Part~(i).}
At a fixed smooth scale this is the standard modulated energy differentiation;
compare \cite{Serfaty1,NguyenRosenzweigSerfaty2022}.  In the present
normalization the calculation is as follows.  Differentiating the discrete,
mixed, and continuum terms in $F_N^\kappa$, using
$\dot x_i=-K_i^W$, $\partial_t\rho=\nabla\!\cdot(\rho u)$, and the oddness of
$G^\kappa$, gives
\begin{align*}
 \frac{\dd}{\dd t}F_N^\kappa
 &=-\frac2N\sum_iK_i^W\cdot K_i^\kappa
 +\frac2N\sum_iK_i^W\cdot u^\kappa(x_i)
 -\frac2N\sum_i\int G^\kappa(x_i-y)\cdot u(y)\rho(y)\,\dd y
 -2\int u\cdot u^\kappa\rho.
\end{align*}
Expanding $-2D_N^\kappa-C_N^{W^\kappa}[u]$ gives exactly the same right-hand
side, proving \eqref{eq:unified-regularized-chain-rule}.

\emph{Part~(ii).}  For arbitrary $0\le s<t\le T$, integration of the
regularized identity gives
\[
 F_N^\kappa(t)-F_N^\kappa(s)
 =\int_s^t\bigl(-2D_N^\kappa(r)-C_N^{W^\kappa}[u](r)\bigr)\,\dd r.
\]
The first convergence in \eqref{eq:unified-chain-rule-limits} yields
$
 F_N^\kappa(t)-F_N^\kappa(s)
 \longrightarrow F_N^W(t)-F_N^W(s),
$
while the two $L^1(0,T)$ convergences imply
\begin{align*}
 \int_s^t\left|
 -2D_N^\kappa-C_N^{W^\kappa}[u]
 +2r_N+C_N^W[u]\right|\,\dd r\le
 2\|D_N^\kappa-r_N\|_{L^1(0,T)}
 +\|C_N^{W^\kappa}[u]-C_N^W[u]\|_{L^1(0,T)}
 \longrightarrow0.
\end{align*}
Therefore
\[
 F_N^W(t)-F_N^W(s)
 =\int_s^t\bigl(-2r_N(r)-C_N^W[u](r)\bigr)\,\dd r
 \qquad(0\le s<t\le T).
\]
The integrand belongs to $L^1(0,T)$ by assumption.  Hence
$F_N^W\in AC([0,T])$ and
\[
 \frac{\dd}{\dd t}F_N^W=-2r_N-C_N^W[u]
\]
for almost every $t\in(0,T)$.
\end{proof}

For each kernel, the proof reduces to the three convergences in
\eqref{eq:unified-chain-rule-limits}.  At fixed $N$, collision exclusion permits
the passage to the limit in the discrete terms.  The quantitative estimates below
are uniform over collision-free configurations.

\subsection{Coupled differential inequality}

Combining the estimate for $Q_N$ with the singular modulated energy identity gives the following differential inequality.

\begin{proposition}
\label{prop:coupled-transport-energy-inequalities}
Let $Q_N\ge0$, $F_N+\eta_N\ge0$, and $r_N\ge0$, and set
$\mathcal E_N:=Q_N+F_N+\eta_N$.  Assume, for almost every time,
\[
 Q_N'\le r_N+\omega(Q_N),
 \qquad
 F_N'=-2r_N-C_N^W[u],
\]
where $\omega$ is the nondecreasing modulus introduced above.

\begin{enumerate}
\item[(i)] If, for some $0\le\theta<1$ and nonnegative functions $b,e$,
\[|C_N^W[u]|\le\theta r_N+b(t)(F_N+\eta_N)+e(t),
\]
then
\[\mathcal E_N'+(1-\theta)r_N
 \le \omega(\mathcal E_N)+b(t)\mathcal E_N+e(t).
\]

\item[(ii)] Let $u=v+w$ and suppose that
\begin{align*}
 |C_N^W[v]|&\le b_v(t)(F_N+\eta_N),\\
 \|w\|_{L^2(\mu_N)}^2
 +2\left|\int w\cdot(K_N^W-u)\rho\right|&\le e_w(t).
\end{align*}
Then
\begin{equation}\label{eq:transport-energy-comparison-estimate}
 \mathcal E_N'+r_N[v]
 \le \omega(\mathcal E_N)+b_v(t)\mathcal E_N+e_w(t).
\end{equation}
In addition,
\[r_N\le2r_N[v]+2\|w\|_{L^2(\mu_N)}^2.\]
\end{enumerate}
\end{proposition}

\begin{proof}
For part~\textup{(i)} we have
\begin{align*}
 \mathcal E_N'
 \le-r_N+\omega(Q_N)-C_N^W[u]\le-(1-\theta)r_N+\omega(\mathcal E_N)+b(t)\mathcal E_N+e(t).
\end{align*}
For part~\textup{(ii)}, \eqref{eq:completion-of-squares} gives
\begin{align*}
 \mathcal E_N'
\le \omega(Q_N)-r_N-C_N^W[v]-C_N^W[w]\le \omega(\mathcal E_N)-r_N[v]+b_v(t)\mathcal E_N+e_w(t).
\end{align*}
The estimate for $r_N$ is \eqref{eq:force-error-comparison}.
\end{proof}

\section{Averaging and comparison principles}
\label{sec:particle-law-comparison}
We first record an abstract comparison principle at a fixed scale and its
averaged form.  This framework is used directly in the Riesz extension.  The Coulomb
proof uses the logarithmic comparison lemmas below together with the Coulomb
estimate at a density-dependent mollification scale.  In both applications the scale is kept
fixed before averaging and is selected only from the resulting averaged
quantity.  The required measurability and integrability are established in
Appendix~\ref{app:measurability}.

\subsection{Abstract comparison principle at a fixed scale}
\label{sec:averaged-law-estimate}

In this subsection, $Q$ denotes the quadratic transport cost associated with the
coupling, $F+\eta$ the modulated energy with its finite-$N$ correction, $R$ the
mean-square force error, and $R_\eps$
the analogous quantity with the mollified limiting field.  The functions
$L(\eps)$, $E_1(\eps)$, and $E_2(\eps)$ quantify, respectively,
the Lipschitz norm of the mollified field, the error in the commutator estimate,
and the error in comparing $R$ with $R_\eps$.  In the applications below, the
correspondence with the particle notation is
\[
 Q\leftrightarrow Q_N,\qquad F\leftrightarrow F_N,\qquad
 R\leftrightarrow r_N,\qquad R_\eps\leftrightarrow r_N^\eps,\qquad
 \mathcal C\leftrightarrow C_N[u].
\]

\subsubsection{A comparison estimate}

Fix $z_*>0$ and a countable set of mollification scales
$\mathcal S\subset(0,1]$.  Let $\omega:[0,\infty)\to[0,\infty)$ be
continuous, nondecreasing, and concave, with $\omega(0)=0$.  Let
$\omega_0:[0,z_*]\to[0,\infty)$ be continuous, nondecreasing, and concave,
with $\omega_0(0)=0$, and assume that $\omega_0$ is positive on $(0,z_*]$ and
satisfies the Osgood condition
\[
 \int_{0^+}\frac{\dd r}{\omega_0(r)}=\infty.
\]
Let $L,E_1,E_2:\mathcal S\to[0,\infty)$.  For each
$z\in(0,z_*]$, choose a scale $\eps_z\in\mathcal S$.  Assume that
$z\mapsto\eps_z$ is Borel and that, for a constant $C_0\ge1$ and every
$0<z\le z_*$,
\begin{equation}\label{eq:abstract-optimization-statement}
 \omega(z)+L(\eps_z)z+E_1(\eps_z)+E_2(\eps_z)
 \le C_0\omega_0(z).
\end{equation}
We also assume that the choice of scale is locally constant from the right:
\begin{equation}\label{eq:abstract-selector-right-stability}
 \forall z\in(0,z_*]\ \exists \eta_z>0:\qquad
 \eps_y=\eps_z\quad
 \text{for }z\le y\le\min\{z+\eta_z,z_*\}.
\end{equation}
Define
\[
 \Xi(r):=\int_r^{z_*}\frac{\dd s}{\omega_0(s)},\qquad 0<r\le z_*.
\]
Then $\Xi$ is continuous and strictly decreasing from $(0,z_*]$ onto
$[0,\infty)$.  Given $A\in(0,z_*]$ and a nonnegative
$\lambda\in L^1(0,T)$ such that
\begin{equation}\label{eq:abstract-majorant-admissibility}
 C_0\int_0^T \lambda(t)\,\dd t<\Xi(A),
\end{equation}
set
\[
 B_A(t):=\Xi^{-1}\!\left(\Xi(A)-C_0\int_0^t \lambda(r)\,\dd r\right).
\]
Then $B_A'=C_0\lambda\,\omega_0(B_A)$ and $B_A(0)=A$.

\begin{theorem}[Comparison principle]
\label{thm:optimized-transport-energy-principle}
Let $Q,F\in AC([0,T])$, let $R,\mathcal C$ be measurable, and for each
$\eps\in\mathcal S$ let $R_\eps\ge0$ be measurable.  Let
$\eta,\zeta\ge0$ be constant in time and define
\[
 \mathcal E:=Q+F+\eta+\zeta.
\]
Assume
\[
 Q\ge0,\qquad F+\eta\ge0,\qquad R\ge0.
\]
Suppose that there exist a nonnegative function $\lambda\in L^1(0,T)$ and a constant
$c_R\ge1$ such that, for almost every time,
\begin{align*}
 Q'\le R+\lambda(t)\omega(Q),\qquad
 F'\le-2R-\mathcal C,
\end{align*}
and, for every $\eps\in\mathcal S$, outside an exceptional set independent
of $\eps$,
\begin{align*}
 -R-\mathcal C
 \le-R_\eps+\lambda(t)\bigl(L(\eps)(F+\eta)+E_1(\eps)\bigr),\qquad
 R\le c_RR_\eps+\lambda(t)E_2(\eps).
\end{align*}
Assume \eqref{eq:abstract-optimization-statement}.  Put $A:=\mathcal E(0)$.
If $0<A\le z_*$ and \eqref{eq:abstract-majorant-admissibility} holds, define
\begin{equation}\label{eq:general-optimized-scale}
 \eps_*(t):=\eps_{B_A(t)}.
\end{equation}
Then
\begin{equation}\label{eq:general-optimized-primary-bound}
 \mathcal E(t)+\int_0^tR_{\eps_*(r)}(r)\,\dd r\le B_A(t),
 \qquad 0\le t\le T,
\end{equation}
and
\[
 \int_0^tR(r)\,\dd r
 \le c_RB_A(t)+B_A(t)-A.
\]
In particular,
\begin{equation}\label{eq:general-optimized-total-bound}
 \sup_{0\le t\le T}\mathcal E(t)
 +\int_0^TR_{\eps_*(t)}(t)\,\dd t
 +\int_0^TR(t)\,\dd t
 \le (c_R+3)B_A(T).
\end{equation}
The scale \eqref{eq:general-optimized-scale} depends only on $B_A$ and not on the particle configuration.
\end{theorem}

\begin{proof}
Choose $\nu_0>0$ so small that $A+\nu_0\le z_*$ and
$C_0\int_0^T \lambda<\Xi(A+\nu_0)$.  For $0<\nu\le\nu_0$, let $B_{A+\nu}$ be
the corresponding function and let
\[
 \tau_\nu:=\inf\{t\in[0,T]:\mathcal E(t)\ge B_{A+\nu}(t)\},
\]
with the convention $\inf\varnothing=T$.  On $[0,\tau_\nu]$, take
$\eps_\nu(t)=\eps_{B_{A+\nu}(t)}$.  Since
$Q,F+\eta\le\mathcal E\le B_{A+\nu}$ and $\omega$ is nondecreasing,
\begin{align*}
 \mathcal E'+R_{\eps_\nu}
 &\le \lambda\bigl(\omega(Q)+L(\eps_\nu)(F+\eta)+E_1(\eps_\nu)\bigr)\\
 &\le C_0\lambda\,\omega_0(B_{A+\nu})=B_{A+\nu}'.
\end{align*}
Integrating from $0$ to $t$ gives
\[
 \mathcal E(t)+\int_0^tR_{\eps_\nu(r)}(r)\,\dd r
 \le B_{A+\nu}(t)-\nu
\]
for $t\le\tau_\nu$.  If $\tau_\nu<T$, the definition of $\tau_\nu$ gives
$\mathcal E(\tau_\nu)=B_{A+\nu}(\tau_\nu)$, which contradicts the preceding
strict inequality.  Hence $\tau_\nu=T$.

To pass to the limit $\nu\downarrow0$, note that the continuity of $\Xi^{-1}$ gives, for every $r\in[0,T]$,
\[
 B_{A+\nu}(r)\downarrow B_A(r)\qquad(\nu\downarrow0).
\]
Fix any sequence $\nu_n\downarrow0$.  By
\eqref{eq:abstract-selector-right-stability}, for every fixed $r$ the
discrete value $\eps_{B_{A+\nu_n}(r)}$ is eventually equal to
$\eps_{B_A(r)}=\eps_*(r)$.  Consequently,
$R_{\eps_{\nu_n}(r)}(r)\to R_{\eps_*(r)}(r)$ pointwise in $r$.  The
nonnegativity of every $R_\eps$ and Fatou's lemma yield, for each
$t\in[0,T]$,
\begin{align*}
 \mathcal E(t)+\int_0^tR_{\eps_*(r)}(r)\,\dd r
 \le \mathcal E(t)+\liminf_{n\to\infty}
       \int_0^tR_{\eps_{\nu_n}(r)}(r)\,\dd r\le \lim_{n\to\infty}
       \bigl(B_{A+\nu_n}(t)-\nu_n\bigr)=B_A(t).
\end{align*}
Since the sequence was arbitrary, this proves
\eqref{eq:general-optimized-primary-bound}.  The passage to the limit uses the local constancy from the right of the chosen dyadic scale.

The estimate for $R$ and \eqref{eq:abstract-optimization-statement} give
\begin{align*}
 \int_0^tR
 \le c_R\int_0^tR_{\eps_*}
 +\int_0^t\lambda(r)E_2(\eps_*(r))\,\dd r\le c_RB_A(t)+C_0\int_0^t\lambda(r)\omega_0(B_A(r))\,\dd r=c_RB_A(t)+B_A(t)-A.
\end{align*}
The last two estimates imply \eqref{eq:general-optimized-total-bound}.
\end{proof}

\begin{proposition}
\label{prop:abstract-fixed-scale-comparison}
Assume the positivity and differential hypotheses of
Theorem~\ref{thm:optimized-transport-energy-principle}, but not its smallness
condition.  Suppose
\[
 \omega(z)\le C_\omega(1+z),\qquad z\ge0,
\]
and that there is a fixed $\bar\eps\in\mathcal S$ for which
$L(\bar\eps)+E_1(\bar\eps)+E_2(\bar\eps)<\infty$.  Then
\[
 \sup_{0\le t\le T}\mathcal E(t)
 +\int_0^TR_{\bar\eps}(t)\,\dd t
 +\int_0^TR(t)\,\dd t
 \le C_T\bigl(1+\mathcal E(0)\bigr),
\]
where $C_T$ depends only on $c_R,C_\omega,\bar\eps$, the values of
$L,E_1,E_2$ at $\bar\eps$, and $\|\lambda\|_{L^1(0,T)}$.
\end{proposition}

\begin{proof}
At the fixed scale $\bar\eps$, the hypotheses give
\[
 \mathcal E'+R_{\bar\eps}
 \le a(t)\mathcal E+b(t),
\]
where $a,b\in L^1(0,T)$ depend only on the quantities listed in the
statement.  Gronwall yields
$\sup_{t\le T}\mathcal E(t)\le C_T(1+\mathcal E(0))$.  Integrating the same
inequality and using $\mathcal E(T)\ge0$ gives
$\int_0^TR_{\bar\eps}\le C_T(1+\mathcal E(0))$.  The estimate
$R\le c_RR_{\bar\eps}+\lambda E_2(\bar\eps)$ then gives the same bound
for $\int_0^TR$, proving the proposition.
\end{proof}

For initial errors outside $[0,z_*]$, we first apply Proposition~\ref{prop:abstract-fixed-scale-comparison}.  Combining the two estimates gives the expressions $A^\gamma+A$ and their finite-$q$ analogues in the applications.

\begin{corollary}
\label{cor:averaged-optimized-comparison}
Let $(\mathcal X_0,\mathcal F,\mathbb P)$ be a probability space.  Assume that
the hypotheses of Theorem~\ref{thm:optimized-transport-energy-principle}
hold for almost every trajectory, with the same $\lambda$, $c_R$, and scale
functions.  Assume also that, for every fixed scale parameter in the countable family, the relevant quantities
and their absolutely continuous derivatives are jointly measurable and
integrable, and that the countable family of scales is valid outside a single
null set.  Define
\[
 \overline{\mathcal E}(t):=\mathbb E[Q(t)+F(t)+\eta]+\zeta,
 \qquad
 \overline R(t):=\mathbb ER(t),
 \qquad
 \overline R_\eps(t):=\mathbb ER_\eps(t),
\]
and put $A:=\overline{\mathcal E}(0)$.  If $0<A\le z_*$ and
\eqref{eq:abstract-majorant-admissibility} holds, choose the deterministic
scale
\begin{equation}\label{eq:averaged-general-deterministic-scale}
 \eps_A(t):=\eps_{B_A(t)}.
\end{equation}
Then
\[
 \sup_{0\le t\le T}\overline{\mathcal E}(t)
 +\int_0^T\overline R_{\eps_A(t)}(t)\,\dd t
 +\int_0^T\overline R(t)\,\dd t
 \le(c_R+3)B_A(T).
\]
The scale depends only on the averaged initial error and deterministic bounds for the limiting solution.  If the hypotheses of
Proposition~\ref{prop:abstract-fixed-scale-comparison} hold almost surely
with deterministic data, the estimate at a fixed scale also passes to the average
and covers arbitrary initial error.
\end{corollary}

\begin{proof}
The integrability assumptions allow differentiation under the fixed expectation
and averaging of every inequality at a fixed scale.  Jensen's inequality gives
\[
 \mathbb E\omega(Q)\le\omega(\mathbb E Q)\le\omega(\overline{\mathcal E}),
 \qquad \mathbb E(F+\eta)\le\overline{\mathcal E}.
\]
Hence the averaged quantities satisfy the deterministic hypotheses of
Theorem~\ref{thm:optimized-transport-energy-principle}.  The common null set
permits the deterministic scale \eqref{eq:averaged-general-deterministic-scale}
to be inserted after averaging.  The optimized estimate follows from that
theorem, and the estimate at a fixed scale follows from
Proposition~\ref{prop:abstract-fixed-scale-comparison}.
\end{proof}

In the Coulomb application and the Riesz application under the stated Besov regularity, the modulus $\omega_0$ is chosen from the logarithmic moduli introduced below.
The remaining input depending on the kernel is the choice of a dyadic scale satisfying \eqref{eq:abstract-optimization-statement}.

\subsection{Logarithmic comparison lemmas}\label{subsec:power-log-specializations}
\begin{equation}\label{eq:ell-definition}
 \ell(r):=1+\log\frac1r,
 \qquad 0<r\le1.
\end{equation}
\begin{lemma}
\label{lem:concave-log-modulus}
Let $0\le\beta\le1$.  Define $\omega_\beta(0)=0$ and
\[
 \omega_\beta(r)=r\left(1+\log\frac1r\right)^\beta,
 \qquad 0<r\le e^{-1},
\]
and extend it for $r>e^{-1}$ by the tangent line at $e^{-1}$.  Then
$\omega_\beta$ is continuous, nondecreasing, and concave.  Moreover,
$r(1+\log_+(1/r))^\beta\le C_\beta\omega_\beta(r)$ for all $r\ge0$, and
\[\int_0^{e^{-1}}\frac{\dd r}{\omega_\beta(r)}=\infty
 \qquad\text{for every }0\le\beta\le1.
\]
Every finite-$q$ logarithmic modulus used below is an Osgood modulus;
$\beta=0$ is the linear/Gronwall case and $\beta=1$ is the limiting
logarithmic case.  Set $\omega_*:=\omega_1$.
\end{lemma}

\begin{proof}
Direct differentiation on $(0,e^{-1})$ gives monotonicity and concavity, and
the tangent continuation preserves both.  With
$y=1+\log(1/r)$,
\[
 \int_0^{e^{-1}}\frac{\dd r}{\omega_\beta(r)}
 =\int_2^\infty y^{-\beta}\,\dd y,
\]
which diverges for every $0\le\beta\le1$.
\end{proof}

\begin{lemma}
\label{lem:finite-q-bihari-comparison}
Let $1\le q<\infty$, let $\beta_q=1-1/q$, and let
$\lambda\in L^1(0,T)$ be nonnegative.  Put
$\Lambda(t)=\int_0^t \lambda(r)\,\dd r$.  Suppose that $Y:[0,T]\to[0,\infty)$ is
continuous and, up to the first exit from
$\overline Y(t):=\sup_{0\le r\le t}Y(r)<e^{-3}$,
\[\overline Y(t)\le A+\int_0^t
 \lambda(r)\omega_{\beta_q}(\overline Y(r))\,\dd r.
\]
There is $a_*\in(0,e^{-3})$, depending only on $q$ and $\Lambda(T)$, such that, if
$0<A\le a_*$, then no exit occurs and
\begin{equation}\label{eq:finite-q-explicit-majorant}
 \overline Y(t)\le
 \exp\!\left(
 1-\left[
 \left(1+\log\frac1A\right)^{1/q}
 -\frac{\Lambda(t)}q
 \right]^q\right),
 \qquad 0\le t\le T.
\end{equation}
For fixed $q<\infty$ and $\Lambda\ge0$, the right-hand side also satisfies,
for all sufficiently small $A$,
\begin{equation}\label{eq:finite-q-majorant-asymptotic}
 \exp\!\left(
 1-\left[
 \left(1+\log\frac1A\right)^{1/q}-\frac \Lambda q
 \right]^q\right)
 \le C_{q,\Lambda}A
 \exp\!\left(C_{q,\Lambda}(1+\log(1/A))^{1-1/q}\right).
\end{equation}
For $q=1$, the right-hand side is $e^{\Lambda}A$.
\end{lemma}

\begin{proof}
Let
\[
V(t)=A+\int_0^t\lambda(r)\omega_{\beta_q}(\overline Y(r))\,\dd r,
\qquad H(t)=1+\log\frac1{V(t)}.
\]
The hypothesis for $\overline Y$ is only assumed before its first exit from
the small regime, so we stop both quantities at that regime.  Set
\[
 \tau:=\inf\left\{t\in[0,T]:
 \max\{\overline Y(t),V(t)\}\ge e^{-3}\right\},
\]
with $\tau=T$ if the set is empty.  For $0\le t<\tau$, the assumed integral
inequality is valid and gives $\overline Y(t)\le V(t)<e^{-3}<e^{-1}$.
Hence, by monotonicity of $\omega_{\beta_q}$ and the identity
$\omega_{\beta_q}(V)=VH^{1-1/q}$ in this range,
\[
V'(t)\le \lambda(t)V(t)H(t)^{1-1/q}
\qquad\text{for a.e. }t\in(0,\tau).
\]
Since $H'(t)=-V'(t)/V(t)$, it follows that
\[
\frac{\dd}{\dd t}H(t)^{1/q}
=\frac1q H(t)^{1/q-1}H'(t)
\ge-\frac1q\lambda(t).
\]
Therefore, on $[0,\tau)$,
\[
H(t)^{1/q}\ge H(0)^{1/q}-\frac{\Lambda(t)}q,
\]
and thus
\[
 V(t)\le
 \exp\!\left(
 1-\left[
 \left(1+\log\frac1A\right)^{1/q}
 -\frac{\Lambda(t)}q
 \right]^q\right).
\]
Choose $a_*$ so small that the bracket stays positive and the
right-hand side is strictly smaller than $e^{-3}$ for every $t\le T$.
If $\tau<T$, continuity and $\overline Y\le V$ on $[0,\tau)$ imply
$V(\tau)\ge e^{-3}$, contradicting the preceding strict bound.  Hence
$\tau=T$, and \eqref{eq:finite-q-explicit-majorant} holds on the full
interval.

For the asymptotic estimate, put
\[
 X=\left(1+\log\frac1A\right)^{1/q},\qquad c=\frac\Lambda q.
\]
For $A$ sufficiently small one has $X>2c$, and the mean value theorem gives
\[
 0\le X^q-(X-c)^q\le q c X^{q-1}
 \le C_{q,\Lambda}X^{q-1}.
\]
Since $e^{1-X^q}=A$, the explicit majorant equals
\[
 A\exp\!\bigl(X^q-(X-c)^q\bigr)
 \le C_{q,\Lambda}A
 \exp\!\left(C_{q,\Lambda}
 (1+\log(1/A))^{1-1/q}\right),
\]
which is \eqref{eq:finite-q-majorant-asymptotic}.  When $q=1$ the
difference is exactly $\Lambda$, and the majorant is $e^{\Lambda}A$.
\end{proof}

\begin{lemma}
\label{lem:logarithmic-osgood-comparison}
Let $A\in(0,e^{-3})$, let $\lambda\in L^1(0,T)$ be nonnegative, and let
$Y:[0,T]\to[0,\infty)$ be continuous.  Write
$\overline Y(t)=\sup_{0\le r\le t}Y(r)$ and
\[\Lambda(t):=\int_0^t \lambda(r)\,\dd r.
\]
Suppose that, up to the first exit from $\overline Y<e^{-3}$,
\[\overline Y(t)\le A+\int_0^t \lambda(r)\omega_*(\overline Y(r))\,\dd r.
\]
There exists $a_T\in(0,e^{-3})$, depending only on $\Lambda(T)$, such that, if
$A\le a_T$, no exit occurs on $[0,T]$ and
\begin{equation}\label{eq:critical-riesz-explicit-osgood-bound}
 \overline Y(t)\le
 \exp\!\bigl(1-e^{-\Lambda(t)}\bigr)
 A^{e^{-\Lambda(t)}},
 \qquad 0\le t\le T.
\end{equation}
\end{lemma}

\begin{proof}
The argument is the limiting case $\beta=1$ of the preceding Bihari computation.
With
\[
 V(t)=A+\int_0^t\lambda(r)\omega_*(\overline Y(r))\,\dd r,
 \qquad H(t)=1+\log\frac1{V(t)},
\]
and the same stopping argument, one has while $V<e^{-3}$ that
$\omega_*(V)=VH$ and hence $H'\ge-\lambda H$.  Therefore
$H(t)\ge H(0)e^{-\Lambda(t)}$, which is equivalent to
\[
 V(t)\le \exp\!\bigl(1-e^{-\Lambda(t)}\bigr)
 A^{e^{-\Lambda(t)}}.
\]
Choosing $a_T$ so that this majorant stays below $e^{-3}$ prevents exit and
proves the claim.
\end{proof}

\begin{lemma}[Logarithmic Osgood comparison with a linear drift]
\label{lem:logarithmic-osgood-linear-drift}
Let $a,b\in L^1(0,T)$ be nonnegative, let $Y\in AC([0,T])$ be nonnegative,
and let $R\ge0$ be measurable.  Put
\[
 A(t):=\int_0^t a(s)\,\dd s,\qquad
 B(t):=\int_0^t b(s)\,\dd s.
\]
Assume that, up to the first exit from $Y<e^{-3}$,
\begin{equation}\label{eq:osgood-linear-drift-hypothesis}
 Y'(t)+R(t)\le a(t)\omega_*(Y(t))+b(t)Y(t)
 \qquad\text{for a.e. }t.
\end{equation}
There exists $a_*\in(0,e^{-3})$, depending only on $A(T)$ and
$B(T)$, such that if $0<Y(0)\le a_*$, then no exit occurs and
\begin{equation}\label{eq:osgood-linear-drift-majorant}
 Y(t)+\int_0^tR(s)\,\dd s
 \le V(t)
 \le \exp\!\bigl(B(t)+1-e^{-A(t)}\bigr)
 Y(0)^{e^{-A(t)}},\qquad 0\le t\le T,
\end{equation}
where $V$ is the scalar comparison solution
\[
 V'=a(t)\omega_*(V)+b(t)V,\qquad V(0)=Y(0).
\]
In particular, the linear coefficient $b$ enters only through the
multiplicative factor $e^{B(t)}$; the power of the initial error is determined
solely by $A(t)=\int_0^t a$.
\end{lemma}

\begin{proof}
While $V<e^{-3}$, set $Z(t)=e^{-B(t)}V(t)$.  Since $0<Z\le V<e^{-3}$ and
$\omega_*(r)=r(1+\log(1/r))$ in this range,
\begin{align*}
 Z'
 &=a(t)e^{-B(t)}\omega_*(e^{B(t)}Z)\\
 &=a(t)Z\left(1+\log\frac1Z-B(t)\right)
 \le a(t)\omega_*(Z).
\end{align*}
The logarithmic Osgood calculation of
Lemma~\ref{lem:logarithmic-osgood-comparison} therefore yields
\[
 Z(t)\le \exp\!\bigl(1-e^{-A(t)}\bigr)Z(0)^{e^{-A(t)}}.
\]
Multiplying by $e^{B(t)}$ gives the second inequality in
\eqref{eq:osgood-linear-drift-majorant}.  Choosing
$a_*$ so that the right-hand side stays strictly below $e^{-3}$
prevents exit for $V$.

The map $y\mapsto a(t)\omega_*(y)+b(t)y$ is continuous and
nondecreasing.  Since $V(0)=Y(0)>0$ and $V$ is nondecreasing, the scalar
equation is locally Lipschitz along the positive range visited by $V$; the
standard first-contact comparison applied to
\eqref{eq:osgood-linear-drift-hypothesis} therefore gives $Y\le V$ before
exit and hence on the full interval by the preceding bound.  Integrating
\eqref{eq:osgood-linear-drift-hypothesis} and using
monotonicity once more,
\[
 Y(t)+\int_0^tR
 \le Y(0)+\int_0^t\bigl[a\omega_*(Y)+bY\bigr]
 \le Y(0)+\int_0^t\bigl[a\omega_*(V)+bV\bigr]=V(t),
\]
which proves the first inequality and completes the proof.
\end{proof}

The following lemma provides the measurable dyadic choice of mollification scale used in the Coulomb estimate and in the Riesz estimate under the stated Besov regularity.

\begin{lemma}
\label{lem:measurable-dyadic-scale}
Let $0<\sigma\le1$.  For $z\in(0,e^{-3}]$ and $j\ge4$, set
$\eps_j=2^{-j}$ and define
\[j_\sigma(z):=\min\{j\ge4:\ \eps_j^\sigma\le z\},
 \qquad
 \eps_\sigma(z):=\eps_{j_\sigma(z)}.
\]
Then $j_\sigma$ and $\eps_\sigma$ are Borel measurable,
$\eps_\sigma$ is locally constant from the right as in
\eqref{eq:abstract-selector-right-stability}, and
\begin{equation}\label{eq:dyadic-running-comparability}
 \eps_\sigma(z)^\sigma\le z
 <2^\sigma\eps_\sigma(z)^\sigma,
 \qquad
 1+\log\frac1{\eps_\sigma(z)}
 \le C_\sigma\left(1+\log\frac1z\right).
\end{equation}
In particular,
\[\eps_\sigma(z)^\sigma
 \left(1+\log\frac1{\eps_\sigma(z)}\right)
 \le C_\sigma\omega_*(z),
 \qquad
 \eps_\sigma(z)^2\le z.
\]
If $Y:[0,T]\to(0,e^{-3}]$ is continuous and
$\overline Y(t)=\sup_{0\le r\le t}Y(r)$, then
$\eps_*(t):=\eps_\sigma(\overline Y(t))$ is Borel measurable and satisfies
the same estimates with $z=\overline Y(t)$.  If a pointwise finite-$N$ estimate
holds for every dyadic $\eps_j$ outside a possibly $j$-dependent null set in
time, then it holds at $\eps_*(t)$ outside one single null set.
\end{lemma}

\begin{proof}
Fix $z\in(0,e^{-3}]$ and write $j=j_\sigma(z)$.  By definition of the
minimum,
\[
 2^{-j\sigma}\le z,
\]
and, whenever $j>4$, the preceding dyadic scale does not satisfy the defining
inequality, so
\[
 z<2^{-(j-1)\sigma}=2^\sigma2^{-j\sigma}.
\]
Since $0<\sigma\le1$,
$2^{-4\sigma}\ge2^{-4}>e^{-3}\ge z$, so in fact $j\ge5$.  Since
$\eps_\sigma(z)=2^{-j}$, the two inequalities give
\[
 \eps_\sigma(z)^\sigma\le z
 <2^\sigma\eps_\sigma(z)^\sigma.
\]
The upper inequality implies
\[
 \eps_\sigma(z)^\sigma>2^{-\sigma}z,
\]
and hence
\begin{align*}
 \log\frac1{\eps_\sigma(z)}
 &<\frac1\sigma\log\frac{2^\sigma}{z}
 =\log2+\frac1\sigma\log\frac1z.
\end{align*}
Therefore
\[
 \ell(\eps_\sigma(z))
 =1+\log\frac1{\eps_\sigma(z)}
 \le 1+\log2+\frac1\sigma\log\frac1z
 \le C_\sigma\ell(z),
\]
which proves \eqref{eq:dyadic-running-comparability}.  Since
$0<\eps_\sigma(z)<1$ and $0<\sigma\le1<2$,
\[
 \eps_\sigma(z)^2\le\eps_\sigma(z)^\sigma\le z.
\]
Moreover,
\begin{align*}
 \eps_\sigma(z)^\sigma\ell(\eps_\sigma(z))
 &\le z\,C_\sigma\ell(z)
 =C_\sigma\omega_*(z),
\end{align*}
because $z\le e^{-3}<e^{-1}$ and hence
$\omega_*(z)=z\ell(z)$.

The level sets of $j_\sigma$ are explicit.  For $j\ge5$,
\[
 \{z:j_\sigma(z)=j\}
 =[2^{-j\sigma},2^{-(j-1)\sigma}),
\]
intersected with $(0,e^{-3}]$, while the level set for $j=4$ is the
corresponding terminal interval.  Hence $j_\sigma$ is Borel measurable, and
so is $\eps_\sigma=2^{-j_\sigma}$.  If $z$ belongs to one of these half-open
intervals, there exists $\eta_z>0$ such that
\[
 j_\sigma(y)=j_\sigma(z),
 \qquad
 z\le y\le\min\{z+\eta_z,e^{-3}\},
\]
which is exactly \eqref{eq:abstract-selector-right-stability}.

If $Y$ is continuous, then
\[
 \overline Y(t):=\max_{0\le r\le t}Y(r)
\]
is continuous and nondecreasing.  Consequently
\[
 \eps_*(t)=\eps_\sigma(\overline Y(t))
\]
is Borel measurable, and the preceding estimates applied with
$z=\overline Y(t)$ give all the asserted bounds.

For the last statement, let $E_j\subset[0,T]$ be a set of full measure on which
the estimate at the scale $\eps_j$ holds.  Since the dyadic family is countable,
\[
 E:=\bigcap_{j\ge4}E_j
\]
has full measure.  For every $t\in E$, the estimate is valid simultaneously
for all dyadic scales and therefore in particular for
$\eps_*(t)\in\{\eps_j:j\ge4\}$.
\end{proof}

In the Coulomb case and the Riesz case under the stated Besov regularity, we first average the inequalities for each fixed dyadic parameter and then choose a dyadic parameter from the resulting averaged quantity.  In the Coulomb case the associated mollification radius is allowed to vary deterministically with time.
The measurability and integrability needed for this argument are proved in
Appendix~\ref{app:measurability}.

The logarithmic endpoint used below is obtained directly from
Corollary~\ref{cor:averaged-optimized-comparison} by taking
\[
 \omega=\omega_0=\omega_*,\qquad
 L(\eps)=\ell(\eps),\qquad
 E_1(\eps)=\eps^\sigma\ell(\eps),\qquad
 E_2(\eps)=\eps^2,
\]
and using Lemmas~\ref{lem:measurable-dyadic-scale} and~\ref{lem:logarithmic-osgood-comparison}.  We use this specialization directly
in the Riesz application under the stated Besov regularity and in the Coulomb application.

\subsection{From a transported coupling to the \texorpdfstring{$N$}{N}-particle law}
\label{subsec:transfer-N-particle}
Let $\pi_N^t$ be the transported coupling of the particle law and
$\rho_t^{\otimes N}$.  Any estimate controlling the expected quadratic
transport cost gives
\[
 \frac1N W_2^2(\rho_N(t),\rho_t^{\otimes N})
 \le \int Q_N(t)\,\dd\pi_N^t.
\]
In the Coulomb argument this cost is first controlled through its weighted quantity
$P_N=\widehat mQ_N$ and is recovered at the end of the proof.  The
quantities $F_N(X_N,\rho_t)$ and $r_N(X_N,\rho_t)$ depend only on the particle
configuration and therefore integrate against $\rho_N(t)$.  Marginals of fixed order follow from the symmetric projection inequality
\eqref{eq:intro-W2-marginal-control}.

\begin{proof}[Proof of Corollary~\ref{cor:product-data-and-marginals}]
The projection of a symmetric optimal coupling onto its first $k$
coordinates gives \eqref{eq:intro-W2-marginal-control}.  The required expectation identity is Lemma~\ref{lem:unified-product-initial-data}.  Substitution into the two principal results gives the rates cited in the statement.
\end{proof}

\section{The limiting Coulomb equation}\label{sec:coulomb-reference-solutions}
We collect the bounded-density Coulomb theory used in the main application.
Existence, bounded-solution uniqueness, and the sharp $L^\infty$ decay are
recalled from the literature, and the field and flow estimates required for the
comparison argument are established below.

\subsection{Global solutions with bounded density}\label{subsec:weak-solutions}
We begin with the field estimates and the solution theory for bounded densities used
in the Coulomb application for all $d\ge2$.  We write $\mathcal P(\R^d)$ and
$\mathcal P_2(\R^d)$ for probability measures and probability measures with
finite second moment, respectively.  For the normalized Coulomb kernel $g_d$ in
\eqref{eq:unified-coulomb-kernel}, set
\[
 E(\mu):=\iint g(x-y)\mu(x)\mu(y)\,\dd x\dd y
\]
whenever the integral is well defined.

A standard decomposition into the regions $|x-y|<1$ and $|x-y|\ge1$ gives
\begin{align}
 \iint |g(x-y)|\mu(x)\mu(y)\,\dd x\dd y
 &\le C\left(1+\|\mu\|_\infty+\int|x|^2\mu(x)\,\dd x\right),
 \label{eq:auto-absolute-log-energy}\\
 0\le E(\mu)&\le C_d(1+\|\mu\|_\infty),\qquad d\ge3.\notag
\end{align}
These bounds are uniform on families for which the displayed norms are
uniformly bounded.

\begin{proposition}[Coulomb field estimates]
\label{prop:unified-coulomb-field}
Let $d\ge2$, let $g=g_d$ be given by
\eqref{eq:unified-coulomb-kernel}, and let
$\rho\in L^1(\R^d)\cap L^\infty(\R^d)$.  Then
$u=\nabla g*\rho$ has a bounded continuous Zygmund representative and
\begin{equation}\label{eq:coulomb-unified-field-bound}
 \|u\|_{L^\infty}+[u]_{\Lambda_*}
 \le C_d\bigl(\|\rho\|_{L^1}+\|\rho\|_{L^\infty}\bigr).
\end{equation}
In particular,
\begin{equation}\label{eq:coulomb-unified-loglip}
 |u(x)-u(y)|
 \le C_d\bigl(\|\rho\|_{L^1}+\|\rho\|_{L^\infty}\bigr)|x-y|
 \left(1+\log_+\frac1{|x-y|}\right).
\end{equation}
Let $\chi\in C_c^\infty(B_1)$ be nonnegative, even, and of unit mass, and
set $\chi_\eps(z)=\eps^{-d}\chi(z/\eps)$.  For $0<\eps\le e^{-2}$,
\begin{equation}\label{eq:coulomb-unified-mollifier-estimates}
 \|u-u*\chi_\eps\|_{L^\infty}
 \le C_d\bigl(\|\rho\|_{L^1}+\|\rho\|_{L^\infty}\bigr)\eps,
 \qquad
 \|\nabla(u*\chi_\eps)\|_{L^\infty}
 \le C_d\bigl(\|\rho\|_{L^1}+\|\rho\|_{L^\infty}\bigr)
 \left(1+\log\frac1\eps\right).
\end{equation}
\end{proposition}

\begin{proof}
Let $K=\nabla g$.  The standard near--far decomposition, using
$|K(z)|\le C_d|z|^{1-d}$, gives
\[
 \|K*\rho\|_\infty\le C_d(\|\rho\|_1+\|\rho\|_\infty),
\]
and the same decomposition gives continuity of the convolution.  The log-Lipschitz regularity of the Coulomb field for bounded densities is
classical; see, for example, \cite[Lemma~4.1]{Serfaty111}.  The Zygmund form is
useful here because even mollification gains one power of the mollification
scale.

Let $\delta_h^2u(x)=u(x+h)+u(x-h)-2u(x)$ and put $r=|h|\le1/4$.
Since $|D^2K(z)|\le C_d|z|^{-d-1}$ away from the origin, splitting the
convolution into $|z|\le4r$, $4r<|z|<1$, and $|z|\ge1$ yields
\[
 \|\delta_h^2u\|_\infty
 \le C_d\|\rho\|_\infty r
 +C_dr^2\|\rho\|_\infty\int_{4r}^1q^{-2}\,\dd q
 +C_dr^2\|\rho\|_1
 \le C_d(\|\rho\|_1+\|\rho\|_\infty)|h|.
\]
For $|h|>1/4$, boundedness of $u$ gives the same estimate after enlarging the
constant.  Hence \eqref{eq:coulomb-unified-field-bound} holds, and the standard
Zygmund embedding gives \eqref{eq:coulomb-unified-loglip}.

Finally, evenness of $\chi$ gives
\[
 u*\chi_\eps-u=\frac12\int\delta_z^2u\,\chi_\eps(z)\,\dd z,
\]
which proves the first estimate in \eqref{eq:coulomb-unified-mollifier-estimates}.
Since $\int\nabla\chi_\eps=0$, the log-Lipschitz bound gives
\[
 \|\nabla(u*\chi_\eps)\|_\infty
 \le \int |u(\cdot-z)-u(\cdot)|\,|\nabla\chi_\eps(z)|\,\dd z
 \le C_d(\|\rho\|_1+\|\rho\|_\infty)
 \left(1+\log\frac1\eps\right),
\]
after absorbing the fixed mollifier norm into the constant.
\end{proof}

\begin{proposition}[Coulomb field estimates under an $L^\infty$ density bound]
\label{prop:coulomb-scale-covariant-field}
Let $d\ge2$, let $\rho$ be a probability density on $\R^d$, and assume
$\|\rho\|_{L^\infty}\le m$ for some $m>0$.  Set
\[
 R:=m^{-1/d},\qquad u:=\nabla g_d*\rho.
\]
Then
\begin{equation}\label{eq:coulomb-scale-field-bounds}
 \|u\|_{L^\infty}\le C_d m^{1-1/d},
 \qquad [u]_{\Lambda_*}\le C_d m,
\end{equation}
and, for every $x,y\in\R^d$,
\begin{equation}\label{eq:coulomb-scale-loglip}
 |u(x)-u(y)|
 \le C_d m|x-y|\left(1+\log_+\frac{R}{|x-y|}\right).
\end{equation}
For every collision-free configuration $X_N$,
\begin{equation}\label{eq:coulomb-scale-pairing}
 \int_{\R^d}\bigl(|K_N(x)|+|u(x)|\bigr)\rho(x)\,\dd x
 \le C_d m^{1-1/d}.
\end{equation}
Finally, let $\chi\in C_c^\infty(B_1)$ be the fixed even mollifier of unit mass
used above.  For $0<\vartheta\le e^{-2}$ set
\[
 u_{\vartheta}:=u*\chi_{R\vartheta}.
\]
Then
\begin{equation}\label{eq:coulomb-scale-mollifier}
 \|u-u_{\vartheta}\|_{L^\infty}
 \le C_d m^{1-1/d}\vartheta,
 \qquad
 \|\nabla u_{\vartheta}\|_{L^\infty}
 \le C_d m\left(1+\log\frac1\vartheta\right).
\end{equation}
The constants are independent of $m$, $N$, and the minimum particle
separation.
\end{proposition}

\begin{proof}
Put $R=m^{-1/d}$ and define
\[
 \widetilde\rho(\xi):=R^d\rho(R\xi).
\]
Then $\widetilde\rho$ has unit mass and
$\|\widetilde\rho\|_{L^\infty}\le R^dm=1$.  Since the Coulomb force is
homogeneous of degree $1-d$, if
$\widetilde u:=\nabla g_d*\widetilde\rho$, then
\begin{equation}\label{eq:coulomb-scale-rescaling}
 u(R\xi)=R^{1-d}\widetilde u(\xi).
\end{equation}
Proposition~\ref{prop:unified-coulomb-field}, applied to
$\widetilde\rho$, gives
\[
 \|\widetilde u\|_\infty+[\widetilde u]_{\Lambda_*}\le C_d,
 \qquad
 |\widetilde u(\xi)-\widetilde u(\eta)|
 \le C_d|\xi-\eta|\left(1+\log_+\frac1{|\xi-\eta|}\right).
\]
Rescaling by \eqref{eq:coulomb-scale-rescaling} yields
\eqref{eq:coulomb-scale-field-bounds} and
\eqref{eq:coulomb-scale-loglip}, because
$R^{1-d}=m^{1-1/d}$ and $R^{-d}=m$.

To prove \eqref{eq:coulomb-scale-pairing}, fix a center $a$ and split at
radius $R$.  Since $|\nabla g_d(z)|\le C_d|z|^{1-d}$,
\[
 \int |\nabla g_d(x-a)|\rho(x)\,\dd x
 \le C_d m\int_0^R\dd r+C_d R^{1-d}\int\rho
 \le C_d m^{1-1/d}.
\]
Averaging this estimate over the particle centers controls the $K_N$ term,
while the $u$ term follows from \eqref{eq:coulomb-scale-field-bounds}.

For the mollifier bounds, the change of variables $x=R\xi$ shows that
convolution at the scale $R\vartheta$ corresponds exactly to
convolution of $\widetilde u$ at scale $\vartheta$:
\[
 (u*\chi_{R\vartheta})(R\xi)
 =R^{1-d}(\widetilde u*\chi_\vartheta)(\xi).
\]
Applying \eqref{eq:coulomb-unified-mollifier-estimates} to
$\widetilde u$ and rescaling once for the function and once more for its
gradient gives \eqref{eq:coulomb-scale-mollifier}.
\end{proof}

\begin{proposition}[Coulomb solutions with bounded density]
\label{prop:bounded-density-flow}
Let $d\ge2$ and $\rho_0\in\mathcal P_2(\R^d)\cap L^\infty(\R^d)$.  Then
\eqref{eq:mean-field-equation} has a unique global weak solution
\[
 \rho\in C([0,\infty);\mathcal P_2(\R^d))
 \cap L^\infty_{\mathrm{loc}}([0,\infty);L^1\cap L^\infty(\R^d)),
\]
and
\begin{equation}\label{eq:bounded-density-Linfty-decay}
 \|\rho_t\|_{L^\infty}
 \le \bigl(t+\|\rho_0\|_{L^\infty}^{-1}\bigr)^{-1}
 =\frac{\|\rho_0\|_{L^\infty}}
 {1+t\|\rho_0\|_{L^\infty}},\qquad t\ge0.
\end{equation}
For every $d\ge2$, the field $u_t=\nabla g_d*\rho_t$ has a jointly Borel,
bounded continuous representative.  Moreover,
\[
 \sup_{t\ge0}\bigl(\|u_t\|_{L^\infty}+[u_t]_{\Lambda_*}\bigr)
 \le C_d\bigl(1+\|\rho_0\|_{L^\infty}\bigr),
\]
so it is log-Lipschitz.  For every radial mollifier
$\chi\in C_c^\infty(B_1)$ as in Proposition~\ref{prop:unified-coulomb-field},
\[
 \|u_t-u_t*\chi_\eps\|_\infty
 \le C_d\bigl(1+\|\rho_0\|_{L^\infty}\bigr)\eps,
 \qquad
 \|\nabla(u_t*\chi_\eps)\|_\infty
 \le C_d\bigl(1+\|\rho_0\|_{L^\infty}\bigr)
 \left(1+\log\frac1\eps\right)
\]
uniformly in $t\ge0$ and $0<\eps\le e^{-2}$.  The characteristic equation
\[
 \dot\Psi^t(x)=-u_t(\Psi^t(x)),\qquad \Psi^0(x)=x,
\]
has a unique global Osgood flow and
\[
 \rho_t=(\Psi^t)_\#\rho_0.
\]
\end{proposition}

\begin{proof}
The existence of nonnegative global weak solutions, uniqueness in the bounded
class, and the precise $L^\infty$ decay estimate
\eqref{eq:bounded-density-Linfty-decay} are contained respectively in
Theorems~3.1, 4.1, and 5.1 of \cite{Serfaty111}.  In particular,
Theorem~5.1 states the initial-data-dependent bound
$\|\rho_t\|_\infty\le(t+\|\rho_0\|_\infty^{-1})^{-1}$ used here.
Lemma~5.2 there gives Wasserstein absolute continuity;
on the $\mathcal P_2$ class considered here this yields the stated time
continuity.  The $L^\infty$ estimate, initially valid for almost every time,
extends to every time by narrow continuity.

Mass conservation and \eqref{eq:bounded-density-Linfty-decay}, followed by
Proposition~\ref{prop:unified-coulomb-field}, give the uniform field,
log-Lipschitz, and mollification estimates.  To obtain a jointly Borel
representative, let $K=\nabla g_d$ and approximate it by bounded continuous
kernels $K_{\delta,R}$ obtained by cutting off near the origin and outside a
large ball.  Narrow continuity of $t\mapsto\rho_t$ makes
$(t,x)\mapsto K_{\delta,R}*\rho_t(x)$ jointly continuous.  The near cutoff
error is $O(\|\rho_0\|_\infty\delta)$ and the far cutoff error is
$O(R^{1-d})$, uniformly on finite time intervals.  Letting
$\delta\downarrow0$ and $R\uparrow\infty$ gives the required representative.
The bounded log-Lipschitz field satisfies an Osgood modulus with an
$L^1_{\mathrm{loc}}$ time coefficient, so
Proposition~\ref{prop:unified-lagrangian-representation} and the superposition
principle give the unique global flow and the pushforward identity.
\end{proof}

For later use in the weighted Coulomb comparison, the density envelope
\eqref{eq:coulomb-density-envelope} satisfies the exact identity
\begin{equation}\label{eq:explicit-coulomb-osgood-integral}
 \int_0^T m(t)\,\dd t
 =\log(1+TM_0),
 \qquad M_0=\|\rho_0\|_{L^\infty}.
\end{equation}
This logarithmic time integral is the source of the polynomial deterioration
in \eqref{eq:coulomb-osgood-exponent}.

\section{Coulomb flows with bounded density}
\label{sec:newtonian-coulomb-proof}
Throughout this section $d\ge2$ and $g=g_d$ is the normalized Coulomb
kernel from \eqref{eq:unified-coulomb-kernel}.  We use
\[
 M_0:=\|\rho_0\|_{L^\infty},\qquad
 m(t)=\frac{M_0}{1+tM_0},\qquad
 R(t):=m(t)^{-1/d},
\]
as in \eqref{eq:coulomb-density-envelope}.  Thus $m'=-m^2$ and
$\|\rho_t\|_\infty\le m(t)$.  The normalized weight $\widehat m$ is defined in
\eqref{eq:coulomb-normalized-transport-weight}.
For every $d\ge2$, the Coulomb force satisfies
$|\nabla g(z)|\lesssim_d |z|^{1-d}$ and is locally integrable.  The comparison
argument is the same in every dimension.  Its dimension-dependent inputs are
the finite-$N$ lower bound and commutator estimate, together with the separation
and energy bounds needed to justify the singular particle dynamics.  The planar
logarithmic case additionally requires its own normalization and moment estimates.

The proof has four steps: the singular modulated energy identity; the
weighted transport estimate; the estimate with mollification radius $R(t)\vartheta$; and
finally averaging, dyadic optimization, and logarithmic comparison.

\begin{lemma}[Coulomb lower bound and commutator estimate]
\label{lem:unified-coulomb-static}
Let $d\ge2$, let $\mu\in\mathcal P_2(\R^d)\cap L^\infty(\R^d)$, and let
$X_N\in\Omega_N$.  With $g=g_d$ and the correction
\eqref{eq:coulomb-finite-N-correction}, the dimensional constants can be chosen
so that
\begin{equation}\label{eq:coulomb-static-lower-bound}
 F_N(X_N,\mu)+\eta_{N,d}(\mu)\ge0,
\end{equation}
and, for every $v\in W^{1,\infty}(\R^d;\R^d)$,
\begin{equation}\label{eq:coulomb-static-commutator}
 |C_N(X_N,\mu;v)|
 \le C_d\|\nabla v\|_{L^\infty}
 \bigl(F_N(X_N,\mu)+\eta_{N,d}(\mu)\bigr).
\end{equation}
The constants are independent of $N$ and of the minimum interparticle
distance.
\end{lemma}

\begin{proof}
If $d\ge3$, then $g=|\mathbb S^{d-1}|^{-1}g_{d-2}$.  The modulated
energy and the first-order commutator scale by the same dimensional factor, so
Lemma~\ref{lem:static-riesz-transfer} gives
\eqref{eq:coulomb-static-lower-bound}--\eqref{eq:coulomb-static-commutator}
after enlarging the fixed constant $A_d$ in
\eqref{eq:coulomb-finite-N-correction}.  If $d=2$, the same statements are
Lemmas~\ref{lem:diagonal-excluded-lower-bound} and~\ref{lem:serfaty-lipschitz-commutator}.  The latter is stated below without a
restriction on $N\|\mu\|_\infty$, after a logarithmic dilation argument.
\end{proof}

The global collision-free Coulomb particle flow and the quantitative estimates
used below are recorded in Proposition~\ref{prop:coulomb-particle-flow} in the
appendix.  The dissipative modulated energy identity is standard in the Coulomb
mean-field method \cite{Serfaty1}.  For the bounded-density reference solution
used here, the identity is justified by the singular regularization argument
below.

\begin{proposition}[Coulomb singular chain rule]
\label{prop:coulomb-singular-chain-rule}
Let $d\ge2$, let $X_N$ be the collision-free Coulomb particle trajectory from
Proposition~\ref{prop:coulomb-particle-flow}, and let $\rho$ be the
solution with bounded density from Proposition~\ref{prop:bounded-density-flow}.
Then
$F_N(t):=F_N(X_N(t),\rho_t)$ belongs to $AC([0,T])$ and
\begin{equation}\label{eq:coulomb-singular-chain-rule}
 \frac{\dd}{\dd t}F_N(t)=-2r_N(t)-C_N[u_t]
 \qquad\text{for a.e. }t\in[0,T].
\end{equation}
Let $\zeta\in C_c^\infty(B_1)$ be nonnegative, radial, and of unit mass, set
$\psi_\kappa:=\zeta_\kappa*\zeta_\kappa$, and define the same even radial
regularization in every dimension by
\[
 g^\kappa:=g*\psi_\kappa.
\]
For the associated quantities from Section~\ref{sec:abstract-framework},
\begin{equation}\label{eq:coulomb-chain-rule-convergences}
 F_N^\kappa\to F_N\ \text{in }C([0,T]),\qquad
 D_N^\kappa\to r_N\ \text{in }L^1(0,T),\qquad
 C_N^\kappa[u]\to C_N[u]\ \text{in }L^1(0,T).
\end{equation}
\end{proposition}

\begin{proof}
The proof is the same in all Coulomb dimensions once the normalized
fundamental solution $g_d$ is used.  Write
\[
 K:=\nabla g,\qquad
 K^\kappa:=\nabla g^\kappa,\qquad
 u^\kappa:=K^\kappa*\rho,
\]
and let $K_i^\kappa$ be the corresponding regularized discrete force.
For fixed $\kappa>0$, the regularized force $K^\kappa$ is smooth and bounded
near the origin, while $K_i\in L^2(0,T)$ by the microscopic dissipation estimate
and $u\in L^2(0,T;L^\infty)$ by Proposition~\ref{prop:bounded-density-flow}.
Consequently all differentiated discrete, mixed, and continuous Lagrangian terms
at the fixed regularization scale are integrable on $[0,T]$.  Thus the
admissibility hypothesis in Proposition~\ref{prop:unified-regularized-chain-rule}
\textup{(i)} is satisfied, and for every fixed $\kappa>0$,
\begin{equation}\label{eq:coulomb-regularized-chain-rule}
 (F_N^\kappa)'=-2D_N^\kappa-C_N^\kappa[u].
\end{equation}
We verify the three limits in
\eqref{eq:coulomb-chain-rule-convergences} directly and uniformly in the
dimension-dependent form of the Coulomb potential.

\emph{Step 1: integrability of $r_N$.}
If $d\ge3$, then $H_N\ge0$ and
Proposition~\ref{prop:coulomb-particle-flow} gives
\[
 \int_0^T\frac1N\sum_i|K_i(t)|^2\,\dd t
 =H_N(X_N^0)-H_N(X_N(T))\le H_N(X_N^0).
\]
If $d=2$, the same conclusion follows from
\eqref{eq:particle-force-error-integrability}.  Since
Proposition~\ref{prop:bounded-density-flow} gives
$u\in L^2(0,T;L^\infty)$, Proposition~\ref{prop:chain-rule-admissibility}
yields
\begin{equation}\label{eq:coulomb-rN-L1}
 r_N\in L^1(0,T).
\end{equation}

\emph{Step 2: convergence of the force term.}
The lower bound on the minimal interparticle distance in Proposition~\ref{prop:coulomb-particle-flow} implies
\[
 d_{N,T}:=\min_{0\le t\le T}\min_{i\ne j}|x_i(t)-x_j(t)|>0.
\]
Hence the regularized Coulomb force converges uniformly on the discrete
particle separations, and therefore
\[
 \delta_K(\kappa):=\max_i\sup_{0\le t\le T}
 |K_i^\kappa(t)-K_i(t)|\longrightarrow0.
\]
The mollification estimate of
Proposition~\ref{prop:unified-coulomb-field} also gives
\[
 \|u^\kappa-u\|_{L^2(0,T;L^\infty)}\longrightarrow0.
\]
Put
\[
 f_i:=K_i-u(x_i),\qquad
 e_i^\kappa:=(K_i^\kappa-K_i)-(u^\kappa-u)(x_i).
\]
Since $K_i^\kappa-u^\kappa(x_i)=f_i+e_i^\kappa$ and
$D_N^\kappa=N^{-1}\sum_i f_i\cdot(f_i+e_i^\kappa)$,
Cauchy--Schwarz and \eqref{eq:coulomb-rN-L1} give
\begin{align}
 \|D_N^\kappa-r_N\|_{L^1(0,T)}
 &\le \|r_N^{1/2}\|_{L^2(0,T)}
 \left(T^{1/2}\delta_K(\kappa)
 +\|u^\kappa-u\|_{L_t^2L_x^\infty}\right)
 \longrightarrow0.\label{eq:coulomb-D-limit-detailed}
\end{align}

\emph{Step 3: convergence of the modulated energy.}
For every $d\ge2$, the normalized Coulomb potential is locally integrable and
satisfies
\[
 |D^2g_d(z)|\le C_d|z|^{-d},\qquad z\ne0.
\]
Local $L^1$ approximation by convolution and the uniform density bound imply
\[
 \|g^\kappa-g_d\|_{L^1(B_2)}\longrightarrow0.
\]
On $|z|\ge2$, evenness of $\psi_\kappa$ and Taylor's formula give, for all
sufficiently small $\kappa$,
\[
 |g^\kappa(z)-g_d(z)|
 \le C_d\kappa^2|z|^{-d}\le C_d\kappa^2.
\]
Consequently
\begin{equation}\label{eq:coulomb-convolution-energy-limit}
 \sup_{0\le t\le T}
 \|(g^\kappa-g_d)*\rho_t\|_{L^\infty}
 \le C_d\|\rho\|_{L_t^\infty L_x^\infty}
 \|g^\kappa-g_d\|_{L^1(B_2)}
 +C_d\kappa^2\longrightarrow0.
\end{equation}
The discrete--discrete part of $F_N^\kappa$ converges uniformly on $[0,T]$
because of $d_{N,T}>0$.  The mixed term is controlled by
\eqref{eq:coulomb-convolution-energy-limit}, and the continuous--continuous
term is bounded by the same supremum after integration against $\rho_t$.
Thus
\begin{equation}\label{eq:coulomb-energy-limit-detailed}
 \sup_{0\le t\le T}|F_N^\kappa(t)-F_N(t)|\longrightarrow0.
\end{equation}

\emph{Step 4: convergence of the commutator.}
For every $d\ge2$, the Coulomb force satisfies
\[
 |K(z)|\le C_d|z|^{1-d}.
\]
Moreover the regularized force satisfies
\begin{equation}\label{eq:coulomb-regularized-kernel-pointwise}
 |K^\kappa(z)|\le C_d(|z|+\kappa)^{1-d}
 \le C_d|z|^{1-d},\qquad z\ne0.
\end{equation}
Indeed, for $|z|\ge4\kappa$ this follows by convolution away from the origin,
whereas for $|z|<4\kappa$ it follows by scaling of the Coulomb force.
By Proposition~\ref{prop:unified-coulomb-field}, for $0<|z|\le1/2$,
\begin{equation}\label{eq:coulomb-loglip-chain-rule}
 |u_t(a+z)-u_t(a)|
 \le C_d(1+m(t))|z|
 \left(1+\log\frac1{|z|}\right),
\end{equation}
uniformly in the center $a$.  Therefore, for either the singular or the
regularized kernel, every mixed near-diagonal contribution is bounded by
\[
 C_d(1+m(t))^2
 \int_0^\delta r\left(1+\log\frac1r\right)\,\dd r,
\]
and the same bound controls the continuous--continuous near-diagonal term
after integration in the outer variable.  This majorant tends to zero as
$\delta\downarrow0$ and belongs to $L^1(0,T)$ uniformly in $\kappa$.

The discrete--discrete commutator is a finite sum evaluated at distances at
least $d_{N,T}$, so it converges uniformly as $\kappa\downarrow0$.  For fixed
$\delta>0$, evenness of the mollifier and Taylor's formula applied to $K$
give
\[
 \sup_{|z|\ge\delta}|K^\kappa(z)-K(z)|\longrightarrow0.
\]
Hence the mixed and continuous--continuous contributions on
$|x-y|\ge\delta$ converge in $L^1(0,T)$ by boundedness of $u$, unit mass of
$\rho_t$, and dominated convergence.  Letting first $\kappa\downarrow0$ and
then $\delta\downarrow0$ yields
\begin{equation}\label{eq:coulomb-commutator-limit-detailed}
 C_N^\kappa[u]\longrightarrow C_N[u]
 \qquad\text{in }L^1(0,T).
\end{equation}
Thus the symmetrized Coulomb commutator is absolutely convergent in the bounded-density class.

\emph{Step 5: passage to the singular identity.}
Integrate \eqref{eq:coulomb-regularized-chain-rule} between arbitrary
$0\le s<t\le T$.  The limits
\eqref{eq:coulomb-D-limit-detailed},
\eqref{eq:coulomb-energy-limit-detailed}, and
\eqref{eq:coulomb-commutator-limit-detailed} satisfy the hypotheses of
Proposition~\ref{prop:unified-regularized-chain-rule}\textup{(ii)} and give
\eqref{eq:coulomb-chain-rule-convergences}.  Hence $F_N\in AC([0,T])$ and
\eqref{eq:coulomb-singular-chain-rule} holds for almost every time.
\end{proof}

\begin{proposition}[Weighted quadratic transport estimate]
\label{prop:coulomb-weighted-transport}
Let $Q_N$ be the quadratic transport cost along a coupled trajectory and set
\begin{equation}\label{eq:coulomb-weighted-Q-definition}
 P_N(t):=\widehat m(t)Q_N(t).
\end{equation}
Then, for almost every $t\in(0,T)$,
\begin{equation}\label{eq:coulomb-weighted-Q-inequality}
 P_N'(t)
 \le r_N(t)+C_d m(t)\omega_*(P_N(t)).
\end{equation}
The constant depends only on the dimension.
\end{proposition}

\begin{proof}
Write $m=m(t)$, $\widehat m=\widehat m(t)$,
$R=R(t)=m^{-1/d}$, and $e_i=K_i-u_t(x_i)$.  Since
$\widehat m'=-m\widehat m$, the exact identity in
Proposition~\ref{prop:unified-transport-cost} gives
\[
 P_N'
 =-m\widehat m Q_N-\frac{2\widehat m}{N}\sum_iq_i\cdot e_i
 -\frac{2\widehat m}{N}\sum_iq_i\cdot
   \bigl(u_t(x_i)-u_t(y_i)\bigr).
\]
Young's inequality in the form
\[
 -2\widehat m q_i\cdot e_i
 \le \frac{\widehat m}{m}|e_i|^2+m\widehat m|q_i|^2
\]
cancels the weighted quadratic term exactly.  Since
$\widehat m/m=(1+M_0)^{-1}\le1$,
\begin{equation}\label{eq:coulomb-weighted-Q-pre-modulus}
 P_N'
 \le r_N+\frac{2\widehat m}{N}\sum_i|q_i|\,
 |u_t(x_i)-u_t(y_i)|.
\end{equation}
If $q_i=0$, the corresponding summand in
\eqref{eq:coulomb-weighted-Q-pre-modulus} vanishes.  For $q_i\ne0$,
Proposition~\ref{prop:coulomb-scale-covariant-field} gives
\[
 |u_t(x_i)-u_t(y_i)|
 \le C_d m|q_i|\left(1+\log_+\frac{R}{|q_i|}\right).
\]
Set $z_i:=\widehat m|q_i|^2$ for all $i$.  By the definition of
$\widehat m$,
\[
 R^2\widehat m=\frac{m^{1-2/d}}{1+M_0}\le1.
\]
Indeed, the numerator equals one when $d=2$; when $d>2$, since
$0<1-2/d<1$ and $m(t)\le M_0$, one has
$m(t)^{1-2/d}\le M_0^{1-2/d}\le1+M_0$.  Hence
$R\sqrt{\widehat m}\le1$.  For $q_i\ne0$ this implies
\[
 \frac{R}{|q_i|}\le\frac1{\sqrt{z_i}},
\]
and therefore
\[
 2\widehat m|q_i|\,|u_t(x_i)-u_t(y_i)|
 \le C_d m\,z_i\left(1+\log_+\frac1{\sqrt{z_i}}\right)
 \le C_d m\,\omega_*(z_i).
\]
The same inequality is trivial when $q_i=0$, since then $z_i=0$ and
$\omega_*(0)=0$.
Concavity and Jensen's inequality give
\[
 \frac1N\sum_i\omega_*(z_i)
 \le\omega_*\!\left(\frac1N\sum_i z_i\right)
 =\omega_*(P_N).
\]
Substitution in \eqref{eq:coulomb-weighted-Q-pre-modulus} proves the claim.
\end{proof}

\begin{proposition}[Coulomb estimate with mollification radius $R(t)\vartheta$]
\label{prop:unified-coulomb-fixed-scale}
Let $d\ge2$.  For a dyadic dimensionless scale
$0<\vartheta\le e^{-2}$ define the mollification radius
\begin{equation}\label{eq:coulomb-physical-mollification-scale}
 \eps_\vartheta(t):=R(t)\vartheta
 =m(t)^{-1/d}\vartheta,
\end{equation}
put
\[
 u_\vartheta(t):=u_t*\chi_{\eps_\vartheta(t)},
 \qquad w_\vartheta:=u-u_\vartheta,
 \qquad r_N^\vartheta:=r_N[u_\vartheta].
\]
Then, for almost every $t\in(0,T)$,
\begin{equation}\label{eq:coulomb-unified-fixed-scale}
 -r_N-C_N[u]
 \le-r_N^\vartheta
 +C_d m(t)\ell(\vartheta)
 \bigl(F_N+\eta_{N,d}\bigr)
 +C_d m(t)^{2-2/d}\vartheta,
\end{equation}
and
\begin{equation}\label{eq:coulomb-unified-force-comparison}
 r_N\le2r_N^\vartheta
 +C_d m(t)^{2-2/d}\vartheta^2.
\end{equation}
All displayed constants depend only on the dimension and are independent of
$N$ and of the minimum interparticle distance.  The dimensionless parameter
$\vartheta$ is fixed in the pathwise estimate, while the mollification radius
$\eps_\vartheta(t)$ varies deterministically with time.  The estimate is
pointwise in time, so this time dependence is external to the differentiation
argument.
\end{proposition}

\begin{proof}
Fix a time for which the pointwise estimates hold and abbreviate
$m=m(t)$ and $R=R(t)$.  The density envelope gives
$\|\rho_t\|_\infty\le m$.  Proposition~\ref{prop:coulomb-scale-covariant-field}
therefore gives
\begin{equation}\label{eq:coulomb-continuous-field-pairing}
 \int_{\R^d}(|K_N(x)|+|u(x)|)\rho_t(x)\,\dd x
 \le C_d m^{1-1/d},
\end{equation}
uniformly in $N$ and the configuration, as well as
\begin{equation}\label{eq:coulomb-rough-mollification-error}
 \|w_\vartheta\|_\infty
 \le C_d m^{1-1/d}\vartheta,
 \qquad
 \|\nabla u_\vartheta\|_\infty
 \le C_d m\ell(\vartheta).
\end{equation}
The completion of squares identity gives
\[
 -r_N-C_N[w_\vartheta]
 =-r_N^\vartheta
 +\|w_\vartheta\|_{L^2(\mu_N)}^2
 +2\int w_\vartheta\cdot(K_N-u)\rho_t.
\]
The preceding bounds yield
\[
 \|w_\vartheta\|_{L^2(\mu_N)}^2
 \le C_d m^{2-2/d}\vartheta^2,
 \qquad
 2\left|\int w_\vartheta\cdot(K_N-u)\rho_t\right|
 \le C_d m^{2-2/d}\vartheta.
\]
Since $0<\vartheta<1$,
\begin{equation}\label{eq:coulomb-rough-remainder-bound}
 -r_N-C_N[w_\vartheta]
 \le-r_N^\vartheta+C_d m^{2-2/d}\vartheta.
\end{equation}
The elementary force comparison similarly gives
\[
 r_N\le2r_N^\vartheta
 +2\|w_\vartheta\|_\infty^2
 \le2r_N^\vartheta+C_d m^{2-2/d}\vartheta^2.
\]
For the commutator with the mollified field,
Lemma~\ref{lem:unified-coulomb-static}, applied to $\rho_t$, and
\eqref{eq:coulomb-rough-mollification-error} yield
\[
 |C_N[u_\vartheta]|
 \le C_d m\ell(\vartheta)
 \bigl(F_N+\eta_{N,d}(\rho_t)\bigr).
\]
The sharp density decay implies $\eta_{N,d}(\rho_t)\le\eta_{N,d}$, while the
finite-$N$ lower bound gives $F_N+\eta_{N,d}\ge0$.  Combining the last estimate
with \eqref{eq:coulomb-rough-remainder-bound} proves
\eqref{eq:coulomb-unified-fixed-scale}.
\end{proof}

For the averaged Coulomb estimates, the two time coefficients are
$C_d m(t)$ and $C_d m(t)^{2-2/d}$, with the dimensional constant enlarged when
necessary.  The first one satisfies
\[
 \int_0^T C_d m(t)\,\dd t=C_d\log(1+TM_0),
\]
and the second satisfies
\begin{equation}\label{eq:coulomb-linear-drift-integral}
 \int_0^T m(t)^{2-2/d}\,\dd t
 =\begin{cases}
 \log(1+TM_0),&d=2,\\[1mm]
 \displaystyle\frac{d}{d-2}M_0^{1-2/d}
 \left[1-(1+TM_0)^{-(1-2/d)}\right],&d\ge3.
 \end{cases}
\end{equation}
Thus the second coefficient is integrable on every finite interval; for
$d\ge3$ its integral is even bounded uniformly as $T\to\infty$.

\begin{proof}[Proof of Theorem~\ref{thm:bounded-density-coulomb-stability}]
Let $\pi_N^0$ be a symmetric optimal coupling of $\rho_N^0$ and
$\rho_0^{\otimes N}$, and push it forward by the particle flow and the
characteristic flow of the limiting equation.  Define
\begin{equation}\label{eq:coulomb-weighted-averaged-energy}
 \mathcal E_N(t):=
 \mathbb E_{\pi_N^0}\!\left[
 \widehat m(t)Q_N(t)+F_N(X_N(t),\rho_t)+\eta_{N,d}\right]
 +\begin{cases}
  \dfrac{1+\log N}{N},&d=2,\\[2mm]
  N^{-2/d},&d\ge3.
 \end{cases}
\end{equation}
For a fixed dyadic $\vartheta$, write
\[
 \overline r_N^{\vartheta}(t)
 :=\mathbb E_{\pi_N^0}r_N^{\vartheta}(t).
\]
Proposition~\ref{prop:coulomb-measurability-statement} justifies
differentiation under the fixed expectation and places the entire dyadic
family on a common set of full measure.  In particular,
$\mathcal E_N\in AC([0,T])$; moreover the last term in
\eqref{eq:coulomb-weighted-averaged-energy} is strictly positive, so
$\mathcal E_N(t)>0$ for every $t$ and the
choice of dyadic scale below is never evaluated at the endpoint $z=0$.  Propositions~\ref{prop:coulomb-weighted-transport},~\ref{prop:coulomb-singular-chain-rule},
and~\ref{prop:unified-coulomb-fixed-scale}, followed by Jensen's inequality,
give for every fixed dyadic $0<\vartheta\le e^{-2}$
\begin{align}
 \mathcal E_N'
 +\overline r_N^{\vartheta}
 &\le C_d m(t)
 \left[\omega_*(\mathcal E_N)
 +\ell(\vartheta)\mathcal E_N\right]
 +C_d m(t)^{2-2/d}\vartheta,
 \label{eq:coulomb-averaged-fixed-scale}\\
 \overline r_N
 &\le2\overline r_N^{\vartheta}
 +C_d m(t)^{2-2/d}\vartheta^2.
 \label{eq:coulomb-averaged-force-comparison}
\end{align}
After increasing $C_d$ if necessary, the Osgood terms carry only the
coefficient $m(t)$, while the higher power $m(t)^{2-2/d}$ multiplies only the scale
errors.

We first treat small initial error.  At $t=0$,
\begin{equation}\label{eq:coulomb-weighted-initial-error}
 \mathcal E_N(0)
 =\widehat m(0)\frac1N
 W_2^2(\rho_N^0,\rho_0^{\otimes N})
 +\int(F_N(X_N,\rho_0)+\eta_{N,d})\,\dd\rho_N^0
 +\begin{cases}
  \dfrac{1+\log N}{N},&d=2,\\[2mm]
  N^{-2/d},&d\ge3.
 \end{cases}
\end{equation}
Since $\widehat m(0)=M_0/(1+M_0)<1$,
\begin{equation}\label{eq:coulomb-weighted-initial-comparison}
 0<\mathcal E_N(0)\le a_{N,d}.
\end{equation}
While $0<\mathcal E_N(t)\le e^{-3}$, choose after averaging
the measurable dyadic scale
\[
 \vartheta_*(t):=\eps_1(\mathcal E_N(t))
\]
from Lemma~\ref{lem:measurable-dyadic-scale}.  The common set of full measure in
Proposition~\ref{prop:coulomb-measurability-statement} permits this insertion.
The dyadic estimates give
\[
 \vartheta_*\le\mathcal E_N<2\vartheta_*,\qquad
 \ell(\vartheta_*)\mathcal E_N
 \le C\omega_*(\mathcal E_N).
\]
After enlarging the dimensional constants once more,
\begin{equation}\label{eq:coulomb-drifted-osgood-inequality}
 \mathcal E_N'
 +\overline r_N^{\vartheta_*}
 \le C_d m(t)\omega_*(\mathcal E_N)
 +C_d m(t)^{2-2/d}\mathcal E_N.
\end{equation}
Lemma~\ref{lem:logarithmic-osgood-linear-drift}, applied with
$a(t)=C_d m(t)$ and $b(t)=C_d m(t)^{2-2/d}$, shows that for
$\mathcal E_N(0)$ below a threshold depending only on the fixed comparison
data,
\begin{equation}\label{eq:coulomb-drifted-majorant}
 \sup_{t\le T}\mathcal E_N(t)
 +\int_0^T\overline r_N^{\vartheta_*}(t)\,\dd t
 \le C_{T,d,M_0}\,\mathcal E_N(0)^{\gamma_{T,d}},
\end{equation}
The exponent supplied directly by
Lemma~\ref{lem:logarithmic-osgood-linear-drift} is bounded below by
$\exp(-C_d\int_0^T m)$.  Increasing the dimensional constant $c_d$ in
\eqref{eq:coulomb-osgood-exponent} if necessary gives
\[
 \exp\!\left(-C_d\int_0^T m(t)\,\dd t\right)
 \ge\gamma_{T,d}
 =(1+TM_0)^{-c_d},
\]
so \eqref{eq:coulomb-drifted-majorant} follows after weakening the power if
necessary.  The factor $C_{T,d,M_0}$ contains
$\exp(C_d\int_0^Tm^{2-2/d})$; this is precisely why the linear drift
does not affect the Osgood exponent.

Let $V$ be the scalar comparison solution in
Lemma~\ref{lem:logarithmic-osgood-linear-drift}.  Since
$\vartheta_*^2\le\vartheta_*\le\mathcal E_N\le V$ and
$V'\ge C_d m(t)^{2-2/d}V$, one has
\[
 \int_0^T C_d m(t)^{2-2/d}\vartheta_*(t)^2\,\dd t
 \le\int_0^T C_d m(t)^{2-2/d}V(t)\,\dd t
 \le V(T)-V(0).
\]
Combining this with \eqref{eq:coulomb-averaged-force-comparison} and
\eqref{eq:coulomb-drifted-majorant} gives
\begin{equation}\label{eq:coulomb-drifted-force-control}
 \int_0^T\overline r_N(t)\,\dd t
 \le C_{T,d,M_0}\,\mathcal E_N(0)^{\gamma_{T,d}}.
\end{equation}

For initial errors outside the small regime, fix one dyadic
$\bar\vartheta\in(0,e^{-2}]$ in
\eqref{eq:coulomb-averaged-fixed-scale}.  Since
$\omega_*(z)\le C(1+z)$ and both $m$ and $m^{2-2/d}$ belong to
$L^1(0,T)$, Gronwall's
lemma and \eqref{eq:coulomb-averaged-force-comparison} give
\[
 \sup_{t\le T}\mathcal E_N(t)
 +\int_0^T\overline r_N(t)\,\dd t
 \le C_{T,d,M_0}(1+\mathcal E_N(0)).
\]
Combining this with the small error estimate gives the claimed global form.
Indeed, if $a_*>0$ denotes the smallness threshold in Lemma~\ref{lem:logarithmic-osgood-linear-drift} and
$\mathcal E_N(0)\ge a_*$, then, since $0<\gamma_{T,d}\le1$,
$1\le a_*^{-\gamma_{T,d}}\mathcal E_N(0)^{\gamma_{T,d}}$.  Hence the
bound obtained at the fixed scale, $C(1+\mathcal E_N(0))$ is bounded by
$C_{T,d,M_0}(\mathcal E_N(0)^{\gamma_{T,d}}+\mathcal E_N(0))$.
Together with $\mathcal E_N(0)\le a_{N,d}$, this yields, for every
initial error,
\begin{equation}\label{eq:coulomb-weighted-final-energy}
 \sup_{t\le T}\mathcal E_N(t)
 +\int_0^T\overline r_N(t)\,\dd t
 \le C_{T,d,\rho_0}
 \bigl(a_{N,d}^{\gamma_{T,d}}+a_{N,d}\bigr).
\end{equation}

We finally recover the unweighted Wasserstein distance.  Since
$\widehat m$ is decreasing and positive,
\[
 \frac1N W_2^2(\rho_N(t),\rho_t^{\otimes N})
 \le\mathbb E_{\pi_N^0}Q_N(t)
 \le\widehat m(T)^{-1}
 \mathcal E_N(t).
\]
Here
\[
 \widehat m(T)^{-1}
 =\frac{(1+M_0)(1+TM_0)}{M_0},
\]
which depends only on the fixed comparison data.  Moreover,
\[
 \int(F_N+\eta_{N,d})\,\dd\rho_N(t)
 \le\mathcal E_N(t).
\]
Thus \eqref{eq:coulomb-weighted-final-energy} proves
\eqref{eq:coulomb-main-estimate}.

For $\rho_N^0=\rho_0^{\otimes N}$, the Wasserstein part of $a_{N,d}$ vanishes, and
Lemma~\ref{lem:unified-product-initial-data} gives the single identity
\[
 \int(F_N(X_N,\rho_0)+\eta_{N,d})\,\dd\rho_0^{\otimes N}(X_N)
 =\eta_{N,d}-\frac1N E(\rho_0).
\]
The left-hand side is nonnegative by the finite-$N$ lower bound.  The Coulomb
near-field/far-field estimates imply $|E(\rho_0)|<\infty$ for every $d\ge2$.  Together
with \eqref{eq:coulomb-finite-N-correction} this gives, uniformly for $N\ge2$,
\begin{equation}\label{eq:coulomb-product-initial-error-unified}
 0\le a_{N,d}\le C_{d,\rho_0}
 \begin{cases}
  \dfrac{1+\log N}{N},&d=2,\\[2mm]
  N^{-2/d},&d\ge3.
 \end{cases}
\end{equation}
Both quantities on the right are smaller than one for $N\ge2$.  Since
$0<\gamma_{T,d}\le1$,
\[
 a_{N,d}^{\gamma_{T,d}}+a_{N,d}
 \le C_{T,d,\rho_0}
 \begin{cases}
  \left(\dfrac{1+\log N}{N}\right)^{\gamma_{T,2}},&d=2,\\[3mm]
  N^{-2\gamma_{T,d}/d},&d\ge3.
 \end{cases}
\]
Substitution in \eqref{eq:coulomb-main-estimate} proves
\eqref{eq:coulomb-product-rate} with the same argument for every $d\ge2$.
\end{proof}

\section{Besov estimates for the super-Coulomb Riesz regime}\label{sec:critical-riesz-proof}
The proof of Theorem~\ref{thm:critical-riesz-propagation} begins with the
field estimates implied by the assumed Besov regularity, proceeds to the
regularized commutator estimate, and finally optimizes the mollification scale
after averaging.  The last step is treated by the corresponding Gronwall, Bihari, or Osgood
inequality, according to $q$.

Let $1\le q\le\infty$.  By
Proposition~\ref{prop:critical-riesz-borel-reference}, the time-dependent
Besov norm has a bounded Borel representative, which we use without changing
notation.  Set
\begin{equation}\label{eq:critical-riesz-measurable-besov-bound}
 \mathcal B_q(t):=1+M_T+\|\rho_t\|_{B^\alpha_{\infty,q}},
\end{equation}
and, for use throughout this section, put
\begin{equation}\label{eq:finite-q-log-exponent}
 \beta_q:=1-\frac1q\quad(q<\infty),\qquad
 \beta_\infty:=1,\qquad \sigma_s:=d-s\in(0,2).
\end{equation}
We also write $\ell(\eps)=1+\log(1/\eps)$.  We use $\kappa$ for regularization
of the singular kernel, $\delta$ for the
centered cutoff, and $\eps$ for mollification of the limiting field.

We use the collision-free particle flow of
Proposition~\ref{prop:riesz-global-particle-flow} and the singular chain rule of
Proposition~\ref{prop:riesz-singular-chain-rule}.  The mollification scale is
kept fixed in the pathwise estimate and selected after averaging.

The threshold $s=d-1$ separates two analytic regimes.  If
$d-2<s<d-1$, then $\nabla g_s\in L^1_{\mathrm{loc}}$ and no first-order increment of the density is needed to estimate the continuous term.  If $d-1\le s<d$, then
$\alpha\ge1$ and the first-order increment supplied by the assumed Besov
regularity compensates the loss of local integrability when the same centered
cutoff is used in all terms.  This split is made explicit in
Lemma~\ref{lem:critical-riesz-rough-local} and
Proposition~\ref{prop:critical-riesz-pv-statement}.
For $\delta>0$ set
\begin{equation}\label{eq:critical-riesz-truncated-kernel}
 G_{s,\delta}(z):=\mathbf1_{\{|z|>\delta\}}\nabla g_s(z).
\end{equation}

\subsection{Besov estimates for the interaction field}

\subsubsection{Almost-Lipschitz and Zygmund estimates}
\begin{lemma}
\label{lem:critical-riesz-product-modulus}
Let $1\le q\le\infty$, let $\beta_q$ be given by
\eqref{eq:finite-q-log-exponent}, and let
$f\in L^\infty(\R^d)\cap B^\alpha_{\infty,q}(\R^d)$ with
$\alpha\in[1,2)$.  Then the Littlewood--Paley series of $f$ converges
uniformly to a continuous representative and
\begin{equation}\label{eq:critical-besov-zygmund-modulus}
 |f(x)-f(y)|
 \le C_{d,\alpha,q}\|f\|_{B^\alpha_{\infty,q}}
 |x-y|\left(1+\log_+\frac1{|x-y|}\right)^{\beta_q}.
\end{equation}
Let $u\in B^1_{\infty,q}$, let $\chi$ be a smooth, compactly supported,
even mollifier of mass one, put $u_\eps=u*\chi_\eps$, and set
$w_\eps=u-u_\eps$.  For $0<\eps\le\frac14$,
\[\begin{aligned}
 \|w_\eps\|_\infty
 &\le C_q\eps\|u\|_{B^1_{\infty,q}},\\
 |w_\eps(x)-w_\eps(y)|
 &\le C_q\|u\|_{B^1_{\infty,q}}
 \min\!\left\{\eps,
 |x-y|\left(1+\log_+\frac1{|x-y|}\right)^{\beta_q}\right\},\\
 \|\nabla u_\eps\|_\infty
 &\le C_q\|u\|_{B^1_{\infty,q}}
 \left(1+\log\frac1\eps\right)^{\beta_q}.
\end{aligned}
\]
Consequently, for
$\rho\in L^\infty\cap B^\alpha_{\infty,q}$,
\begin{equation}\label{eq:critical-riesz-product-modulus}
\begin{aligned}
|w_\eps(x)\rho(x)-w_\eps(y)\rho(y)|\le C_q\|u\|_{B^1_{\infty,q}}
 \bigl(\|\rho\|_\infty+\|\rho\|_{B^\alpha_{\infty,q}}\bigr)
 \left[
 \min\!\left\{\eps,r\ell(r)^{\beta_q}\right\}
 +\eps r\ell(r)^{\beta_q}\right],
\end{aligned}
\end{equation}
where $r=|x-y|\le1/2$.
\end{lemma}

\begin{proof}
The modulus, approximation, and gradient estimates are standard consequences of
the Littlewood--Paley characterization of Besov spaces and Bernstein
inequalities; see \cite[Chapters~1--2]{Bahouri}.  For the last estimate write
\[
 w_\eps(x)\rho(x)-w_\eps(y)\rho(y)
 =w_\eps(x)\bigl(\rho(x)-\rho(y)\bigr)
 +\rho(y)\bigl(w_\eps(x)-w_\eps(y)\bigr),
\]
and apply the preceding bounds to $w_\eps$ and to the
$B^\alpha_{\infty,q}$ representative of $\rho$.
\end{proof}

\subsubsection{Estimates for the Riesz interaction field}
\begin{proposition}[Riesz field under the assumed Besov regularity]
\label{prop:critical-riesz-field}
Let $d-2<s<d$, set $\alpha=s-d+2\in(0,2)$, and let
$1\le q\le\infty$.  Assume
\[
 \rho\in L^1(\R^d)\cap L^\infty(\R^d)
 \cap B^\alpha_{\infty,q}(\R^d).
\]
Define $u=\nabla g_s*\rho$ in distributions.  Then $u$ has a bounded
continuous representative in $B^1_{\infty,q}$ and
\begin{equation}\label{eq:critical-riesz-zygmund-bound}
 \|u\|_{L^\infty}+\|u\|_{B^1_{\infty,q}}
 \le C_{d,s,q}\left(
 \|\rho\|_{L^1}+\|\rho\|_{L^\infty}
 +\|\rho\|_{B^\alpha_{\infty,q}}\right).
\end{equation}
With $\beta_q$ as in \eqref{eq:finite-q-log-exponent}, for every smooth
radial even mollifier and $0<\eps\le e^{-2}$,
\begin{equation}\label{eq:critical-riesz-mollifier-estimates}
\begin{aligned}
 \|u-u_\eps\|_{L^\infty}
 &\le C_{d,s,q}\left(1+\|\rho\|_{L^1}+\|\rho\|_{L^\infty}+\|\rho\|_{B^\alpha_{\infty,q}}\right)\eps,\\
 \|\nabla u_\eps\|_{L^\infty}
 &\le C_{d,s,q}\left(1+\|\rho\|_{L^1}+\|\rho\|_{L^\infty}+\|\rho\|_{B^\alpha_{\infty,q}}\right)\ell(\eps)^{\beta_q}.
\end{aligned}
\end{equation}
In addition,
\begin{equation}\label{eq:critical-riesz-loglip}
 |u(x)-u(y)|
 \le C_{d,s,q}\left(1+\|\rho\|_{L^1}+\|\rho\|_{L^\infty}+\|\rho\|_{B^\alpha_{\infty,q}}\right)|x-y|
 \left(1+\log_+\frac1{|x-y|}\right)^{\beta_q}.
\end{equation}
If $\alpha\ge1$, the continuous density representative satisfies
\begin{equation}\label{eq:critical-riesz-density-modulus}
 |\rho(x)-\rho(y)|
 \le C_{d,s,q}\left(1+\|\rho\|_{L^1}+\|\rho\|_{L^\infty}+\|\rho\|_{B^\alpha_{\infty,q}}\right)|x-y|
 \left(1+\log_+\frac1{|x-y|}\right)^{\beta_q}.
\end{equation}
For $q=\infty$, this is the Zygmund estimate used below; for $q=1$, $B^1_{\infty,1}$ embeds into $W^{1,\infty}$.
When $\rho=\rho_t$ depends on time, the preceding estimates hold for almost
every $t$.  Proposition~\ref{prop:critical-riesz-borel-reference} extends the
Besov bound to every time and provides a jointly Borel representative of the
field, which we continue to denote by $u_t$.
\end{proposition}

\begin{proof}
The multiplier of $u$ is $c_{d,s}i\xi|\xi|^{s-d}$, of order
$\alpha-1$.  The standard Besov multiplier theorem
\cite[Chapters~1--2]{Bahouri} therefore gives, for every high-frequency block,
\[
 2^j\|\Delta_j u\|_{L^\infty}
 \le C_{d,s}2^{j\alpha}\|\Delta_j\rho\|_{L^\infty}.
\]
For the low-frequency block, the localized multiplier is integrable near the
origin because its radial $L^1_\xi$ integrand behaves like $r^s$; hence
$\|\Delta_{-1}u\|_{L^\infty}\le C_{d,s}\|\rho\|_{L^1}$.  These two estimates
prove \eqref{eq:critical-riesz-zygmund-bound} and give a bounded continuous
representative.  The remaining assertions follow from
Lemma~\ref{lem:critical-riesz-product-modulus}, applied to $u$ and, when
$\alpha\ge1$, to $\rho$.
\end{proof}

\begin{lemma}
\label{lem:critical-riesz-radial-integrals}
Let $0\le\beta\le1$, $0<r\le\frac12$, $0<\gamma\le2$,
$0<\sigma\le1$, and $0<\eps\le e^{-2}$.  Then
\begin{align}
 \int_0^r t^{\gamma-1}\ell(t)^\beta\,\dd t
 &\le C_{\gamma,\beta}r^\gamma\ell(r)^\beta,\notag\\
 \eps^2\int_{4\eps}^1t^{\gamma-3}\ell(t)^\beta\,\dd t
 &\le C_{\gamma,\beta}
 \begin{cases}
 \eps^\gamma\ell(\eps)^\beta,&0<\gamma<2,\\
 \eps^2\ell(\eps)^{\beta+1},&\gamma=2,
 \end{cases}
 \label{eq:critical-riesz-radial-mollifier}\\
 \int_0^1\min\{\eps,t\ell(t)^\beta\}t^{\sigma-2}\,\dd t
 &\le C_{\sigma,\beta}
 \begin{cases}
 \eps^\sigma\ell(\eps)^\beta,&0<\sigma<1,\\
 \eps\ell(\eps),&\sigma=1,
 \end{cases}
 \label{eq:critical-riesz-radial-rough-bound}\\
 \int_0^1t^{\sigma-1}\ell(t)^\beta\,\dd t
 &\le C_{\sigma,\beta}.\notag
\end{align}
\end{lemma}

\begin{proof}
The proof is given in Appendix~\ref{app:riesz-technical-proofs}.
\end{proof}

For $1\le q\le\infty$, $d-2<s<d$, and $0<\eps\le e^{-2}$, set
\[e_{q,s}(\eps):=
 \begin{cases}
  \eps,&d-2<s<d-1,\\
  \eps\ell(\eps),&s=d-1,\\
  \eps^{d-s}\ell(\eps)^{\beta_q},&d-1<s<d.
 \end{cases}
\]
For $q=\infty$, this is bounded by
$C\eps^{\min\{1,d-s\}}\ell(\eps)$, the bound used in the $q=\infty$ case of the Riesz theorem under the stated Besov regularity.

The three scales follow from the corresponding radial integrals.  In the
principal value range $d-1\le s<d$, recall that
\(\sigma_s=d-s\in(0,1]\) and put \(\beta:=\beta_q\).  The product estimate
\eqref{eq:critical-riesz-product-modulus} gives, with
\(r=|z|\le1/2\),
\[
 |w_\eps(a+z)\rho(a+z)-w_\eps(a)\rho(a)|
 \le C\mathcal B_q(t)^2
 \left[\min\{\eps,r\ell(r)^\beta\}
       +\eps r\ell(r)^\beta\right].
\]
Since \(|\nabla g_s(z)|\simeq r^{-s-1}\), the near-field radial integral satisfies
\begin{equation}\label{eq:critical-three-regime-radial-remainder}
\begin{aligned}
 &\int_0^{1/2}
 \left[\min\{\eps,r\ell(r)^\beta\}
       +\eps r\ell(r)^\beta\right]r^{\sigma_s-2}\,\dd r\\
 &\qquad\le
 \int_0^1\min\{\eps,r\ell(r)^\beta\}r^{\sigma_s-2}\,\dd r
 +\eps\int_0^1r^{\sigma_s-1}\ell(r)^\beta\,\dd r.
\end{aligned}
\end{equation}
Lemma~\ref{lem:critical-riesz-radial-integrals} therefore gives
\begin{equation}\label{eq:critical-three-regime-radial-scales}
 \int_0^{1/2}
 \left[\min\{\eps,r\ell(r)^\beta\}
       +\eps r\ell(r)^\beta\right]r^{\sigma_s-2}\,\dd r
 \le C_{s,q}
 \begin{cases}
  \eps\ell(\eps),&\sigma_s=1\;(s=d-1),\\
  \eps^{\sigma_s}\ell(\eps)^\beta,&0<\sigma_s<1\;(d-1<s<d).
 \end{cases}
\end{equation}
The second integral in \eqref{eq:critical-three-regime-radial-remainder} is \(O(\eps)\) and is absorbed by either line of \eqref{eq:critical-three-regime-radial-scales}.  If
\(d-2<s<d-1\), then \(s+1<d\) and \(\nabla g_s\in L^1_{\mathrm{loc}}\), so the ordinary integral representation applies.  The near-field/far-field estimate together
with \(\|w_\eps\|_\infty\le C\mathcal B_q(t)\eps\) gives an \(O(\eps)\)
error term.  These are exactly the three cases in \(e_{q,s}\).

\begin{remark}
\label{rem:critical-riesz-exponent-regimes}
The threshold $s=d-1$ is the point where the force loses local
integrability.  Below it, the $L^\infty$ mollification error alone yields the
factor $\eps$.  At the threshold the plateau
\(\min\{\eps,r\ell(r)^{\beta_q}\}=\eps\) over an intermediate range is
integrated against \(r^{-1}\dd r\), producing the extra logarithm
$\eps\ell(\eps)$.  Above the threshold, with $\sigma_s=d-s\in(0,1)$, the same
radial balance gives
$\eps^{\sigma_s}\ell(\eps)^{\beta_q}$.  The logarithmic factor at $s=d-1$ is the borderline contribution of the
radial integral.  For $q=\infty$, the bound
$C\eps^{\min\{1,d-s\}}\ell(\eps)$ provides a uniform form for the three
regimes in the Osgood argument.
\end{remark}

\begin{lemma}
\label{lem:finite-q-endpoint-dyadic-selector}
Let $1\le q<\infty$, put $\beta=\beta_q$, and fix $d-2<s<d$.
There are $z_{q,s}\in(0,e^{-3})$, $C_{q,s}\ge1$, and a Borel map
\[
 \eps_{q,s}:(0,z_{q,s}]\longrightarrow\{2^{-j}:j\ge4\}
\]
that is locally constant from the right as in
\eqref{eq:abstract-selector-right-stability} and, for every
$0<z\le z_{q,s}$,
\begin{align*}
 z\ell(\eps_{q,s}(z))^\beta
 +e_{q,s}(\eps_{q,s}(z))
 &\le C_{q,s}z\ell(z)^\beta,\\
 \eps_{q,s}(z)^2&\le C_{q,s}z.
\end{align*}
After reducing $z_{q,s}$ if necessary, \eqref{eq:abstract-optimization-statement} holds with
\[
 \omega=\omega_0=\omega_\beta,\qquad
 L(\eps)=\ell(\eps)^\beta,\qquad
 E_1(\eps)=e_{q,s}(\eps),\qquad
 E_2(\eps)=\eps^2.
\]
\end{lemma}

\begin{proof}
The proof is given in Appendix~\ref{app:riesz-technical-proofs}.
\end{proof}

\subsection{Commutator estimate under the assumed Besov regularity}

\subsubsection{Symmetrized form of the commutator and regularization}
For the Riesz kernel we retain the notation of
Section~\ref{sec:abstract-framework} and abbreviate only
$C_N[v]:=C_N^{g_s}(X_N,\rho;v)$ and use $r_N$ as defined in
\eqref{eq:unified-force-error}.

\begin{lemma}
\label{lem:riesz-even-mollification-kernel}
Let $K_s=\nabla g_s$, let $\psi\in C_c^\infty(\R^d)$ be even, nonnegative,
of unit mass, and supported in $B_1$, and set
$K_s^\kappa:=\nabla(g_s*\psi_\kappa)$.  Then, for $0<\kappa\le1$,
\begin{equation}\label{eq:critical-even-kernel-estimates}
 |K_s^\kappa(z)|\le C_{d,s,\psi}(|z|+\kappa)^{-s-1},\qquad
 |K_s^\kappa(z)-K_s(z)|\le C_{d,s,\psi}\kappa^2|z|^{-s-3}
 \quad (|z|\ge4\kappa).
\end{equation}
\end{lemma}
\begin{proof}
If $|z|\le2\kappa$, rescaling $g_s*\psi_\kappa$ gives
$|K_s^\kappa(z)|\le C\kappa^{-s-1}$.  If $|z|>2\kappa$, the support of
$\psi_\kappa$ stays away from the singularity and
$K_s^\kappa(z)=\int K_s(z-y)\psi_\kappa(y)\,\dd y$, with
$|z-y|\simeq|z|$; this proves the first bound.  For $|z|\ge4\kappa$,
apply Taylor's formula to $K_s(z-y)$ to second order for $|y|\le\kappa$.  Evenness of
$\psi$ gives $\int y\,\psi_\kappa(y)\,\dd y=0$, while
$\|\nabla^2K_s(\xi)\|\le C_{d,s}|\xi|^{-s-3}$ on
$|\xi-z|\le\kappa$.  Integrating the quadratic remainder and using
$\int|y|^2\psi_\kappa(y)\,\dd y=O(\kappa^2)$ gives the second estimate.
\end{proof}

\begin{proposition}
\label{prop:critical-riesz-regularization}
Let $N\ge2$ and $d-2<s<d$.  Let $\rho$ be a limiting solution with the Besov
regularity specified in Definition~\ref{def:critical-riesz-reference-solution},
let $X_N\in C^1([0,T];\Omega_N)$ be a collision-free Riesz particle
trajectory, and let
$w\in L^\infty(0,T;L^\infty\cap\Lambda_*)$.  Let $C_N^\kappa[w]$ denote the commutator with the even
regularized kernel $K_s^\kappa=\nabla(g_s*\psi_\kappa)$ and let
$C_N^\delta[w]$ use the same cutoff $|x-y|>\delta$ in all components.  Put
\[
 \sigma_s=d-s\in(0,2),\quad
 M_w=\operatorname*{ess\,sup}_{t<T}(\|w(t,\cdot)\|_\infty+[w(t,\cdot)]_{\Lambda_*}),
 \quad B_T=\|\mathcal B_\infty\|_{L^\infty(0,T)},
\]
\[
 r_{N,T}=\min_{t\le T}\min_{i\ne j}|x_i(t)-x_j(t)|>0.
\]
Then the following formula for $C_N[w]$, obtained by symmetrization, is absolutely convergent:
\begin{equation}\label{eq:critical-riesz-symmetric-difference-expansion}
\begin{aligned}
 C_N[w]={}&\frac1{N^2}\sum_{i\ne j}
 (w(x_i)-w(x_j))\cdot K_s(x_i-x_j)\\
 &-\frac2N\sum_{i=1}^N\int_{\R^d}
 (w(x_i)-w(x))\cdot K_s(x_i-x)\rho(x)\,\dd x\\
 &+\iint_{\R^d\times\R^d}
 (w(x)-w(y))\cdot K_s(x-y)\rho(x)\rho(y)\,\dd x\dd y,
\end{aligned}
\end{equation}
where the last two lines are absolutely convergent.  If
\begin{equation}\label{eq:critical-riesz-hard-cutoff-range}
 0<\delta<\min\{1/2,r_{N,T}\},
\end{equation}
then
\[
 \|C_N^\delta[w]-C_N[w]\|_{L^\infty(0,T)}
 \le C_{d,s}M_wB_T\delta^{\sigma_s}(1+|\log\delta|).
\]
For every
\[
 0<\kappa<\min\{1/8,r_{N,T}/4\},
\]
\[
 \|C_N^\kappa[w]-C_N[w]\|_{L^\infty(0,T)}
 \le C_{d,s,\psi}M_wB_T\kappa^{\sigma_s}(1+|\log\kappa|)
 +C_{d,s,\psi}M_w\kappa^2r_{N,T}^{-s-3}.
\]
Accordingly, the right-hand side of \eqref{eq:critical-riesz-symmetric-difference-expansion} is the common $L^1(0,T)$ limit of $C_N^\delta[w]$ as $\delta\downarrow0$ and of $C_N^\kappa[w]$ as $\kappa\downarrow0$.
The role of $r_{N,T}>0$ is confined to this fixed-$N$ singular limit; the
subsequent quantitative estimates are independent of the minimum separation.  The identification of the common limit with the
principal value representation obtained from centered cutoffs when $d-1\le s<d$
is proved in Proposition~\ref{prop:critical-riesz-pv-statement}\textup{(iii)}.
\end{proposition}

\begin{proof}
Expanding $\nu_N=\mu_N-\rho$ and using the oddness of $K_s$ gives
\eqref{eq:critical-riesz-symmetric-difference-expansion}.  The discrete--discrete
term is a finite sum because the configuration is collision-free.  For the
mixed and continuous terms, the only issue is the diagonal.  For $0<|z|<1$,
the Zygmund estimate gives
\begin{equation}\label{eq:critical-regularization-zygmund-majorant}
 |w(a)-w(a-z)|\le C M_w |z|\ell(|z|),
\end{equation}
and hence
\[
 |w(a)-w(a-z)|\,|\nabla g_s(z)|\rho(a-z)
 \le CM_wB_T|z|^{-s}\ell(|z|).
\]
The corresponding radial majorant is $r^{d-s-1}\ell(r)$ and is integrable
because $s<d$.  On $|z|\ge1$, both $w$ and the force are bounded, while
$\rho$ has unit mass.  This proves the asserted absolute convergence.  Since
$\sigma_s=d-s\in(0,2)$,
\[
 \int_0^R r^{d-s-1}\ell(r)\,\dd r
 \le C_{d,s}R^{\sigma_s}\ell(R),\qquad 0<R\le\frac12.
\]
For $\delta$ in the range
\eqref{eq:critical-riesz-hard-cutoff-range}, the discrete term is unchanged.
Applying the preceding radial estimate once to each mixed term and once after
integrating the outer variable in the continuous term yields
\[
 \|C_N^\delta[w]-C_N[w]\|_{L^\infty(0,T)}
 \le C_{d,s}M_wB_T
 \delta^{\sigma_s}\ell(\delta),
\]
which is the asserted hard cutoff bound.

We turn next to the even regularization.  Set
\[
 D_\kappa(z):=K_s^\kappa(z)-K_s(z).
\]
By Lemma~\ref{lem:riesz-even-mollification-kernel},
\[
 |K_s^\kappa(z)|\le C(|z|+\kappa)^{-s-1},\qquad
 |D_\kappa(z)|\le C\kappa^2|z|^{-s-3}
 \quad (|z|\ge4\kappa).
\]
Use expansion \eqref{eq:critical-riesz-symmetric-difference-expansion} for
both $K_s^\kappa$ and $K_s$.  Their difference separates one
discrete--discrete term, one mixed block with coefficient $-2/N$, and one
continuous--continuous block.  In the last two blocks the factor multiplying
$D_\kappa$ is always a difference of values of $w$, so the diagonal
cancellation in \eqref{eq:critical-regularization-zygmund-majorant} is retained.
We estimate these two blocks on the near, intermediate, and far regions.

In the near region $0<|z|\le4\kappa$, we combine
\eqref{eq:critical-regularization-zygmund-majorant} with the first estimate of
Lemma~\ref{lem:riesz-even-mollification-kernel} to obtain
\begin{align*}
 &\int_{|z|\le4\kappa}|w(a)-w(a-z)|
 \bigl(|K_s^\kappa(z)|+|K_s(z)|\bigr)\,\rho(a-z)\,\dd z\\
 &\qquad\le C M_wB_T
 \left[
 \kappa^{-s-1}\int_0^{4\kappa}r^d\ell(r)\,\dd r
 +\int_0^{4\kappa}r^{d-s-1}\ell(r)\,\dd r
 \right]\\
 &\qquad\le C_{d,s}M_wB_T
 \kappa^{\sigma_s}\ell(\kappa).
\end{align*}
The same bound holds for the continuous--continuous contribution after
integrating the outer density, since $\rho$ has unit mass.

In the intermediate annulus $4\kappa<|z|<1$, the second estimate in
Lemma~\ref{lem:riesz-even-mollification-kernel} applies and gives
\begin{align*}
 &\int_{4\kappa<|z|<1}|w(a)-w(a-z)|
 |D_\kappa(z)|\,\rho(a-z)\,\dd z\\
 &\qquad\le C M_wB_T\kappa^2
 \int_{4\kappa}^1 r^{d-s-3}\ell(r)\,\dd r\\
 &\qquad\le C_{d,s}M_wB_T
 \kappa^{\sigma_s}\ell(\kappa),
\end{align*}
where the last step uses $0<\sigma_s=d-s<2$.

In the region $|z|\ge1$, boundedness of $w$, unit mass of $\rho$, and
Lemma~\ref{lem:riesz-even-mollification-kernel} implies
\[
 \int_{|z|\ge1}|w(a)-w(a-z)|\,|D_\kappa(z)|\rho(a-z)\,\dd z
 \le C_{d,s}M_w\kappa^2
 \le C_{d,s}M_w\kappa^{\sigma_s}.
\]
Again the continuous--continuous term satisfies the same estimate after
integrating its outer variable.  Let $J_\kappa$ denote the sum of the mixed and continuous contributions to
$C_N^\kappa[w]-C_N[w]$.  Combining the three regions gives
\begin{equation}\label{eq:critical-even-continuum-error}
 |J_\kappa|
 \le C_{d,s,\psi}M_wB_T
 \kappa^{\sigma_s}(1+|\log\kappa|).
\end{equation}

It remains to estimate the discrete--discrete term.  If
$0<\kappa<\min\{1/8,r_{N,T}/4\}$, then every separation satisfies
$|x_i-x_j|\ge4\kappa$, so
\begin{align*}
 \frac1{N^2}\sum_{i\ne j}
 |w(x_i)-w(x_j)|\,|D_\kappa(x_i-x_j)|\le C_{d,s,\psi}M_w\kappa^2r_{N,T}^{-s-3}.
\end{align*}
Combining this estimate with \eqref{eq:critical-even-continuum-error} yields
the asserted even regularization bound and concludes the proof.
\end{proof}

\subsubsection{Completion of squares for \texorpdfstring{$d-1\le s<d$}{d-1 <= s < d}}
\begin{proposition}[Completion of squares with centered cutoffs]
\label{prop:critical-riesz-mollified-defect}
For $0<\eps\le e^{-2}$ put $w_\eps:=u-u_\eps$ and, in the notation of
Proposition~\ref{prop:completion-of-squares}, set
$r_N^\eps:=r_N[u_\eps]$.
Define the continuous term by
\[\mathcal J_{N,s}^{\eps}:=
 \begin{cases}
 2\displaystyle\int w_\eps(x)\cdot
 (K_N(x)-u(x))\rho(x)\,\dd x,
 &d-2<s<d-1,\\[2mm]
 2\displaystyle\lim_{\delta\downarrow0}\left[
 \frac1N\sum_{i=1}^N\int w_\eps(x)\rho(x)\cdot
 G_{s,\delta}(x-x_i)\,\dd x
 -\int w_\eps(x)\cdot u_\delta(x)\rho(x)\,\dd x\right],
 &d-1\le s<d,
 \end{cases}
\]
where $G_{s,\delta}=\mathbf 1_{\{|\cdot|>\delta\}}\nabla g_s$ and
$u_\delta=G_{s,\delta}*\rho$.  In the second case, the same centered cutoff is
used in both terms before the limit is taken.  Appendix~\ref{sec:principal-value-representation}
identifies the resulting limit with both the symmetrized formula and the
principal value representation.  Then, for almost every time,
\begin{equation}\label{eq:critical-riesz-exact-mollified-defect}
 -r_N-C_N[w_\eps]
 =-r_N^{\eps}
 +\frac1N\sum_i|w_\eps(x_i)|^2
 +\mathcal J_{N,s}^{\eps},
\end{equation}
and
\begin{equation}\label{eq:critical-riesz-defect-comparison}
 r_N\le2r_N^{\eps}
 +\frac2N\sum_i|w_\eps(x_i)|^2.
\end{equation}
\end{proposition}

\begin{proof}
For $d-2<s<d-1$, apply Proposition~\ref{prop:completion-of-squares} with $v=u_\eps$ and $w=w_\eps$.

Assume $d-1\le s<d$.  Since convolution with the smooth mollifier does not
increase the $L^\infty$ norm or the Zygmund seminorm,
\[
 \|w_\eps(t)\|_{L^\infty}+[w_\eps(t)]_{\Lambda_*}
 \le 2\bigl(\|u(t)\|_{L^\infty}+[u(t)]_{\Lambda_*}\bigr)
 \le C_{d,s}\mathcal B_\infty(t)
\]
for almost every time.  Thus $w_\eps$ satisfies the hypotheses of
Proposition~\ref{prop:critical-riesz-regularization} and of the
principal value representation below.  Put
\[
 G_{s,\delta}:=\mathbf 1_{\{|\cdot|>\delta\}}\nabla g_s,
 \quad K_i^\delta:=\frac1N\sum_{j\ne i}G_{s,\delta}(x_i-x_j),
 \quad u_\delta:=G_{s,\delta}*\rho,
\]
\[
 r_{N,\delta}:=\frac1N\sum_i|K_i^\delta-u_\delta(x_i)|^2.
\]
Let $C_{N,\delta}[w_\eps]$ be the commutator in which the same
$G_{s,\delta}$ is used in the discrete, mixed, and continuum pieces, and set
\begin{equation}\label{eq:critical-riesz-cutoff-continuous-remainder}
 \mathcal J_{N,s}^{\eps,\delta}
 :=2\int w_\eps(x)\cdot(K_N^\delta(x)-u_\delta(x))\rho(x)\,\dd x.
\end{equation}
For fixed $\delta>0$, $G_{s,\delta}$ is bounded and vanishes on $B_\delta$,
so every term is absolutely integrable.  The algebra of
Proposition~\ref{prop:completion-of-squares}, applied to the cutoff field,
gives
\begin{equation}\label{eq:critical-riesz-cutoff-completion-square}
\begin{aligned}
-r_{N,\delta}-C_{N,\delta}[w_\eps]
&=-\frac1N\sum_i|K_i^\delta-u_\delta(x_i)+w_\eps(x_i)|^2+\frac1N\sum_i|w_\eps(x_i)|^2
 +\mathcal J_{N,s}^{\eps,\delta}.
\end{aligned}
\end{equation}
We now pass to the limit term by term, except that the two singular terms in
$\mathcal J_{N,s}^{\eps,\delta}$ are kept together.  Once
$\delta<\frac12\min_{i\ne j}|x_i-x_j|$, one has $K_i^\delta=K_i$ for every
$i$.  Part~\textup{(i)} of
Proposition~\ref{prop:critical-riesz-pv-statement} gives
\[
 \|u_\delta-u\|_{L^\infty(\R^d)}\longrightarrow0
\]
for almost every fixed time.  Hence, at the finitely many particle locations,
\[
 r_{N,\delta}\longrightarrow r_N,
 \qquad
 K_i^\delta-u_\delta(x_i)+w_\eps(x_i)
 \longrightarrow K_i-u(x_i)+w_\eps(x_i)=K_i-u_\eps(x_i),
\]
and therefore
\[
 \frac1N\sum_i|K_i^\delta-u_\delta(x_i)+w_\eps(x_i)|^2
 \longrightarrow r_N^\eps.
\]
Moreover, Proposition~\ref{prop:critical-riesz-regularization} gives
\[
 C_{N,\delta}[w_\eps]\longrightarrow C_N[w_\eps].
\]
Finally, expanding \eqref{eq:critical-riesz-cutoff-continuous-remainder}
gives
\[
 \mathcal J_{N,s}^{\eps,\delta}
 =2\left[\frac1N\sum_{i=1}^N\int
 w_\eps(x)\rho(x)\cdot G_{s,\delta}(x-x_i)\,\dd x
 -\int w_\eps\cdot u_\delta\rho\right].
\]
Part~\textup{(ii)} of Proposition~\ref{prop:critical-riesz-pv-statement}
shows that the bracket has a finite limit as $\delta\downarrow0$; by the
definition in the statement this limit is $\frac12\mathcal J_{N,s}^\eps$.
Thus
$\mathcal J_{N,s}^{\eps,\delta}\to\mathcal J_{N,s}^\eps$.
Passing to the limit in
\eqref{eq:critical-riesz-cutoff-completion-square} proves
\eqref{eq:critical-riesz-exact-mollified-defect}.  Part~\textup{(iii)} of the
same proposition, together with
Proposition~\ref{prop:critical-riesz-regularization}, shows independently
that this centered cutoff representative agrees with the absolutely
convergent formula obtained by symmetrization and with the even regularization limit.
The elementary estimate
\[
 |K_i-u(x_i)|^2\le2|K_i-u_\eps(x_i)|^2+2|w_\eps(x_i)|^2
\]
gives \eqref{eq:critical-riesz-defect-comparison}.
\end{proof}

\begin{lemma}
\label{lem:critical-riesz-rough-local}
Let $1\le q\le\infty$, assume the corresponding Besov bound, and
use the function $\mathcal B_q(t)$ defined in \eqref{eq:critical-riesz-measurable-besov-bound}.  If $d-2<s<d-1$ and
$w_\eps=u-u_\eps$, then, for almost every time,
\[
\left|2\int w_\eps(x)\cdot
 (K_N(x)-u(x))\rho(x)\,\dd x\right|
 \le C_{d,s,q}\mathcal B_q(t)^2\eps.
\]
\end{lemma}

\begin{proof}
By Lemma~\ref{lem:near-far-locally-integrable-field} and the triangle inequality, the left-hand side is bounded by
$C_{d,s}\|w_\eps\|_\infty(1+\|\rho\|_\infty)$.  Estimate
\eqref{eq:critical-riesz-mollifier-estimates} gives
$\|w_\eps\|_\infty\le C_{d,s,q}\mathcal B_q(t)\eps$, while the density factor
is bounded by $\mathcal B_q(t)$.
\end{proof}

\subsubsection{Proof of the commutator estimate}
\begin{proposition}[Commutator estimate under $B^{s-d+2}_{\infty,q}$ regularity]
\label{prop:critical-riesz-commutator}
\begin{enumerate}
\item[(i)] \emph{Case $q=\infty$.} For every
$0<\eps\le e^{-2}$ and almost every $t\in(0,T)$,
\[\begin{aligned}
 -r_N-C_N[u]
 \le-r_N^{\eps}
 +C_{d,s}\mathcal B_\infty(t)\left(1+\log\frac1\eps\right)
 (F_N+\eta_N)+C_{d,s}\mathcal B_\infty(t)^2
 \eps^{\min\{1,\sigma_s\}}\left(1+\log\frac1\eps\right).
\end{aligned}
\]
Furthermore,
\[r_N\le2r_N^{\eps}+C_{d,s}\mathcal B_\infty(t)^2\eps^2.
\]
\item[(ii)] \emph{Case $1\le q<\infty$.} Assume \eqref{eq:critical-riesz-finite-q-bound} for some
$1\le q<\infty$.  Then, for every $0<\eps\le e^{-2}$ and almost every
$t\in(0,T)$,
\[\begin{aligned}
 -r_N-C_N[u]
 \le-r_N^\eps
 +C_{d,s,q}\mathcal B_q(t)\ell(\eps)^{\beta_q}
 (F_N+\eta_N)+C_{d,s,q}\mathcal B_q(t)^2
 e_{q,s}(\eps),
\end{aligned}
\]
and
\[r_N\le2r_N^\eps+C_{d,s,q}\mathcal B_q(t)^2\eps^2.
\]
\end{enumerate}
\end{proposition}

\begin{proof}
\emph{Step 1: the commutator with $u_\eps$.}
Set $u=u_\eps+w_\eps$.  By the sharp Lipschitz commutator estimate in
Lemma~\ref{lem:static-riesz-transfer},
\[
 |C_N[u_\eps]|
 \le
 \begin{cases}
 C\mathcal B_\infty(t)\ell(\eps)(F_N+\eta_N),&q=\infty,\\
 C\mathcal B_q(t)\ell(\eps)^{\beta_q}(F_N+\eta_N),&1\le q<\infty.
 \end{cases}
\]

\emph{Step 2: completion of squares for $w_\eps=u-u_\eps$.}
Proposition~\ref{prop:critical-riesz-mollified-defect} gives
\[
 -r_N-C_N[w_\eps]
 =-r_N^\eps+\|w_\eps\|_{L^2(\mu_N)}^2+\mathcal J_{N,s}^\eps.
\]
Also,
\[
 \|w_\eps\|_{L^2(\mu_N)}^2
 \le
 \begin{cases}
 C\mathcal B_\infty(t)^2\eps^2,&q=\infty,\\
 C\mathcal B_q(t)^2\eps^2,&1\le q<\infty.
 \end{cases}
\]

\emph{Step 3: estimate of $\mathcal J_{N,s}^{\eps}$.}
The Littlewood--Paley estimates imply, for almost every time,
\[
 \|w_\eps\|_{L^\infty}+[w_\eps]_{\Lambda_*}
 \le C_{d,s,q}\mathcal B_q(t),
\]
so when $d-1\le s<d$ $w_\eps$ satisfies the hypotheses of the
principal value statement.  If $d-2<s<d-1$, the force is locally integrable
and Lemma~\ref{lem:critical-riesz-rough-local} gives the first line
$e_{q,s}(\eps)=\eps$.  If $s=d-1$, Part~\textup{(iv)} of
Proposition~\ref{prop:critical-riesz-pv-statement} and
\eqref{eq:critical-three-regime-radial-scales} give the borderline bound
$\eps\ell(\eps)$.  If $d-1<s<d$, the same proposition gives
$\eps^{d-s}\ell(\eps)^{\beta_q}$.  Together with
$\|w_\eps\|_{L^2(\mu_N)}^2\le C\mathcal B_q(t)^2\eps^2$, which is smaller
than each of these three errors for $0<\eps<1$, we obtain
\[
 |\mathcal J_{N,s}^\eps|+\|w_\eps\|_{L^2(\mu_N)}^2
 \le
 \begin{cases}
 C\mathcal B_\infty(t)^2\eps^{\min\{1,\sigma_s\}}\ell(\eps),&q=\infty,\\
 C\mathcal B_q(t)^2e_{q,s}(\eps),&1\le q<\infty.
 \end{cases}
\]
For $q=\infty$, the first line follows from the three bounds in the three cases and $\min\{1,\sigma_s\}=\min\{1,d-s\}$.

\emph{Step 4: combine the contributions of $u_\eps$ and $w_\eps$.}
Together with $C_N[u]=C_N[u_\eps]+C_N[w_\eps]$, the preceding estimates give
the two commutator bounds in the statement.

\emph{Step 5: estimate $r_N$.}
The estimate
\[
 r_N\le2r_N^\eps+2\|u-u_\eps\|_\infty^2
\]
gives the corresponding estimates for $r_N$.
\end{proof}

\subsection{Proof of the Besov weak--strong estimate}
\label{sec:critical-extension-statements}
We combine the Besov field and commutator estimates from the
preceding subsections with the singular chain rule and the averaged comparison principle.  Throughout, $d-2<s<d$ and the limiting solution
satisfies Definition~\ref{def:critical-riesz-reference-solution}.

\subsubsection{Measurable representatives and the singular chain rule}
\label{subsec:critical-riesz-main-results}
We work with the class of limiting solutions and the finite-$N$ quantities introduced
in Section~\ref{sec:main-results}.  The exponent $\alpha=s-d+2$ is the
density regularity for which the Riesz interaction field reaches
the Zygmund endpoint $B^1_{\infty,\infty}$.

Proposition~\ref{prop:critical-riesz-field} gives
\[
 \rho\in B^{\alpha}_{\infty,\infty}
 \quad\Longrightarrow\quad
 u=\nabla g_s*\rho\in B^1_{\infty,\infty}=\Lambda_*.
\]
The Zygmund space $B^1_{\infty,\infty}$ is strictly larger than
$W^{1,\infty}$; see \cite{Bahouri}.  Throughout the weak--strong argument,
a limiting solution satisfying
Definition~\ref{def:critical-riesz-reference-solution} is fixed.

Proposition~\ref{prop:critical-riesz-borel-reference} extends both the
$L^\infty$ density bound and the assumed Besov bound to every time and provides
jointly Borel representatives of the density and the interaction field.
Proposition~\ref{prop:critical-riesz-field} gives a bounded Zygmund interaction
field.  More precisely, with the Borel function $\mathcal B_\infty$ from
\eqref{eq:critical-riesz-measurable-besov-bound},
\[
 \|u(t)\|_\infty\le C_{d,s}\mathcal B_\infty(t),\qquad
 |u_t(x)-u_t(y)|\le C_{d,s}\mathcal B_\infty(t)|x-y|
 \left(1+\log_+\frac1{|x-y|}\right)
\]
for every time after choosing the representatives above.  Since
$\mathcal B_\infty$ is bounded on $[0,T]$, the first estimate gives
$u\in L^2(0,T;L^\infty)$, while the coefficient $C_{d,s}\mathcal B_\infty(t)$ belongs to $L^1(0,T)$.  The modulus
$r\mapsto r(1+\log_+(1/r))$ is Osgood.  Hence, for the characteristic velocity
$-u$, Proposition~\ref{prop:unified-lagrangian-representation} yields the unique
Osgood flow $\Psi^t$ with
\begin{equation}\label{eq:critical-riesz-automatic-pushforward}
 \rho_t=(\Psi^t)_\#\rho_0\qquad(0\le t\le T).
\end{equation}
Moreover, Proposition~\ref{prop:critical-riesz-field} gives for the even
mollifications
\[
 \|u^\kappa-u\|_{L^2(0,T;L^\infty)}
 \le C_{d,s}T^{1/2}\kappa
 \|\mathcal B_\infty\|_{L^\infty(0,T)}\longrightarrow0.
\]
Thus the three inputs of Proposition~\ref{prop:riesz-singular-chain-rule}---the
Lagrangian pushforward representation, \eqref{eq:riesz-abstract-chain-rule-assumptions},
and \eqref{eq:riesz-mollified-field-convergence-assumption}---are all explicit
consequences of the assumptions in Definition~\ref{def:critical-riesz-reference-solution}.  For finite $q$ the sharper
$B^1_{\infty,q}$ bounds are used later in the commutator estimate, whereas this
chain rule step only needs the weaker $q=\infty$ modulus.  All estimates below are understood with these representatives; the final
Wasserstein estimates hold for every time by continuity in $\mathcal P_2$.

Proposition~\ref{prop:riesz-singular-chain-rule} then gives
\[\frac{\dd}{\dd t}F_N(X_N(t),\rho_t)
 =-2r_N(t)-C_N[u_t]
 \quad\text{for a.e. }t\in(0,T).
\]
The discrete, mixed, and continuous terms are regularized at the same scale
before passing to the limit.  In the range $d-1\le s<d$, the corresponding
principal value representation is the one in
Appendix~\ref{sec:principal-value-representation}.  Its near-diagonal
majorant is
$r^{d-s-1}(1+|\log r|)$, which is integrable for every $s<d$.

\subsubsection{Averaging and conclusion of the proof}

To conclude the proof of Theorem~\ref{thm:critical-riesz-propagation}, we combine the estimate at a fixed scale with the averaged comparison principle.  Let $\pi_N^0$ be a
symmetric optimal coupling of $\rho_N^0$ and $\rho_0^{\otimes N}$ and let
$\pi_N^t$ be its pushforward by the particle and limiting flows.  With all
trajectory quantities viewed as functions of the initial pair, set
\[
 \overline{\mathcal E}_N(t)
 :=\mathbb E_{\pi_N^0}\!\left[Q_N(t)+F_N(X_N(t),\rho_t)+\eta_N\right]
 +N^{s/d-1}.
\]
For a fixed dyadic $\eps$, Proposition~\ref{prop:endpoint-measurability-statement}
justifies differentiation under the fixed expectation and places the whole
countable scale family on a single set of full measure.  Since $\pi_N^t$ is
the pushforward of $\pi_N^0$, these expectations are the corresponding
integrals with respect to $\pi_N^t$.

For $1\le q<\infty$, the transport estimate, the singular chain rule,
Proposition~\ref{prop:critical-riesz-commutator}, and Jensen's inequality give
\begin{align}
 \overline{\mathcal E}_N'(t)+\overline r_N^\eps(t)
 &\le C_{d,s,q}\bigl(\mathcal B_q(t)+\mathcal B_q(t)^2\bigr)\Bigl[
 \omega_{\beta_q}(\overline{\mathcal E}_N(t))
 +\ell(\eps)^{\beta_q}\overline{\mathcal E}_N(t)
 +e_{q,s}(\eps)\Bigr],
 \label{eq:critical-averaged-fixed-scale-finite-q}\\
 \overline r_N(t)&\le2\overline r_N^\eps(t)
 +C_{d,s,q}\bigl(\mathcal B_q(t)+\mathcal B_q(t)^2\bigr)\eps^2.
 \label{eq:critical-averaged-force-finite-q}
\end{align}
Here $\overline r_N^\eps=\mathbb E_{\pi_N^0}r_N^\eps$ and
$\overline r_N=\mathbb E_{\pi_N^0}r_N$.  Apply
Corollary~\ref{cor:averaged-optimized-comparison} with
\[
 \omega=\omega_0=\omega_{\beta_q},\qquad
 L(\eps)=\ell(\eps)^{\beta_q},\qquad
 E_1(\eps)=e_{q,s}(\eps),\qquad
 E_2(\eps)=\eps^2,
\]
and use Lemma~\ref{lem:finite-q-endpoint-dyadic-selector}.  The time coefficient
$C_{d,s,q}(\mathcal B_q+\mathcal B_q^2)$ has finite integral on $[0,T]$.
Lemma~\ref{lem:finite-q-bihari-comparison} gives the stated Bihari bound for
small initial error, while Proposition~\ref{prop:abstract-fixed-scale-comparison}
gives $C_{T,d,s,q,\rho}(1+\overline{\mathcal E}_N(0))$ outside that regime.
Since the threshold is fixed and positive, the two estimates combine, after
enlarging the constant, into the following bound for every initial error:
\begin{equation}\label{eq:critical-averaged-final-finite-q}
 \sup_{t\le T}\overline{\mathcal E}_N(t)
 +\int_0^T\overline r_N(t)\,\dd t
 \le C_{T,d,s,q,\rho}\,\overline{\mathcal E}_N(0)
 \exp\!\left(C_{T,d,s,q,\rho}
 \bigl(1+\log_+(1/\overline{\mathcal E}_N(0))\bigr)^{1-1/q}\right).
\end{equation}

For $q=\infty$ the same argument gives
\begin{align}
 \overline{\mathcal E}_N'(t)+\overline r_N^\eps(t)
 &\le C_{d,s}\bigl(\mathcal B_\infty(t)+\mathcal B_\infty(t)^2\bigr)\Bigl[
 \omega_*(\overline{\mathcal E}_N(t))
 +\ell(\eps)\overline{\mathcal E}_N(t)
 +\eps^{\min\{1,\sigma_s\}}\ell(\eps)\Bigr],
 \label{eq:critical-averaged-fixed-scale-infinity}\\
 \overline r_N(t)&\le2\overline r_N^\eps(t)
 +C_{d,s}\bigl(\mathcal B_\infty(t)+\mathcal B_\infty(t)^2\bigr)\eps^2.
\end{align}
Using Corollary~\ref{cor:averaged-optimized-comparison} with
$\omega=\omega_0=\omega_*$, Lemmas~\ref{lem:measurable-dyadic-scale} and~\ref{lem:logarithmic-osgood-comparison}, and using $\mathcal B_\infty\ge1$, so that
$\mathcal B_\infty+\mathcal B_\infty^2\le2\mathcal B_\infty^2$, and
choosing the constant in $\gamma_T$ large enough to dominate the optimization
constant, we obtain
\begin{equation}\label{eq:critical-averaged-final-infinity}
 \sup_{t\le T}\overline{\mathcal E}_N(t)
 +\int_0^T\overline r_N(t)\,\dd t
 \le C_{T,d,s,\rho}\bigl(\overline{\mathcal E}_N(0)^{\gamma_T}+\overline{\mathcal E}_N(0)\bigr).
\end{equation}

For the optimal initial coupling, $\overline{\mathcal E}_N(0)=a_N$.  Moreover
\[
 \frac1N W_2^2(\rho_N(t),\rho_t^{\otimes N})
 \le\mathbb E_{\pi_N^0}Q_N(t),
\]
and $F_N$ and $r_N$ depend only on the particle configuration.  Thus
\eqref{eq:critical-averaged-final-infinity} gives
\eqref{eq:critical-riesz-main-estimate}, while
\eqref{eq:critical-averaged-final-finite-q} gives
\eqref{eq:critical-riesz-finite-q-main-estimate}.

For $\rho_N^0=\rho_0^{\otimes N}$, Lemma~\ref{lem:unified-product-initial-data} gives
$a_N\le C_{d,s,\rho}N^{s/d-1}$.  The case $q=\infty$ therefore yields
$N^{-(1-s/d)\gamma_T}$.  If $1<q<\infty$, substitution into
\eqref{eq:critical-riesz-finite-q-main-estimate} gives
\[
 C N^{s/d-1}\exp\!\left(C(\log N)^{1-1/q}\right),
\]
while for $q=1$ the finite-$q$ estimate is linear and gives
$C N^{s/d-1}$.  Enlarging the constant to absorb the finitely many small values
of $N$ proves all three lines of \eqref{eq:critical-product-rates}.

The same estimates also prove uniqueness in the class of Definition~\ref{def:critical-riesz-reference-solution}.  For uniqueness, apply the $q=\infty$ part of
Theorem~\ref{thm:critical-riesz-propagation} under the regularity in
Definition~\ref{def:critical-riesz-reference-solution}.  Evolve the single tensorized
law $\rho_0^{\otimes N}$ by the Riesz particle flow and denote its
one-particle marginal by $\rho_{N,1}(t)$.  Comparing this same particle law
independently with two reference solutions $\rho$ and
$\widetilde\rho$ having the same initial density gives, with possibly
different constants and Osgood exponents,
\[
 \sup_{t\le T}W_2^2(\rho_{N,1}(t),\rho_t)\longrightarrow0,
 \qquad
 \sup_{t\le T}W_2^2(\rho_{N,1}(t),\widetilde\rho_t)\longrightarrow0.
\]
The triangle inequality therefore implies
$W_2(\rho_t,\widetilde\rho_t)=0$ for every $t\in[0,T]$.
This proves Corollary~\ref{cor:critical-riesz-uniqueness}.

\subsection{Smooth limiting solutions in the super-Coulomb Riesz range}
The Riesz theorem under the stated Besov regularity is conditional on the existence of a limiting solution with the stated regularity.  The following standard observation gives a sufficient condition for a smooth solution to satisfy these assumptions.

\begin{proposition}
\label{prop:riesz-Sobolev-admissibility}
Let $d\ge2$, $d-2<s<d$, $m>d/2+3$, and let $\rho$ be a nonnegative
distributional solution of \eqref{eq:riesz-mean-field-main} on $[0,T]$ such that
\[
 \rho\in C([0,T];L^1(\R^d)\cap H^m(\R^d)),\qquad
 \rho_0\in\mathcal P_2(\R^d),\qquad \int\rho_t=1.
\]
Then, for every $1\le q\le\infty$,
\[
 \rho\in L^\infty(0,T;B^{s-d+2}_{\infty,q}),
\]
and $\rho$ satisfies Definition~\ref{def:critical-riesz-reference-solution}
and the finite-$q$ hypothesis \eqref{eq:critical-riesz-finite-q-bound} when
$q<\infty$.  In particular, the corresponding field is bounded and has the
regularity stated in Proposition~\ref{prop:critical-riesz-field}.
\end{proposition}

\begin{proof}
Set $\alpha=s-d+2\in(0,2)$.  Sobolev--Besov embedding gives
\[
 H^m=B^m_{2,2}\hookrightarrow B^{m-d/2}_{\infty,2}
 \hookrightarrow B^\alpha_{\infty,q},\qquad 1\le q\le\infty,
\]
because $m-d/2>3>\alpha$.  The multiplier of
$u=\nabla g_s*\rho$ is $c_{d,s}i\xi|\xi|^{s-d}$; the same Littlewood--Paley
estimate used in Proposition~\ref{prop:critical-riesz-field}, together with the
Sobolev surplus, gives $u\in L^\infty(0,T;W^{1,\infty})$.  Hence the second
moment propagates on $[0,T]$, and the asserted $\mathcal P_2$ continuity follows
from the Lipschitz characteristic representation.  This verifies the assumptions of Definition~\ref{def:critical-riesz-reference-solution}.  Local classical solutions furnished by the smooth theory, for
example \cite[Theorem~1.1]{ChoiJeong} in this repulsive range, provide
concrete instances on their interval of existence.
\end{proof}

\section{Initial data and Wasserstein control}
\label{subsec:prepared-initial-laws}
\subsection{A counterexample on \texorpdfstring{$\mathbb R^d$}{R d}}
The following example shows that convergence of the modulated energy and Kac chaos do not imply convergence in the normalized squared Wasserstein distance on $\R^d$.

\begin{proposition}[Counterexample on $\mathbb R^d$]
\label{prop:whole-space-modulated energy-obstruction}
Let $d\ge2$, $0<s<d$, and let $\rho\in C_c^\infty(\R^d)$ be a probability density with $\operatorname{supp}\rho\subset B(0,R_0)$.  Choose $A_{d,s}>0$ so that the lower bound in Lemma~\ref{lem:static-riesz-transfer} holds and set
\[
 \eta_N(\rho):=A_{d,s}\|\rho\|_{L^\infty}^{s/d}N^{s/d-1}.
\]
Set $R_N:=\sqrt N$, $z_N:=R_Ne_1$, and
\begin{equation}\label{eq:outlier-law-definition}
 \rho_N^\star:=\frac1N\sum_{\ell=1}^N
 \rho^{\otimes(\ell-1)}\otimes\delta_{z_N}\otimes\rho^{\otimes(N-\ell)}.
\end{equation}
Then $\rho_N^\star\in\mathcal P_2((\R^d)^N)$ is symmetric and
$\rho_N^\star(\Delta_N)=0$.  With
\[
 E_s(\rho):=\iint g_s(x-y)\rho(x)\rho(y)\,\dd x\dd y,
\]
one has
\begin{equation}\label{eq:outlier-modulated energy-expectation}
 \int F_N(X_N,\rho)\,\dd\rho_N^\star(X_N)
 =\left(-\frac1N+\frac2{N^2}\right)E_s(\rho)
 -\frac2{N^2}(g_s*\rho)(z_N).
\end{equation}
In particular, $(g_s*\rho)(z_N)=O(R_N^{-s})$ and
\[
 \int\bigl(F_N(X_N,\rho)+\eta_N(\rho)\bigr)\,\dd\rho_N^\star(X_N)
 \longrightarrow0.
\]
On the other hand,
\begin{equation}\label{eq:outlier-joint-W2-lower-bound}
 \frac1N W_2^2\bigl(\rho_N^\star,\rho^{\otimes N}\bigr)
 \ge \frac{(R_N-R_0)^2}{N}\longrightarrow1,
\end{equation}
and every configuration in the support of $\rho_N^\star$ satisfies
\begin{equation}\label{eq:outlier-empirical-W2-lower-bound}
 W_2^2\bigl(\mu_N(X_N),\rho\bigr)
 \ge \frac{(R_N-R_0)^2}{N}\longrightarrow1.
\end{equation}
The normalized second moments are also uniformly bounded:
\begin{equation}\label{eq:outlier-uniform-normalized-second-moment}
 \sup_N\int\frac1N\sum_{i=1}^N|x_i|^2\,\dd\rho_N^\star(X_N)<\infty,
\end{equation}
and, for every fixed $k$ and $N\ge k$,
\begin{equation}\label{eq:outlier-fixed-marginal-formula}
 (\rho_N^\star)_{:k}
 =\left(1-\frac{k}{N}\right)\rho^{\otimes k}
 +\frac1N\sum_{\ell=1}^k
 \rho^{\otimes(\ell-1)}\otimes\delta_{z_N}\otimes
 \rho^{\otimes(k-\ell)}.
\end{equation}
Consequently,
\begin{equation}\label{eq:outlier-kac-chaos}
 (\rho_N^\star)_{:k}\rightharpoonup\rho^{\otimes k}
 \qquad\text{for every fixed }k.
\end{equation}
There also exist collision-free configurations $X_N^\star$ such that
\begin{equation}\label{eq:outlier-deterministic-modulated energy-vs-W2}
 F_N(X_N^\star,\rho)+\eta_N(\rho)\longrightarrow0,
 \qquad
 \liminf_{N\to\infty}W_2^2\bigl(\mu_N(X_N^\star),\rho\bigr)\ge1.
\end{equation}
\end{proposition}

\begin{proof}
\emph{Energy calculation.}  Symmetry follows from \eqref{eq:outlier-law-definition}.
The collision set has zero mass because the $N-1$ random coordinates have smooth
densities and the fixed point $z_N$ is hit with probability zero.  For a fixed
outlier label, there are $(N-1)(N-2)$ ordered pairs of two $\rho$-distributed
particles and $2(N-1)$ ordered pairs involving the outlier.  The normalization in
\eqref{eq:unified-FN} gives the expectation
\[
 \frac{(N-1)(N-2)}{N^2}E_s(\rho)
 +\frac{2(N-1)}{N^2}(g_s*\rho)(z_N).
\]
The mixed term involving one particle and the continuum has expectation
\[
 -\frac{2(N-1)}N E_s(\rho)-\frac2N(g_s*\rho)(z_N),
\]
and adding the continuum energy $E_s(\rho)$ gives \eqref{eq:outlier-modulated energy-expectation}.  Since $z_N$ leaves every fixed compact set and $g_s(x)=s^{-1}|x|^{-s}$, one has $(g_s*\rho)(z_N)=O(R_N^{-s})$.

\smallskip
\noindent\emph{Quadratic transport and moments.}  For every $X_N$ in the support of
$\rho_N^\star$, one coordinate, say $x_\ell$, equals $z_N$.  Every $Y_N$ in
the support of $\rho^{\otimes N}$ satisfies $|y_\ell|\le R_0$, so
\[
 \sum_{i=1}^N|x_i-y_i|^2\ge |z_N-y_\ell|^2\ge(R_N-R_0)^2.
\]
Taking the infimum over all couplings yields \eqref{eq:outlier-joint-W2-lower-bound}.
The empirical measure $\mu_N(X_N)$ has mass $1/N$ at $z_N$, while $\rho$ is
supported in $B(0,R_0)$.  Any coupling between them must transport that mass by
at least $R_N-R_0$, which proves \eqref{eq:outlier-empirical-W2-lower-bound}.
Also,
\[
 \int\frac1N\sum_i|x_i|^2\,\dd\rho_N^\star
 =\frac{N-1}{N}\int|x|^2\rho(x)\,\dd x+\frac{R_N^2}{N},
\]
which proves \eqref{eq:outlier-uniform-normalized-second-moment}.  \smallskip
\noindent\emph{Marginals of fixed order.}  Formula \eqref{eq:outlier-fixed-marginal-formula} follows by separating the event that the uniformly random outlier label lies among the first $k$ coordinates from its complement.  Since the exceptional part has total mass $k/N$, testing against any bounded continuous function gives \eqref{eq:outlier-kac-chaos}.

\smallskip
\noindent\emph{Deterministic configurations.}  The finite-$N$ lower bound gives
$F_N+\eta_N(\rho)\ge0$, and \eqref{eq:outlier-modulated energy-expectation}
shows that its expectation under $\rho_N^\star$ tends to zero.  Markov's
inequality therefore provides collision-free configurations $X_N^\star$ such
that $F_N(X_N^\star,\rho)+\eta_N(\rho)\to0$.  Estimate
\eqref{eq:outlier-empirical-W2-lower-bound} then gives
\eqref{eq:outlier-deterministic-modulated energy-vs-W2}.
\end{proof}

\begin{remark}
\label{rem:whole-space-vs-compact-W2}
The same outlier construction works for the planar logarithmic kernel $g(x)=-(2\pi)^{-1}\log|x|$.  In that case $(g*\rho)(z_N)=-(2\pi)^{-1}\log R_N+O(1)$, so the extra outlier contribution in \eqref{eq:outlier-modulated energy-expectation} is $O((\log N)/N^2)$; after adding the logarithmic finite-$N$ error term, the expectation of $F_N+\eta_{N,2}$ still tends to zero, while \eqref{eq:outlier-joint-W2-lower-bound} is unchanged.

The obstruction is specific to the noncompact setting.  On a compact metric state space $K$, narrow convergence and $W_2$ convergence are equivalent because the quadratic transportation cost is bounded; see \cite{AmbrosioAGS,Villani2009}.  Accordingly, whenever a periodic modulated energy estimate implies narrow convergence of empirical measures, it also implies qualitative empirical $W_2$ convergence.  For symmetric particle laws, the qualitative relations between empirical measure convergence, convergence of marginals of fixed order, and Wasserstein formulations of chaos are discussed in \cite{HaurayMischler2014}.  These qualitative equivalences do not by themselves give a quantitative estimate for the normalized squared Wasserstein distance between the $N$-particle law and $\rho^{\otimes N}$.
\end{remark}

\subsection{Correlated initial laws}
The two principal theorems also apply to symmetric correlated initial laws
whose normalized quadratic transportation cost and modulated energy satisfy the
stated quantitative initial assumptions.  The counterexample on $\mathbb R^d$ above shows why these two quantities cannot in general be replaced by the modulated energy alone.

\section{A consistency estimate for the empirical continuity equation}\label{sec:dynamical-consequences}
We record the following direct consequence of the estimate for $\overline r_N$.

\begin{proposition}
\label{prop:force-error-consistency}
Let $\rho_N(t)$ be the law obtained by pushing forward $\rho_N^0$ under the
particle flow, and write $X_N(t)=(x_1(t),\ldots,x_N(t))$.  Then
\begin{equation}\label{eq:force-error-velocity}
 \int_0^T \overline r_N(t)\,\dd t
 =\int\frac1N\sum_{i=1}^N\int_0^T
 |\dot x_i(t)+u_t(x_i(t))|^2\,\dd t\,\dd\rho_N^0(X_N).
\end{equation}
For $\phi\in C_c^1([0,T]\times\mathbb R^d)$, define
\[
 \mathcal R_N(\phi;X_N):=\langle\mu_N(T),\phi_T\rangle
 -\langle\mu_N(0),\phi_0\rangle
 -\int_0^T\langle\mu_N(t),\partial_t\phi_t-u_t\cdot\nabla\phi_t\rangle\,\dd t.
\]
Then
\[
 \int|\mathcal R_N(\phi;X_N)|\,\dd\rho_N^0(X_N)
 \le \|\nabla\phi\|_{L_t^2L_x^\infty}
 \left(\int_0^T\overline r_N(t)\,\dd t\right)^{1/2}.
\]
\end{proposition}

\begin{proof}
Since $\dot x_i=-K_i(X_N)$, the first identity is the definition of
$\overline r_N$.  Differentiating $\phi_t(x_i(t))$, summing in $i$, and
integrating in time gives
\[
 \mathcal R_N(\phi;X_N)
 =-\int_0^T\frac1N\sum_i
 (K_i-u_t(x_i))\cdot\nabla\phi_t(x_i)\,\dd t.
\]
Cauchy--Schwarz in time and the particle index proves the estimate.
\end{proof}

\section{Conclusion}\label{sec:conclusion}
The argument developed here combines the quadratic transportation cost with
the dissipative modulated energy identity.
The remaining negative term given by the mean-square force error enters an exact completion of squares,
which controls the non-Lipschitz remainder after mollification and leaves the
sharp first-order commutator estimate to act on a Lipschitz comparison field.

For Coulomb flows, the sharp $L^\infty$ decay supplies a Riccati density
envelope and the scale $R(t)=m(t)^{-1/d}$.  The corresponding weighted transport
estimate yields quantitative comparison with the global bounded-density
solution on every prescribed finite interval.  In addition to the Wasserstein distance on the full
$N$-particle law and the modulated energy, the estimate controls the time
integral of the mean-square force error and therefore the residual of the
empirical continuity equation.

For super-Coulomb Riesz interactions, the same mechanism gives weak--strong
stability for reference solutions in $B^{s-d+2}_{\infty,q}$, with Gronwall,
Bihari, and Osgood comparisons according to $q$.  The particle approximation
also yields uniqueness in the corresponding Besov class.  The three Riesz
singularity regimes are reflected in the mollification errors, including the
borderline logarithmic correction at $s=d-1$.

The analysis suggests further questions concerning sharper Osgood rates for
bounded-density Coulomb flows, long-time estimates on $\mathbb R^d$, and the
use of analogous negative terms involving the force error in other singular mean-field systems.
\appendix

\section{Technical proofs for the Riesz extension}\label{app:riesz-technical-proofs}
The elementary radial estimates and the arguments for choosing the dyadic scale
used in the Riesz extension are collected here so that the main proof can focus
on the commutator estimates and the averaging argument.

\begin{proof}[Proof of Lemma~\ref{lem:critical-riesz-radial-integrals}]
Recall that $\ell(t)=1+\log(1/t)$ for $0<t\le1$.

For the first estimate, write $t=re^{-y}$.  Then
$\dd t=-re^{-y}\dd y=-t\dd y$ and
$\ell(re^{-y})=\ell(r)+y$, so
\begin{align*}
 \int_0^r t^{\gamma-1}\ell(t)^\beta\,\dd t
 &=r^\gamma\int_0^\infty
 e^{-\gamma y}(\ell(r)+y)^\beta\,\dd y.
\end{align*}
Since $0\le\beta\le1$,
\[
 (a+b)^\beta\le a^\beta+b^\beta,
 \qquad a,b\ge0.
\]
As $\ell(r)\ge1$, we have
\begin{align*}
 \int_0^r t^{\gamma-1}\ell(t)^\beta\,\dd t
 &\le r^\gamma\left[
 \ell(r)^\beta\int_0^\infty e^{-\gamma y}\,\dd y
 +\int_0^\infty e^{-\gamma y}y^\beta\,\dd y\right]\\
 &=r^\gamma\left[
 \frac1\gamma\ell(r)^\beta
 +\gamma^{-1-\beta}\Gamma(1+\beta)\right]\\
 &\le C_{\gamma,\beta}r^\gamma\ell(r)^\beta.
\end{align*}

For \eqref{eq:critical-riesz-radial-mollifier}, first assume $0<\gamma<2$.
Set $t=4\eps e^y$ and
$L=\log(1/(4\eps))$.  Then
\begin{align*}
 \eps^2\int_{4\eps}^1t^{\gamma-3}\ell(t)^\beta\,\dd t
 &=\eps^2(4\eps)^{\gamma-2}
 \int_0^L e^{-(2-\gamma)y}
 \bigl(\ell(4\eps)-y\bigr)^\beta\,\dd y\\
 &\quad\le \eps^2(4\eps)^{\gamma-2}\ell(4\eps)^\beta
 \int_0^\infty e^{-(2-\gamma)y}\,\dd y\\
 &\quad\le C_\gamma\eps^\gamma\ell(\eps)^\beta.
\end{align*}
If $\gamma=2$, use $y=\ell(t)$, so $\dd y=-\dd t/t$:
\begin{align*}
 \eps^2\int_{4\eps}^1t^{-1}\ell(t)^\beta\,\dd t
 =\eps^2\int_1^{\ell(4\eps)}y^\beta\,\dd y
 =\frac{\eps^2}{\beta+1}
 \bigl(\ell(4\eps)^{\beta+1}-1\bigr)
\le C_\beta\eps^2\ell(\eps)^{\beta+1}.
\end{align*}
This proves \eqref{eq:critical-riesz-radial-mollifier}.

For \eqref{eq:critical-riesz-radial-rough-bound}, split the integral at
$t=\eps$:
\begin{align*}
 I
 :=\int_0^1\min\{\eps,t\ell(t)^\beta\}t^{\sigma-2}\,\dd t\le \int_0^\eps t^{\sigma-1}\ell(t)^\beta\,\dd t
 +\eps\int_\eps^1t^{\sigma-2}\,\dd t.
\end{align*}
The first estimate already proved with $\gamma=\sigma$ yields
\[
 \int_0^\eps t^{\sigma-1}\ell(t)^\beta\,\dd t
 \le C_{\sigma,\beta}\eps^\sigma\ell(\eps)^\beta.
\]
If $0<\sigma<1$, then
\begin{align*}
 \eps\int_\eps^1t^{\sigma-2}\,\dd t
 =\frac{\eps}{1-\sigma}
 \bigl(\eps^{\sigma-1}-1\bigr)\le \frac1{1-\sigma}\eps^\sigma
 \le C_{\sigma,\beta}\eps^\sigma\ell(\eps)^\beta.
\end{align*}
If $\sigma=1$, then
\[
 \eps\int_\eps^1\frac{\dd t}{t}
 =\eps\log\frac1\eps
 \le\eps\ell(\eps).
\]
This proves \eqref{eq:critical-riesz-radial-rough-bound}.

For the last estimate, set $t=e^{-y}$.  Since
$\ell(e^{-y})=1+y$,
\begin{align*}
 \int_0^1t^{\sigma-1}\ell(t)^\beta\,\dd t
 =\int_0^\infty e^{-\sigma y}(1+y)^\beta\,\dd y\le \int_0^\infty e^{-\sigma y}(1+y)\,\dd y=\frac1\sigma+\frac1{\sigma^2}
 \le C_{\sigma,\beta},
\end{align*}
where we used $0\le\beta\le1$.  The proof is complete.
\end{proof}

\begin{proof}[Proof of Lemma~\ref{lem:finite-q-endpoint-dyadic-selector}]
Write \(\beta=\beta_q\in[0,1)\).  Define the continuous scale
\[\widetilde\eps_{q,s}(z):=
 \begin{cases}
 z\ell(z)^\beta,&d-2<s<d-1,\\
 z\ell(z)^{\beta-1},&s=d-1,\\
 z^{1/(d-s)},&d-1<s<d.
 \end{cases}
\]
For sufficiently small $z$, each of these functions is continuous and
strictly increasing.  Indeed,
\[
 \frac{\dd}{\dd z}\bigl(z\ell(z)^a\bigr)
 =\ell(z)^{a-1}\bigl(\ell(z)-a\bigr),
\]
which is positive for the exponents $a=\beta$ and $a=\beta-1$ once
$z$ is small; the third case is immediate.  Reduce $z_{q,s}$ so that
$\widetilde\eps_{q,s}(z)\le2^{-5}$ and
$\widetilde\eps_{q,s}(z)^2\le z$ on $(0,z_{q,s}]$.

Let $\eps_{q,s}(z)$ be the largest dyadic number $2^{-j}$, $j\ge4$, not
exceeding $\widetilde\eps_{q,s}(z)$.  Then
\[
 \frac12\widetilde\eps_{q,s}(z)<\eps_{q,s}(z)
 \le \widetilde\eps_{q,s}(z),
\]
and the monotonicity just proved shows that the dyadic rounding is a Borel
step function which is locally constant from the right.  Hence
\eqref{eq:abstract-selector-right-stability} holds.  The dyadic comparability
also gives, in all three cases,
\begin{equation}\label{eq:finite-q-selector-log-comparability}
 \ell(\eps_{q,s}(z))\le C_{q,s}\ell(z).
\end{equation}

We now check the three terms that must be absorbed by
$\omega_\beta(z)=z\ell(z)^\beta$ for $z\le z_{q,s}<e^{-3}$.  First, the
coefficient in the commutator estimate for the mollified field satisfies, by
\eqref{eq:finite-q-selector-log-comparability},
\[
 z\ell(\eps_{q,s}(z))^\beta
 \le C_{q,s}z\ell(z)^\beta.
\]
For the error term $e_{q,s}$, if $d-2<s<d-1$, then
\[
 e_{q,s}(\eps)=\eps
 \le z\ell(z)^\beta.
\]
At $s=d-1$,
\[
 e_{q,d-1}(\eps)
 =\eps\ell(\eps)
 \le C z\ell(z)^{\beta-1}\ell(z)
 =Cz\ell(z)^\beta.
\]
If $d-1<s<d$, recall $\sigma_s=d-s\in(0,1)$.  Then
\[
 e_{q,s}(\eps)
 =\eps^{\sigma_s}\ell(\eps)^\beta
 \le C z\ell(z)^\beta.
\]
Finally,
\[
 \eps_{q,s}(z)^2\le z\le z\ell(z)^\beta=\omega_\beta(z).
\]
Since the transport term itself is $\omega_\beta(z)$, these four estimates
prove \eqref{eq:abstract-optimization-statement} with
$\omega=\omega_0=\omega_\beta$, after increasing the constant.  The concavity
and Osgood properties required there are exactly those of
Lemma~\ref{lem:concave-log-modulus}.
\end{proof}

\section{Principal values for \texorpdfstring{$d-1\le s<d$}{d-1 <= s < d}}\label{sec:principal-value-representation}
In the range $d-1\le s<d$, the commutator is defined by the absolutely
convergent formula obtained by symmetrization in Section~\ref{sec:critical-riesz-proof}.
The centered principal values and the even regularizations used in the
chain rule argument converge to the same quantity.

\subsection*{Regularization of the commutator}
Let
\[
 G_{s,\delta}=\mathbf 1_{\{|z|>\delta\}}\nabla g_s.
\]
We compare the formula obtained by symmetrization with the expression obtained by
using the same centered cutoff in the discrete, mixed, and continuous terms, and
with an even regularization of the kernel.  At fixed regularization all integrals
are absolutely convergent.  Passing to the limit gives the principal value formulas below.  Throughout this appendix, $C_N[u]$ denotes the absolutely convergent formula obtained by symmetrization.

\begin{proposition}[Principal value representation of the commutator]
\label{prop:critical-riesz-pv-statement}
Assume $d-1\le s<d$.  Then the following identities and estimates hold.
\begin{enumerate}
\item[(i)] \emph{Radial principal values.} Assume $d-1\le s<d$ and let $\rho$ be a limiting solution satisfying Definition~\ref{def:critical-riesz-reference-solution}.  Put $\sigma_s=d-s\in(0,1]$ and
$u_\delta:=G_{s,\delta}*\rho$, with $G_{s,\delta}$ defined in
\eqref{eq:critical-riesz-truncated-kernel}.
Then, for almost every $t\in(0,T)$ and every $a\in\R^d$, the continuous
Zygmund representative of $u=\nabla g_s*\rho$ is
\begin{equation}\label{eq:critical-riesz-pv-density}
\begin{aligned}
 u(t,a)
 &=\operatorname{p.v.}\!\int_{\R^d}\nabla g_s(a-x)\rho_t(x)\,\dd x\\
 &=\int_{|z|<1}\nabla g_s(-z)
   [\rho_t(a+z)-\rho_t(a)]\,\dd z
 +\int_{|z|\ge1}\nabla g_s(-z)\rho_t(a+z)\,\dd z.
\end{aligned}
\end{equation}
The truncated fields satisfy the uniform estimate
\begin{equation}\label{eq:critical-riesz-pv-field-uniform-rate}
 \operatorname*{ess\,sup}_{0<t<T}\sup_{a\in\R^d}
 |u_\delta(t,a)-u(t,a)|
 \le C_{d,s}\|\mathcal B_\infty\|_{L^\infty(0,T)}
 \delta^{\sigma_s}(1+|\log\delta|),
 \qquad 0<\delta\le\tfrac12.
\end{equation}
The radial principal value agrees with the continuous representative of
the distributional Riesz field, uniformly in space for almost every time.
\item[(ii)] \emph{Truncated commutator identity.} Fix a collision-free
configuration $X_N\in\Omega_N$.  Let
$w\in L^\infty(\R^d)\cap\Lambda_*$ and let
\[
 \rho\in L^1(\R^d)\cap L^\infty(\R^d)
 \cap B^\alpha_{\infty,\infty}(\R^d),
 \qquad \alpha=s-d+2\in[1,2),
\]
be a density of unit mass.  For $\delta>0$ put
$u_\delta:=G_{s,\delta}*\rho$, with $G_{s,\delta}$ defined in
\eqref{eq:critical-riesz-truncated-kernel}, and define $C_N^{\delta}[w]$ by inserting the same cutoff
$|x-y|>\delta$ in every discrete, mixed, and continuous term.  Then
\begin{equation}\label{eq:critical-riesz-commutator-representation-truncated}
\begin{aligned}
C_N^{\delta}[w]
={}&\frac2N\sum_{i=1}^Nw(x_i)\cdot
 \bigl(K_i^{\delta}-u_\delta(x_i)\bigr)\\
&-\frac2N\sum_{i=1}^N\int_{|x-x_i|>\delta}
 w(x)\cdot\nabla g_s(x-x_i)\rho(x)\,\dd x
 +2\int w(x)\cdot u_\delta(x)\rho(x)\,\dd x,
\end{aligned}
\end{equation}
where
$K_i^{\delta}=N^{-1}\sum_{j\ne i,\ |x_i-x_j|>\delta}
\nabla g_s(x_i-x_j)$.

As $\delta\downarrow0$,
\begin{equation}\label{eq:critical-riesz-commutator-representation-pv}
\begin{aligned}
 C_N[w]
 ={}&\frac2N\sum_{i=1}^N
 w(x_i)\cdot(K_i-u(x_i))\\
 &-\frac2N\sum_{i=1}^N
 \operatorname{p.v.}\!\int_{\R^d}
 w(x)\cdot\nabla g_s(x-x_i)\rho(x)\,\dd x
 +2\int_{\R^d}w(x)\cdot u(x)\rho(x)\,\dd x,
\end{aligned}
\end{equation}
where every principal value and $u$ are defined by the same radial
truncation.
\item[(iii)] \emph{Equivalence with the principal value representation.} Assume $d-1\le s<d$ and the hypotheses and notation of
Proposition~\ref{prop:critical-riesz-regularization}.  Then the formula obtained by symmetrization agrees almost everywhere with
\eqref{eq:critical-riesz-commutator-representation-pv}, with the same centered
radial cutoff in every principal value.  One has
\begin{equation}\label{eq:critical-riesz-pv-time-majorant}
\begin{aligned}
 |C_N[w](t)|
 &\le C_{d,s}M_w\left(\left(\frac1N\sum_i|K_i(X_N(t))|^2\right)^{1/2}+B_T^2\right),\\
 \int_0^T\left(\frac1N\sum_i|K_i(X_N(t))|^2\right)^{1/2}\!\dd t
 &\le T^{1/2}H_N(X_N(0))^{1/2}.
\end{aligned}
\end{equation}
\item[(iv)] \emph{Estimate of the continuous term.} Let $1\le q\le\infty$, assume the corresponding Besov bound, and
use the function $\mathcal B_q(t)$ defined in \eqref{eq:critical-riesz-measurable-besov-bound}.  If $d-1\le s<d$, then $\sigma_s=d-s\in(0,1]$ and we put $w_\eps=u-u_\eps$.  Then, uniformly in the particle
center $a$,
\[
\left|\operatorname{p.v.}\!\int
 w_\eps(x)\rho(x)\cdot\nabla g_s(x-a)\,\dd x\right|
 \le C_{d,s,q}\mathcal B_q(t)^2e_{q,s}(\eps),
\]
and
\[
\left|\int w_\eps\cdot u\rho\right|
 \le C_{d,s,q}\mathcal B_q(t)^2\eps.
\]
\end{enumerate}
\end{proposition}

\begin{proof}
\emph{Part~(i).}  By oddness of $\nabla g_s$ on centered annuli,
\[
 u_\delta(t,a)=\int_{\delta<|z|<1}\nabla g_s(-z)
 [\rho_t(a+z)-\rho_t(a)]\,\dd z
 +\int_{|z|\ge1}\nabla g_s(-z)\rho_t(a+z)\,\dd z.
\]
Since $\alpha=s-d+2\ge1$, the density modulus
\eqref{eq:critical-riesz-density-modulus} bounds the omitted ball by
\[
 C\|\mathcal B_\infty\|_{L^\infty(0,T)}
 \int_0^\delta r^{d-s-1}\left(1+\log\frac1r\right)\,\dd r
 \le C\|\mathcal B_\infty\|_{L^\infty(0,T)}
 \delta^{d-s}(1+|\log\delta|).
\]
This proves \eqref{eq:critical-riesz-pv-density} and
\eqref{eq:critical-riesz-pv-field-uniform-rate}; the far part is absolutely
integrable because $|\nabla g_s|$ is bounded on $|z|\ge1$.  Applying the same
centered cancellation to a test function shows that the uniform limit equals
$\nabla g_s*\rho_t$ in distributions, hence agrees with its continuous
representative.

\emph{Part~(ii).}  For fixed $\delta$, with
$\nu_N=N^{-1}\sum_i\delta_{x_i}-\rho$, oddness gives the absolutely
convergent identity
\[
 C_N^\delta[w]=2\int w\cdot(G_{s,\delta}*\nu_N)\,\dd\nu_N.
\]
Expanding the outer measure and using
$(G_{s,\delta}*\nu_N)(x_i)=K_i^\delta-u_\delta(x_i)$ yields
\eqref{eq:critical-riesz-commutator-representation-truncated} with the stated
coefficients.  The mixed principal values exist because
\[
 |w(a+z)\rho(a+z)-w(a)\rho(a)|
 \le C|z|\left(1+\log\frac1{|z|}\right),
\]
so the radial majorant is $r^{d-s-1}(1+\log(1/r))\in L^1(0,1)$.  The
continuous--continuous term in the formula obtained by symmetrization is absolutely integrable by
the same estimate.  Letting $\delta\downarrow0$, absence of collisions gives
$K_i^\delta=K_i$ for small $\delta$, Part~(i) gives
$u_\delta\to u$ uniformly, and dominated convergence in the mixed and continuous terms after symmetrization proves
\eqref{eq:critical-riesz-commutator-representation-pv}.

\emph{Part~(iii).}  Part~(ii) identifies the centered cutoff limit with the
principal value expression, while
Proposition~\ref{prop:critical-riesz-regularization} identifies the same limit
and the even regularization limit with the absolutely convergent formula obtained by symmetrization.  Thus all three quantities coincide.
Moreover, the particle contribution is bounded by
$2M_w((N^{-1}\sum_i|K_i|^2)^{1/2}+\|u\|_\infty)$, and the preceding uniform
bound controls the mixed and continuous terms by
$CM_wB_T^2$.  Finally
\[
 \int_0^T\frac1N\sum_i|K_i(X_N(t))|^2\,\dd t
 \le H_N(X_N(0))
\]
follows from discrete energy dissipation, and Cauchy--Schwarz gives
\eqref{eq:critical-riesz-pv-time-majorant}.

\emph{Part~(iv).}  For $w_\eps=u-u_\eps$, the near-diagonal term after symmetrization
is bounded by the radial integral in
\eqref{eq:critical-three-regime-radial-remainder}; hence
\eqref{eq:critical-three-regime-radial-scales} gives
$C\mathcal B_q(t)^2e_{q,s}(\eps)$.  The regions $|z|\ge1/2$ are
$O(\mathcal B_q(t)^2\eps)$ and are absorbed by this bound.  Finally
\[
 \left|\int w_\eps\cdot u\rho\right|
 \le \|w_\eps\|_\infty\|u\|_\infty
 \le C\mathcal B_q(t)^2\eps,
\]
which proves the last assertion.
\end{proof}

\section{Coulomb particle dynamics}\label{app:coulomb-particle-dynamics}
The global collision-free flow is standard for repulsive Coulomb/Riesz
gradient systems \cite{Duerinckx2016,Serfaty1}; the proposition below records
the quantitative estimates used in the singular chain rule.

\begin{proposition}[Quantitative estimates for the Coulomb particle flow]
\label{prop:coulomb-particle-flow}
Let $d\ge2$, $N\ge2$, and $X_N^0\in\Omega_N$.  The particle system for
$g_d$ has a unique global smooth solution
$X_N\in C^\infty([0,\infty);\Omega_N)$ and
\begin{equation}\label{eq:coulomb-particle-dissipation}
 \frac{\dd}{\dd t}H_N(X_N(t))
 =-\frac1N\sum_{i=1}^N|K_i(X_N(t))|^2.
\end{equation}
More precisely:
\begin{enumerate}
\item[(i)] If $d\ge3$, then
\begin{equation}\label{eq:coulomb-particle-moment-high-d}
 \frac{\dd}{\dd t}m_{2,N}(X_N(t))=2(d-2)H_N(X_N(t)),
\end{equation}
\begin{equation}\label{eq:coulomb-high-d-separation}
 \min_{i\ne j}|x_i(t)-x_j(t)|
 \ge\Bigl((d-2)|\mathbb S^{d-1}|N^2H_N(X_N^0)\Bigr)^{-1/(d-2)},
\end{equation}
and
\[
 m_{2,N}(X_N(t))
 \le m_{2,N}(X_N^0)+2(d-2)tH_N(X_N^0).
\]
\item[(ii)] If $d=2$, then
\begin{equation}\label{eq:particle-moment-d2}
 m_{2,N}(X_N(t))=m_{2,N}(X_N^0)+\frac{N-1}{2\pi N}t.
\end{equation}
For every $T>0$, define
\[
 D_{N,T}:=\max\left\{1,
 2\sqrt{N\left(m_{2,N}(X_N^0)+\frac{N-1}{2\pi N}T\right)}\right\},
\]
\begin{equation}\label{eq:BN-planar}
 B_{N,T}:=\frac{N-1}{4\pi N}\log D_{N,T},
\end{equation}
and
\[
 A_{N,T}:=N^2H_N(X_N^0)
 +\left(\binom N2-1\right)\frac{\log D_{N,T}}{2\pi}.
\]
Then
\[
 \inf_{0\le t\le T}H_N(X_N(t))\ge-B_{N,T},
\]
\begin{equation}\label{eq:explicit-particle-separation}
 \min_{0\le t\le T}\min_{i\ne j}|x_i(t)-x_j(t)|
 \ge \min\{1,e^{-2\pi A_{N,T}}\},
\end{equation}
and
\begin{equation}\label{eq:particle-force-error-integrability}
 \int_0^T\frac1N\sum_{i=1}^N|K_i(X_N(t))|^2\,\dd t
 \le H_N(X_N^0)+B_{N,T}<\infty.
\end{equation}
\end{enumerate}
\end{proposition}

\begin{proof}
Global existence and absence of collisions for the repulsive Coulomb particle
flow are standard; see, for example, \cite{Duerinckx2016,Serfaty1}.  We record
the quantitative bounds needed in the singular limiting argument.  The
energy dissipation and virial identities follow from
Proposition~\ref{prop:unified-particle-identities}.

If $d\ge3$, then
\[
 -z\cdot\nabla g_d(z)=(d-2)g_d(z),
\]
so
\[
 \frac{\dd}{\dd t}m_{2,N}(X_N(t))=2(d-2)H_N(X_N(t)).
\]
Since $g_d\ge0$ and $H_N(t)\le H_N(0)$, every unordered pair satisfies
\[
 \frac1{N^2}g_d(x_i(t)-x_j(t))\le H_N(X_N^0),
\]
which gives \eqref{eq:coulomb-high-d-separation}.  Integrating the preceding
virial identity and using $H_N(t)\le H_N(0)$ gives the asserted second-moment
bound and \eqref{eq:coulomb-particle-moment-high-d}.

Suppose now $d=2$.  Since
\[
 -z\cdot\nabla g_2(z)=\frac1{2\pi},
\]
the discrete virial identity gives
\[
 \frac{\dd}{\dd t}m_{2,N}(X_N(t))=\frac{N-1}{2\pi N},
\]
and hence \eqref{eq:particle-moment-d2}.  For $0\le t\le T$,
\[
 |x_i(t)-x_j(t)|
 \le2\sqrt{N\left(m_{2,N}(X_N^0)+\frac{N-1}{2\pi N}T\right)}
 \le D_{N,T}.
\]
Because $g_2(z)=-(2\pi)^{-1}\log|z|$, every unordered pair is bounded below
by $-(2\pi)^{-1}\log D_{N,T}$.  Therefore
\[
 H_N(X_N(t))\ge-B_{N,T}.
\]
Fixing one pair and using this lower bound for all the remaining pairs gives
\[
 g_2(x_i(t)-x_j(t))
 \le N^2H_N(X_N^0)
 +\left(\binom N2-1\right)\frac{\log D_{N,T}}{2\pi}
 =A_{N,T},
\]
which proves \eqref{eq:explicit-particle-separation}.  Finally, integrating
\eqref{eq:coulomb-particle-dissipation} and using
$H_N(X_N(T))\ge-B_{N,T}$ gives
\[
 \int_0^T\frac1N\sum_{i=1}^N|K_i(X_N(t))|^2\,\dd t
 \le H_N(X_N^0)+B_{N,T},
\]
which is \eqref{eq:particle-force-error-integrability}.
\end{proof}

\section{Measurability and integrability for transported couplings}\label{app:measurability}
We establish the measurability and integrability properties used when the
pathwise estimates are averaged over a coupling transported by the two flows.
We first treat a countable family of fixed scales and then apply the result to
the Riesz case under the stated Besov regularity and the Coulomb case.

\begin{lemma}
\label{lem:parameterized-borel-integration}
Let $(\Theta,\mathcal A)$ be a standard Borel space and let
$f:\Theta\times\mathbb R^m\to\mathbb R$ be Borel.  If
$\int_{\mathbb R^m}|f(\theta,z)|\,\dd z<\infty$ for every $\theta$, then
\[
 \theta\longmapsto\int_{\mathbb R^m}f(\theta,z)\,\dd z
\]
is Borel.  The same conclusion holds for iterated integrals and for a Borel
family of probability densities $q(\theta,z)$ after replacing $f$ by $fq$.
If integrability holds only almost everywhere with respect to a Borel
probability measure on $\Theta$, the integral has a Borel representative after
modification on a Borel null set.
\end{lemma}

\begin{proof}
For nonnegative Borel kernels this is the standard theorem on integration of Borel kernels, obtained first for indicator rectangles and then by a monotone class
argument; apply it to the positive and negative parts.  Iterated integrals
follow recursively.  If absolute integrability holds only almost everywhere,
the set on which the integral is finite is Borel, and redefining the integral
on its Borel null complement gives the stated representative.
\end{proof}

\begin{proposition}
\label{prop:abstract-coupling-space-measurability}
Let $(\mathcal X,\mathcal F,\mathbb P)$ be a standard Borel probability space and
let $\mathcal D\subset(0,1)$ be countable.  Suppose that along a transported
coupling there are jointly Borel representatives
\[
 Q,\quad F,\quad R,\quad \mathcal C,\quad Q',\quad F',\quad
 R_\eps\qquad(\eps\in\mathcal D),
\]
such that $Q$ and $F$ are absolutely continuous in time for
$\mathbb P$-almost every initial datum.  Assume that, for every
$\eps\in\mathcal D$, the required transport estimate at the fixed scale parameter, chain rule,
commutator, and force estimates hold outside a null set, and that all terms
appearing in those estimates at fixed scales are integrable.  In particular, assume
\begin{equation}\label{eq:abstract-product-space-integrability}
 \mathbb E_{\mathbb P}\!\left[
 |Q(0)|+|F(0)|
 +\int_0^T\bigl(|Q'|+|F'|+R+R_\eps\bigr)\,\dd t
 \right]<\infty.
\end{equation}
Then all identities for $\eps\in\mathcal D$ hold on a common Borel set of
full measure in $[0,T]\times\mathcal X$, every displayed term may be averaged by
Tonelli--Fubini, and any Borel $\mathcal D$-valued map may be
inserted after the inequalities at fixed mollification scale have been averaged.
\end{proposition}

\begin{proof}
For each $\eps\in\mathcal D$, enlarge the exceptional set to a Borel null
set and intersect the complements.  Countability of $\mathcal D$ leaves a
Borel set of full measure on which the entire family of fixed scales is valid.
The integrability assumption gives Tonelli--Fubini for every term.  Composition with a Borel $\mathcal D$-valued map preserves measurability.
\end{proof}

\begin{proposition}
\label{prop:critical-riesz-borel-reference}
Let $\rho$ satisfy Definition~\ref{def:critical-riesz-reference-solution}.
There is a jointly Borel density representative of $\rho_t$ for every $t$,
which we continue to denote by $\rho(t,x)$, and it satisfies
\begin{equation}\label{eq:critical-riesz-all-time-Linfty}
 \sup_{0\le t\le T}\|\rho_t\|_{L^\infty}\le M_T.
\end{equation}
For every $1\le q\le\infty$ for which the corresponding Besov bound is
assumed, the dyadic Besov norm
$t\mapsto\|\rho_t\|_{B^\alpha_{\infty,q}}$ is Borel, finite for every time,
and bounded by its essential supremum in time.  The field
$u=\nabla g_s*\rho$ has a jointly Borel representative which, for every
$t$, is the bounded continuous Riesz field of
Proposition~\ref{prop:critical-riesz-field}.
\end{proposition}

\begin{proof}
Fix a nonnegative $\chi\in C_c^\infty$ of unit mass and set
$\rho_n(t,x)=\int\chi_{1/n}(x-y)\,\dd\rho_t(y)$.  Narrow continuity makes
$\rho_n$ jointly continuous.  If $t_k\to t$ is chosen from the set of full measure
 on which the $L^\infty$ bound holds, then for every nonnegative
$\varphi\in C_c$,
\[
 \int\varphi\,\dd\rho_t
 =\lim_k\int\varphi\rho_{t_k}\le M_T\int\varphi,
\]
so the same $L^\infty$ bound holds at every time.  For each fixed $t$, the preceding estimate shows that $\rho_t$ is absolutely
continuous with an $L^\infty$ density.  The standard approximation by mollifiers and
Lebesgue differentiation therefore imply
$\rho_n(t,x)\to\rho_t(x)$ for almost every $x$.  Consequently, after modifying the representative on a null set if necessary,
we may take
\[
 \rho(t,x):=\limsup_n\rho_n(t,x),
\]
which is jointly Borel and represents $\rho_t$ for every $t$, with the bound
\eqref{eq:critical-riesz-all-time-Linfty}.

Let $\varphi_j$ be the kernel of $\Delta_j$ and, for $j\ge-1$, put
$b_j(t)=2^{j\alpha}\|\Delta_j\rho_t\|_\infty$.  Since
$\varphi_j*\rho_t$ is continuous in $x$ and depends continuously on $t$ at
each $x$,
\[
 b_j(t)=2^{j\alpha}\sup_{x\in\mathbb Q^d}|(\varphi_j*\rho_t)(x)|
\]
is Borel.  If $t_n\to t$ through times at which the assumed Besov bound
holds, then $b_j(t)\le\liminf_n b_j(t_n)$.  Fatou in the dyadic index for
$q<\infty$, and the supremum in $j$ for $q=\infty$, extend the essential
bound to every $t$.

Let $\Theta_j$ be the kernel of
$\Delta_j\nabla g_s*$.  The functions
$U_j(t,x)=\int\Theta_j(x-y)\,\dd\rho_t(y)$ are jointly continuous.  The
uniform-in-time dyadic bound and Proposition~\ref{prop:critical-riesz-field} give,
for the high-frequency tail,
\[
 \sup_{0\le t\le T}\sup_x\sum_{j\ge J}|U_j(t,x)|
 \le C2^{-J}
 \sup_{0\le t\le T}\|\rho_t\|_{B^\alpha_{\infty,\infty}}
 \longrightarrow0.
\]
Hence the Littlewood--Paley series converges uniformly on
$[0,T]\times\mathbb R^d$.  Its limit is jointly Borel (indeed continuous
blockwise with a uniform tail) and, for every time, equals the bounded
continuous representative of the distributional Riesz field.
\end{proof}

\begin{proposition}
\label{prop:endpoint-measurability-statement}
Let $d\ge2$, $d-2<s<d$, let $\rho$ satisfy
Definition~\ref{def:critical-riesz-reference-solution}, and let $\pi_N^0$ be
a coupling of a symmetric particle law $\rho_N^0$ and
$\rho_0^{\otimes N}$ with finite normalized quadratic transportation cost, no collision
mass, and
\[
 \int\bigl(F_N(X_N,\rho_0)+\eta_N\bigr)\,\dd\rho_N^0(X_N)<\infty.
\]
Transport the coupling by the particle and limiting flows, using the
representatives of Proposition~\ref{prop:critical-riesz-borel-reference}.
Then
\[
 Q_N,\quad F_N,\quad r_N,\quad r_N^{2^{-j}}\ (j\ge4),\quad
 C_N[u],\quad Q_N',\quad F_N'
\]
have jointly Borel representatives.  When $d-1\le s<d$ the
commutator is given by the absolutely convergent formula obtained by symmetrization and agrees almost everywhere with both the centered cutoff limit and the even regularization limit.  For every fixed $j\ge4$,
\begin{equation}\label{eq:critical-riesz-product-space-integrability}
\begin{aligned}
\mathbb E_{\pi_N^0}\Bigg[&\bigl|Q_N(0)+F_N(X_N(0),\rho_0)+\eta_N+N^{s/d-1}\bigr|\\
&+\int_0^T\Big(|Q_N'|+|F_N'|+r_N+r_N^{2^{-j}}
 +C_{d,s}\bigl(\mathcal B_\infty(t)+\mathcal B_\infty(t)^2\bigr)\big[\omega_*(Q_N)
 +\ell(2^{-j})(F_N+\eta_N)\\
&\hspace{40mm}+2^{-j\min\{1,\sigma_s\}}\ell(2^{-j})+2^{-2j}\big]\Big)\,\dd t\Bigg]<\infty,
\end{aligned}
\end{equation}
If $1\le q<\infty$ and \eqref{eq:critical-riesz-finite-q-bound} holds, the same integrability statement holds with $C_{d,s,q}(\mathcal B_q+\mathcal B_q^2)$, $\omega_{\beta_q}$, $\ell(2^{-j})^{\beta_q}$, and $e_{q,s}(2^{-j})$ in place of the corresponding $q=\infty$ terms.
Consequently all dyadic identities required by
Corollary~\ref{cor:averaged-optimized-comparison} hold on a common
set of full measure.
\end{proposition}

\begin{proof}
The Borel representatives of $\rho$ and $u$ are provided by
Proposition~\ref{prop:critical-riesz-borel-reference}; the particle flow and
all finite sums are Borel.  The symmetrized commutator is a Borel
parameter integral because its near-diagonal majorant is
$C|z|^{-s}(1+\log(1/|z|))$, which is locally integrable for $s<d$, and its
far field is absolutely integrable.  Proposition~\ref{prop:critical-riesz-regularization}
identifies this representative with both the centered cutoff limit and the even regularization limit.

Lemma~\ref{lem:initial-microscopic-energy-statement}\textup{(ii)} and the
discrete energy dissipation identity give
\[
 \mathbb E_{\pi_N^0}\int_0^T\frac1N\sum_i|K_i|^2\,\dd t<\infty,
\]
and hence $r_N$ and every fixed $r_N^\eps$ are integrable with respect to
$\dd t\otimes\dd\pi_N^0$.  The particle second-moment bound in
Proposition~\ref{prop:riesz-global-particle-flow}, together with
$\rho\in C([0,T];\mathcal P_2)$, gives
$Q_N\in L^1(\dd t\otimes\dd\pi_N^0)$.  From the explicit trajectory identity,
\[
 |Q_N'|\le 2Q_N^{1/2}r_N^{1/2}
 +4\|u_t\|_{L^\infty}Q_N^{1/2},
\]
so Cauchy--Schwarz and $u\in L^2(0,T;L^\infty)$ yield
$Q_N'\in L^1(\dd t\otimes\dd\pi_N^0)$.

Expanding $F_N$ in terms of the discrete interaction energy and the bounded
Riesz potential gives $F_N\in L^1(\dd t\otimes\dd\pi_N^0)$.  A fixed
dyadic scale $\bar\eps$ gives the commutator integrability required for the
subsequent comparison.  The sharp first-order estimate gives
\[
 |C_N[u_{\bar\eps}]|
 \le C_{d,s}\|\nabla u_{\bar\eps}\|_{L^\infty}(F_N+\eta_N),
\]
which is integrable.  If $d-2<s<d-1$, the absolutely convergent identity
\eqref{eq:abs-rough-identity}, the near-field/far-field estimate, and
$r_N\in L^1$ show that $C_N[u-u_{\bar\eps}]\in L^1$.  If
$d-1\le s<d$, Proposition~\ref{prop:critical-riesz-pv-statement}\textup{(iii)}
gives the time majorant
\[
 |C_N[u-u_{\bar\eps}]|
 \le C_{d,s,\bar\eps,\rho}
 \left(\left(\frac1N\sum_i|K_i|^2\right)^{1/2}+1\right),
\]
and the square-root term belongs to $L^1(0,T)$ by Cauchy--Schwarz and the microscopic
dissipation bound.  Thus $C_N[u]\in L^1$ in the super-Coulomb Riesz regime, and
the singular chain rule yields $F_N'=-2r_N-C_N[u]\in L^1$.  This establishes the required integrability before the averaged comparison
inequality is invoked.  Since $\mathcal B_\infty$ is bounded and Borel and
$\omega_*(z)\le C(1+z)$, all remaining terms at fixed scale in
\eqref{eq:critical-riesz-product-space-integrability} are integrable.
Proposition~\ref{prop:abstract-coupling-space-measurability} completes the
proof.
\end{proof}

\begin{proposition}
\label{prop:coulomb-measurability-statement}
Under the hypotheses of Theorem~\ref{thm:bounded-density-coulomb-stability},
let $\pi_N^0$ be a symmetric coupling of the initial particle law and
$\rho_0^{\otimes N}$ with finite normalized quadratic transportation cost,
and transport it by the particle and limiting flows.  Put
$P_N=\widehat mQ_N$.  For $j\ge4$ define
\[
 \eps_j(t):=R(t)2^{-j},\qquad
 u_j(t):=u_t*\chi_{\eps_j(t)},\qquad
 r_N^{2^{-j}}:=r_N[u_j].
\]
Then
\[
 Q_N,\quad P_N,\quad F_N,\quad r_N,\quad
 r_N^{2^{-j}}\ (j\ge4),\quad C_N[u],\quad
 Q_N',\quad P_N',\quad F_N'
\]
have jointly Borel representatives as functions of time and the initial point
in the coupling.  Moreover, $Q_N$, $P_N$, and $F_N$ are absolutely
continuous along $\pi_N^0$-almost every coupled trajectory.  For every fixed
$j\ge4$,
\begin{equation}\label{eq:coulomb-product-space-integrability}
\begin{aligned}
 \mathbb E_{\pi_N^0}\Bigg[&
 \bigl|P_N(0)+F_N(X_N(0),\rho_0)+\eta+\zeta\bigr|\\
 &+\int_0^T\Big(
 |P_N'|+|F_N'|+r_N+r_N^{2^{-j}}\\
 &\hspace{17mm}
 +C_d m(t)\big[
 \omega_*(P_N)+\ell(2^{-j})(F_N+\eta)\big]
 +C_d m(t)^{2-2/d}\big[2^{-j}+2^{-2j}\big]\Big)\,\dd t
 \Bigg]<\infty,
\end{aligned}
\end{equation}
where $\eta=\eta_{N,d}$ and
\[
 \zeta=
 \begin{cases}
  \dfrac{1+\log N}{N},&d=2,\\[2mm]
  N^{-2/d},&d\ge3.
 \end{cases}
\]
Consequently all identities
for the countable family of fixed dimensionless dyadic scales hold on a common
Borel set of full $(\dd t\otimes\dd\pi_N^0)$-measure.  Differentiation may be
passed through the fixed expectation, and the measurable choice of dyadic scale may
be made only after averaging.  The corresponding mollification radius
$\eps_j(t)=R(t)2^{-j}$ is deterministic and Borel in time.
\end{proposition}

\begin{proof}
The particle flow is Borel on the collision-free set and may be extended
arbitrarily on the Borel collision set.  The uniform-in-time $L^\infty$ bound
and narrow continuity of $\rho_t$ give a jointly Borel density representative
by the same mollifier approximation as in
Proposition~\ref{prop:critical-riesz-borel-reference}; the Coulomb field has a
jointly Borel continuous representative by
Proposition~\ref{prop:bounded-density-flow}.  Since
$t\mapsto\eps_j(t)$ is deterministic and continuous, parameterized Borel
integration shows that $(t,x)\mapsto u_j(t,x)$ is jointly Borel for every
$j$.  Finite particle sums are therefore Borel, while the mixed and continuous
terms in $F_N$ and $C_N[u]$ are Borel by
Lemma~\ref{lem:parameterized-borel-integration} and the near-field/far-field bounds.  We
take $F_N'=-2r_N-C_N[u]$ as the Borel representative of the chain rule
derivative and use the explicit trajectory formula for $Q_N'$.  Since $\widehat m$ is $C^1$,
\[
 P_N'=\widehat m'Q_N
 +\widehat mQ_N'
\]
is also a jointly Borel representative.

It remains to verify integrability with respect to
$\dd t\otimes\dd\pi_N^0$.  Lemma~\ref{lem:initial-microscopic-energy-statement}\textup{(i)} gives
\[
 \mathbb E_{\pi_N^0}\int_0^T\frac1N\sum_i|K_i|^2\,\dd t<\infty.
\]
Since $u$ is bounded on $[0,T]$,
\[
 r_N\le \frac2N\sum_i|K_i|^2+2\|u_t\|_\infty^2,
\]
so $r_N$ is integrable.  For each fixed $j$, the mollifier estimate in Proposition~\ref{prop:coulomb-scale-covariant-field} gives
\[
 r_N^{2^{-j}}
 \le2r_N+2\|u-u_j\|_\infty^2
 \le2r_N+C_d m(t)^{2-2/d}2^{-2j},
\]
and hence $r_N^{2^{-j}}$ is integrable as well.

The particle second-moment bound and
$\rho\in C([0,T];\mathcal P_2)$ imply
$Q_N\in L^1(\dd t\otimes\dd\pi_N^0)$, and
\[
 |Q_N'|\le2Q_N^{1/2}r_N^{1/2}
 +4\|u_t\|_\infty Q_N^{1/2}
\]
gives $Q_N'\in L^1$ by Cauchy--Schwarz.  On every fixed finite interval,
$m$ and $m'$ are bounded, so the same is true for
$P_N$ and $P_N'$.

Writing $h_t=g_d*\rho_t$, the identity
\[
 2H_N=F_N+\frac2N\sum_i h_t(x_i)-E(\rho_t)
\]
and the Coulomb potential bounds give
$F_N\in L^1(\dd t\otimes\dd\pi_N^0)$; in dimension two use
$|h_t(x)|\le C_{T,\rho_0}(1+|x|^2)$ together with the integrability of
$H_N(0)^++B_{N,T}$.

A fixed-scale decomposition yields the commutator integrability needed for
averaging.  Fix $\bar\vartheta=2^{-4}$ and write
$u=u_{\bar\vartheta}+w_{\bar\vartheta}$ with mollification radius
$R(t)\bar\vartheta$.  The sharp first-order commutator estimate and
Proposition~\ref{prop:coulomb-scale-covariant-field} give
\[
 |C_N[u_{\bar\vartheta}]|
 \le C_d m(t)(F_N+\eta),
\]
which is integrable.  The absolutely convergent identity
\eqref{eq:abs-rough-identity}, Cauchy--Schwarz, and
\eqref{eq:coulomb-scale-pairing} give
\[
 |C_N[w_{\bar\vartheta}]|
 \le C_d m(t)^{1-1/d}
 \bigl(r_N^{1/2}+m(t)^{1-1/d}\bigr),
\]
which is integrable on $[0,T]$.  Thus $C_N[u]$ and
$F_N'=-2r_N-C_N[u]$ are integrable prior to the comparison argument.

Finally, $m,m^{2-2/d}\in L^1(0,T)$ by the explicit density envelope and
\eqref{eq:coulomb-linear-drift-integral},
$\omega_*(z)\le C(1+z)$, and $F_N+\eta\ge0$.  Hence all remaining terms in
\eqref{eq:coulomb-product-space-integrability} are integrable.
Proposition~\ref{prop:abstract-coupling-space-measurability} applied to the
countable dimensionless scale set $\{2^{-j}:j\ge4\}$ gives the common full
measure set, Tonelli--Fubini, differentiation under the fixed expectation,
and the measurable dyadic choice after averaging.
\end{proof}

\subsection{Integrability of the initial discrete interaction energy}

\begin{lemma}
\label{lem:initial-microscopic-energy-statement}
\begin{enumerate}
\item[(i)] \emph{Coulomb interaction in all dimensions.}
Let \(d\ge2\), let
\(\rho\in\mathcal P_2(\R^d)\cap L^\infty(\R^d)\), and let
\(\rho_N\in\mathcal P_2((\R^d)^N)\) be symmetric and give no mass to the
collision set.  Assume
\[
 \int\bigl(F_N(X_N,\rho)+\eta\bigr)\,\dd\rho_N(X_N)<\infty,
 \qquad F_N+\eta\ge0,
\]
where \(F_N\) is the modulated energy for \(g_d\) and
\(\eta<\infty\) is deterministic.  If \(d\ge3\), then
\[
 \mathbb E_{\rho_N}H_N(X_N)<\infty.
\]
If \(d=2\), then
\[
 \mathbb E_{\rho_N}H_N(X_N)^+<\infty,
 \qquad
 \mathbb E_{\rho_N}B_{N,T}<\infty
 \quad\text{for every }T>0,
\]
with \(B_{N,T}\) from \eqref{eq:BN-planar}.  In every dimension \(d\ge2\),
the global Coulomb particle flow satisfies
\begin{equation}\label{eq:expected-coulomb-microscopic-dissipation}
 \mathbb E_{\rho_N}\int_0^T
 \frac1N\sum_{i=1}^N|K_i(X_N(t))|^2\,\dd t<\infty
 \qquad(T>0).
\end{equation}

\item[(ii)] \emph{Riesz interaction.}
Let \(0<s<d\), let \(\rho\in\mathcal P_2\cap L^\infty\), and let
\(\rho_N\) be a symmetric probability law with no collision mass.  Assume
\[
 \int\bigl(F_N(X_N,\rho)+\eta\bigr)\,\dd \rho_N(X_N)<\infty
\]
for a deterministic finite-\(N\) term \(\eta\) such that
\(F_N+\eta\ge0\).  Then
\(\mathbb E_{\rho_N}H_N<\infty\).  It follows that the particle second
moment remains integrable on every finite time interval under
Proposition~\ref{prop:riesz-global-particle-flow}.
\end{enumerate}
\end{lemma}

\begin{proof}
\emph{Part~(i), \(d\ge3\).}
Let \(h=g_d*\rho\) and
\(E(\rho)=\int h\,\rho\).  The Coulomb near-field/far-field estimate gives
\[
 \|h\|_\infty+|E(\rho)|
 \le C_d(1+\|\rho\|_\infty).
\]
Expanding the modulated energy,
\[
 2H_N(X_N)
 =F_N(X_N,\rho)+\frac2N\sum_{i=1}^Nh(x_i)-E(\rho).
\]
Since \(g_d\ge0\), one has \(H_N\ge0\), and therefore
\[
 0\le2H_N(X_N)
 \le F_N(X_N,\rho)+\eta+2\|h\|_\infty+|E(\rho)|.
\]
The right-hand side is \(\rho_N\)-integrable, so
\(\mathbb EH_N(X_N)<\infty\).  Proposition~\ref{prop:coulomb-particle-flow}
and the discrete energy dissipation identity give pathwise
\[
 \int_0^T\frac1N\sum_i|K_i(X_N(t))|^2\,\dd t
 =H_N(X_N^0)-H_N(X_N(T))
 \le H_N(X_N^0).
\]
Tonelli's theorem yields
\eqref{eq:expected-coulomb-microscopic-dissipation}.

\emph{Part~(i), \(d=2\).}
Let \(h=g_2*\rho\).  Since \(g_2(z)\le0\) for
\(|z|\ge1\) and its positive near-field part is locally integrable,
\[
 \sup_{x\in\R^2}h(x)
 \le \frac{\|\rho\|_\infty}{2\pi}
 \int_{|z|<1}|\log|z||\,\dd z<\infty.
\]
The continuum logarithmic energy is absolutely finite by
\eqref{eq:auto-absolute-log-energy}.  The same expansion gives
\[
 H_N(X_N)
 \le\frac12\bigl(F_N(X_N,\rho)+\eta\bigr)
 +\sup_{x\in\R^2}h(x)+\frac12|E(\rho)|.
\]
Since $F_N+\eta\ge0$, this implies, for instance,
\[
 H_N(X_N)^+
 \le\frac12\bigl(F_N(X_N,\rho)+\eta\bigr)
 +\left|\sup_{x\in\R^2}h(x)\right|+\frac12|E(\rho)|,
\]
and hence \(H_N^+\in L^1(\rho_N)\).

Since \(\rho_N\in\mathcal P_2((\R^2)^N)\),
\(\mathbb E_{\rho_N}m_{2,N}(X_N)<\infty\).  From the definition of
\(D_{N,T}\),
\[
 \log D_{N,T}
 \le C_{N,T}\bigl(1+m_{2,N}(X_N^0)\bigr),
\]
and hence \(\mathbb E_{\rho_N}B_{N,T}<\infty\).  The planar part of
Proposition~\ref{prop:coulomb-particle-flow} gives pathwise
\[
 \int_0^T\frac1N\sum_i|K_i(X_N(t))|^2\,\dd t
 \le H_N(X_N^0)+B_{N,T}
 \le H_N(X_N^0)^++B_{N,T}.
\]
Tonelli's theorem again proves
\eqref{eq:expected-coulomb-microscopic-dissipation}.

\emph{Part~(ii).}
Let \(h_s=g_s*\rho\) and recall \(E_s\) from
\eqref{eq:continuum-riesz-energy}.  Expanding the modulated energy gives
\[
 2H_N(X_N)=F_N(X_N,\rho)
 +\frac2N\sum_{i=1}^Nh_s(x_i)-E_s(\rho).
\]
Since \(\rho\in L^1\cap L^\infty\) and \(0<s<d\), splitting the convolution
at unit distance gives, for every \(x\in\R^d\),
\[
\begin{aligned}
 |h_s(x)|
 &\le \frac1s\|\rho\|_\infty
 \int_{|z|<1}|z|^{-s}\,\dd z
 +\frac1s\int_{|z|\ge1}\rho(x-z)\,\dd z\\
 &\le \frac{|\mathbb S^{d-1}|}{s(d-s)}\|\rho\|_\infty
 +\frac1s\|\rho\|_1.
\end{aligned}
\]
Thus \(h_s\in L^\infty\) and
\(|E_s(\rho)|\le\|h_s\|_\infty\|\rho\|_1<\infty\).  Consequently,
\[
 2H_N(X_N)
 \le F_N(X_N,\rho)+\eta+2\|h_s\|_\infty+|E_s(\rho)|,
\]
whose right-hand side is \(\rho_N\)-integrable.  The moment assertion follows
from Proposition~\ref{prop:riesz-global-particle-flow}.
\end{proof}

\section{Singular chain rule for power-law Riesz kernels}
\label{app:lipschitz-riesz-chain-rule}
We apply the identities of Section~\ref{sec:abstract-framework} to the Riesz
kernel $g_s$.  It remains to prove collision exclusion and the three limits in
\eqref{eq:unified-chain-rule-limits}.

For reference, write
\begin{equation}\label{eq:riesz-particle-system-main}
 \dot x_i=-K_i^{g_s}(X_N),\qquad
 K_i^{g_s}(X_N)=\frac1N\sum_{j\ne i}\nabla g_s(x_i-x_j),
\end{equation}
and
\begin{equation}\label{eq:riesz-mean-field-main}
 \partial_t\rho=\nabla\!\cdot(\rho u),\qquad
 u=\nabla g_s*\rho.
\end{equation}
We also use
\begin{equation}\label{eq:continuum-riesz-energy}
 E_s(\rho):=\iint g_s(x-y)\rho(x)\rho(y)\,\dd x\dd y.
\end{equation}

\begin{proposition}[Quantitative estimates for the Riesz particle flow]
\label{prop:riesz-global-particle-flow}
Let $0<s<d$ and $X_N^0\in\Omega_N$.  Then \eqref{eq:riesz-particle-system-main} has a unique global smooth solution in $\Omega_N$.  With $H_N^{g_s}$ and $m_{2,N}$ from \eqref{eq:unified-force-energy-moment},
\begin{equation}\label{eq:riesz-particle-dissipation}
 \frac{\dd}{\dd t}H_N^{g_s}(X_N(t))
 =-\frac1N\sum_{i=1}^N|K_i^{g_s}(X_N(t))|^2,
\end{equation}
\begin{equation}\label{eq:riesz-particle-moment}
 \frac{\dd}{\dd t}m_{2,N}(X_N(t))=2sH_N^{g_s}(X_N(t)),
\end{equation}
and, for $N\ge2$,
\begin{equation}\label{eq:riesz-explicit-separation-moment}
 \min_{i\ne j}|x_i(t)-x_j(t)|
 \ge \bigl(sN^2H_N^{g_s}(X_N^0)\bigr)^{-1/s},\qquad
 m_{2,N}(X_N(t))\le m_{2,N}(X_N^0)+2stH_N^{g_s}(X_N^0).
\end{equation}
\end{proposition}

\begin{proposition}[Riesz singular chain rule]
\label{prop:riesz-singular-chain-rule}
Let $X_N\in C^1([0,T];\Omega_N)$ solve \eqref{eq:riesz-particle-system-main}, and let
$\rho\in C([0,T];\mathcal P_2(\R^d))\cap L^\infty(0,T;L^1\cap L^\infty)$ be transported by the Lagrangian flow of $-u$, where $u=\nabla g_s*\rho$ distributionally.  Assume
\begin{equation}\label{eq:riesz-abstract-chain-rule-assumptions}
 u\in L^2(0,T;L^\infty),\qquad
 |u_t(x)-u_t(y)|\le L(t)|x-y|\left(1+\log_+\frac1{|x-y|}\right),\qquad L\in L^1(0,T),
\end{equation}
and that the even mollifications satisfy
\begin{equation}\label{eq:riesz-mollified-field-convergence-assumption}
 u^\kappa=u*\psi_\kappa\longrightarrow u
 \quad\text{in }L^2(0,T;L^\infty(\R^d)).
\end{equation}
When $d-1\le s<d$, the singular commutator $C_N^{g_s}[u]$ is understood
through the absolutely convergent symmetrized representation of
Proposition~\ref{prop:critical-riesz-regularization}; equivalently, it is the
common centered principal value and even regularization limit identified in
Proposition~\ref{prop:critical-riesz-pv-statement}.  Then the three limits in
\eqref{eq:unified-chain-rule-limits} hold for $W=g_s$; in particular
\begin{equation}\label{eq:riesz-chain-rule-convergences}
 F_N^\kappa\to F_N^{g_s}\ \text{in }C([0,T]),\qquad
 D_N^\kappa\to r_N\ \text{in }L^1(0,T),\qquad
 C_N^{g_s^\kappa}[u]\to C_N^{g_s}[u]\ \text{in }L^1(0,T),
\end{equation}
so that
\begin{equation}\label{eq:riesz-modulated-identity}
 \frac{\dd}{\dd t}F_N^{g_s}(X_N(t),\rho_t)
 =-2r_N(t)-C_N^{g_s}[u_t]
 \qquad\text{for a.e. }t\in[0,T].
\end{equation}
In the Besov class of Theorem~\ref{thm:critical-riesz-propagation}, these
hypotheses follow from Proposition~\ref{prop:critical-riesz-field} and the
Osgood pushforward representation.  The Coulomb case with bounded density is
treated directly in Proposition~\ref{prop:coulomb-singular-chain-rule}.
\end{proposition}

\begin{proof}[Proof of Proposition~\ref{prop:riesz-global-particle-flow}]
Global smoothness and collision avoidance for repulsive Riesz particle flows are
standard; see \cite{Duerinckx2016,NguyenRosenzweigSerfaty2022}.  We record the
quantitative bounds used below.  Proposition~\ref{prop:unified-particle-identities}
gives \eqref{eq:riesz-particle-dissipation} and
\eqref{eq:riesz-particle-moment}.  Since $H_N^{g_s}$ is nonincreasing and every
pair contribution is nonnegative,
\[
 \frac1{sN^2}|x_i(t)-x_j(t)|^{-s}
 \le H_N^{g_s}(X_N(t))\le H_N^{g_s}(X_N^0),
\]
which gives the separation estimate in
\eqref{eq:riesz-explicit-separation-moment}; integrating
\eqref{eq:riesz-particle-moment} gives the stated moment bound.
\end{proof}

\begin{proof}[Proof of Proposition~\ref{prop:riesz-singular-chain-rule}]
Before the singular limiting argument, we make explicit a consequence of the
time regularity of $\rho$ that is used below.  Put
\[
 M:=\operatorname*{ess\,sup}_{0<t<T}\|\rho_t\|_{L^\infty}<\infty.
\]
Since $t\mapsto\rho_t$ is narrowly continuous, the same density bound holds at
every time after choosing the canonical $L^\infty$ representative.  Indeed,
for fixed $t\in[0,T]$, choose $t_n\to t$ from the set of full measure on which
$\|\rho_{t_n}\|_\infty\le M$.  Then, for every nonnegative
$\varphi\in C_c(\mathbb R^d)$,
\[
 \int \varphi\,\dd\rho_t
 =\lim_{n\to\infty}\int \varphi\rho_{t_n}
 \le M\int\varphi,
\]
so $\rho_t$ has a density bounded by $M$.  In particular, all
uniform-in-time convolution estimates used in Step~4 are justified by the
hypotheses of the proposition rather than by an additional pointwise in time
assumption.

\emph{Step 1: time integrability of $r_N$.}
The discrete energy dissipation identity gives
\[
 \int_0^T\frac1N\sum_i|K_i(X_N(t))|^2\,\dd t
 \le H_N(X_N(0)).
\]
Consequently,
\[r_N(t)\le\frac2N\sum_i|K_i(X_N(t))|^2+2\|u(t)\|_\infty^2
 \in L^1(0,T).
\]

\emph{Step 2: identity at fixed scale.}
Let $\zeta\in C_c^\infty(B_1)$ be nonnegative, radial, and of mass one.  Set
$\psi_\kappa=\zeta_\kappa*\zeta_\kappa$ and $g_s^\kappa=g_s*\psi_\kappa$.
The quantities $F_N^\kappa$, $D_N^\kappa$, and $C_N^{g_s^\kappa}$ are those
defined in Section~\ref{sec:abstract-framework}.  For each fixed $\kappa>0$,
the regularized force $\nabla g_s^\kappa$ is bounded and smooth at the origin
and remains bounded at infinity.  Hence $K_i^\kappa$ and $u^\kappa$ are bounded
for this fixed scale.  The terms produced by differentiating $F_N^\kappa$ are
therefore integrable in time: the factors involving the unregularized particle
velocity are controlled by
$N^{-1}\sum_i|K_i|^2\in L^1(0,T)$, while the continuum factors are controlled
by $u\in L^2(0,T;L^\infty)$.  Thus the fixed scale admissibility hypotheses of
Proposition~\ref{prop:unified-regularized-chain-rule}\textup{(i)} are satisfied.
Applying that proposition to the Lagrangian limiting solution yields
\begin{equation}\label{eq:riesz-regularized-exact-identity}
 \frac\dd{\dd t}F_N^\kappa=-2D_N^\kappa-C_N^\kappa[u].
\end{equation}

\emph{Step 3: force term.}
A continuous collision-free trajectory has positive minimal separation on
$[0,T]$.  Consequently,
\[
 \delta_K(\kappa):=\max_i\sup_{t\le T}|K_i^\kappa(t)-K_i(t)|\to0.
\]
Writing $f_i=K_i-u(x_i)$ and
$e_i^\kappa=(K_i^\kappa-K_i)-(u^\kappa-u)(x_i)$, Cauchy--Schwarz gives
\begin{equation}\label{eq:riesz-cross-limit-bound}
\begin{aligned}
 \|D_N^\kappa-r_N\|_{L^1(0,T)}
 \le \|r_N^{1/2}\|_{L^2(0,T)}
 \left(T^{1/2}\delta_K(\kappa)
 +\|u^\kappa-u\|_{L_t^2L_x^\infty}\right)\longrightarrow0.
\end{aligned}
\end{equation}

\emph{Step 4: energy.}
The discrete--discrete contribution converges uniformly by separation.  For
the mixed contribution, local $L^1$ approximation of the locally integrable
kernel and the uniform density bound control $|z|<2$, while evenness and
Taylor's formula give
$|g_s^\kappa(z)-g_s(z)|\le C_{d,s}\kappa^2|z|^{-s-2}$ for $|z|\ge1$.
It follows that
\[\sup_{0\le t\le T}\|(g_s^\kappa-g_s)*\rho_t\|_\infty\longrightarrow0.
\]
The same supremum controls the continuous--continuous term, so
\begin{equation}\label{eq:riesz-regularized-energy-limit}
 \sup_{0\le t\le T}|F_N^\kappa(t)-F_N(t)|\longrightarrow0.
\end{equation}

\emph{Step 5: commutator.}
The discrete--discrete part is a finite sum away from the diagonal.  For a
mixed term centered at $a$, the contribution of $|z|<\delta$ to either the
regularized expression, as well as the singular formula obtained by symmetrization, is bounded uniformly
in $a$ by
\[
 C\|\rho\|_{L_t^\infty L_x^\infty}L(t)
 \int_0^\delta r^{d-s-1}\left(1+\log\frac1r\right)\,\dd r.
\]
This tends to zero in $L^1(0,T)$ because $s<d$.  On $|z|\ge\delta$ the
regularized kernels converge uniformly and the integrands are dominated by
$C_{\delta,d,s}(L(t)+\|u(t)\|_\infty)$ times a probability density.  Dominated
convergence proves the mixed limit uniformly in $a$ and, after integration in
the outer variable, the continuous--continuous limit.  Consequently,
\begin{equation}\label{eq:riesz-regularized-commutator-limit}
 C_N^\kappa[u]\longrightarrow C_N[u]
 \quad\text{in }L^1(0,T).
\end{equation}
The limit is the absolutely convergent formula for the commutator obtained by symmetrization.

\emph{Step 6: passage to the limit.}
Integrate \eqref{eq:riesz-regularized-exact-identity} between two times and
let $\kappa\downarrow0$.  The convergences
\eqref{eq:riesz-cross-limit-bound}, \eqref{eq:riesz-regularized-energy-limit},
and \eqref{eq:riesz-regularized-commutator-limit} yield
\eqref{eq:riesz-chain-rule-convergences}.  The function $F_N$ is therefore
absolutely continuous and satisfies \eqref{eq:riesz-modulated-identity}.
\end{proof}

\section{Finite-\texorpdfstring{$N$}{N} lower bounds, commutator estimates, and normalization}\label{sec:static-inputs}

We state the finite-$N$ lower bounds and first-order commutator estimates in
the normalization used above and track the conversion of the kernel, energy,
commutator, and density scales explicitly.  The power-law and logarithmic
dilations below extend the cited estimates across the relevant microscopic
scale regimes.

\paragraph{Normalization conventions.}
The conventions in the cited estimates are as follows.  In
\cite{HessChildsRosenzweigSerfaty2025} the power-law kernel is
$g_s(x)=s^{-1}|x|^{-s}$, the empirical charge is
$N^{-1}\sum_i\delta_{x_i}-\mu$, and the modulated energy contains the factor
$1/2$.  The sharp lower bound has additive scale
$\|\mu\|_\infty^{s/d}N^{s/d-1}$.  In the Coulomb and super-Coulomb range,
\cite[Theorem~1.1, equation~(1.7)]{Rosenzweig11222} uses the same Riesz
normalization, the microscopic length $\lambda=(N\|\mu\|_\infty)^{-1/d}$,
and the full first-order commutator; the energy again has the factor $1/2$,
and the logarithmic case includes the correction $-(\log\lambda)/(2N)$.  For the planar logarithmic kernel we keep separate the
unnormalized kernel $g_S=-\log|x|$ and our fundamental solution
$g=(2\pi)^{-1}g_S$.  The factors $2\pi$, $N^2$, and $1/2$ are kept explicit throughout.  The power-law dilation below preserves the commutator
norm and sends $\lambda$ to $L\lambda$; in the logarithmic dilation the
constant shift of the kernel is retained because diagonal exclusion gives a
nonzero off-diagonal mass for the atomic part.  The formulas in the following
subsections implement these conversions explicitly.

\subsection{Normalization of the sharp power-law Riesz estimates}
Let
\[
 \nu_N:=\frac1N\sum_{i=1}^N\delta_{x_i}-\mu,
 \qquad g_s(x)=\frac1s|x|^{-s},\qquad 0<s<d.
\]
Both sharp estimates used below employ this same normalization of the power-law
Riesz kernel.  Their modulated energy, however, contains a factor \(1/2\).
Accordingly, introduce
\[
 \widetilde F_N(X_N,\mu)
 :=\frac12{\iint}_{(\R^d)^2\setminus\Delta_2}
 g_s(x-y)\,\dd\nu_N(x)\dd\nu_N(y),
\]
and the full first-order commutator
\[
 \widetilde I_{N,s}[v]
 :={\iint}_{(\R^d)^2\setminus\Delta_2}
 (v(x)-v(y))\cdot\nabla g_s(x-y)
 \,\dd\nu_N(x)\dd\nu_N(y).
\]
Thus our normalization is
\begin{equation}\label{eq:static-normalization-conversion}
 F_N=2\widetilde F_N,\qquad C_N[v]=\widetilde I_{N,s}[v].
\end{equation}

We use the sharp lower bound of Hess-Childs, Rosenzweig, and Serfaty for
$0<s<d$, and the first-order commutator estimate of Rosenzweig and Serfaty in
the Coulomb and super-Coulomb range $d-2\le s<d$.  The latter depends on the transport field through
$\|\nabla v\|_{L^\infty}$ and is initially stated for
$\lambda_N<1$.  The dilation below extends the estimate to arbitrary
$\lambda_N$.  The planar logarithmic
case is treated separately in the final two subsections of this appendix.

\begin{lemma}
\label{lem:static-riesz-transfer}
Let $d\ge2$, $0<s<d$, let $\mu\in L^1\cap L^\infty$ be a density of unit mass, and let $X_N\in\Omega_N$ be pairwise distinct.  Set
\[
 \lambda_N:=(N\|\mu\|_{L^\infty})^{-1/d},
 \qquad
 e_{N,s}(\mu):=\|\mu\|_{L^\infty}^{s/d}N^{s/d-1}
 =\|\mu\|_{L^\infty}\lambda_N^{d-s}.
\]
Since \(0<s<d\), the Riesz energy of \(\mu\) is finite under
\(\mu\in L^1\cap L^\infty\): splitting into near and far regions uses local integrability
of \(|x|^{-s}\) and boundedness of \(g_s\) on \(\{|x|\ge1\}\).  Hence the
finite background energy hypothesis in the cited finite-$N$ functional inequalities is automatic
in this power-law setting.
After enlarging constants depending only on the displayed parameters, the sharp
estimates from the cited works take the following form in our normalization.
\begin{enumerate}
\item[(i)] The sharp lower bound of Hess-Childs, Rosenzweig, and Serfaty
\cite[Proposition~2.11 and Remark~2.13]{HessChildsRosenzweigSerfaty2025},
taken at $p=\infty$, yields
\[
 F_N(X_N,\mu)+C_{d,s}e_{N,s}(\mu)\ge0.
\]
\item[(ii)] If $d-2\le s<d$, the first-order global estimate of
Rosenzweig and Serfaty
\cite[Theorem~1.1, equation~(1.7)]{Rosenzweig11222}, together with the
dilation argument below, yields
\[
 |C_N(X_N,\mu;v)|
 \le C_{d,s}\|\nabla v\|_{L^\infty}
 \bigl(F_N(X_N,\mu)+C_{d,s}e_{N,s}(\mu)\bigr).
\]
The argument uses this first-order estimate, whose transport dependence is
through $\|\nabla v\|_{L^\infty}$.
The constants are independent of $N$ and of the minimum particle separation.
\end{enumerate}
\end{lemma}

\begin{proof}
\emph{Normalization of the energy and the first-order commutator.}
In the notation of the cited estimates one has
\[
 \widetilde F_N
 =\frac12{\iint}_{(\R^d)^2\setminus\Delta_2}
 g_s(x-y)\,\dd\nu_N(x)\dd\nu_N(y),
 \qquad
 \widetilde I_{N,s}[v]
 ={\iint}_{(\R^d)^2\setminus\Delta_2}
 (v(x)-v(y))\cdot\nabla g_s(x-y)\,\dd\nu_N(x)\dd\nu_N(y).
\]
Moreover,
\begin{equation}\label{eq:static-source-first-variation}
 \left.\frac{\dd}{\dd t}\right|_{t=0}
 \widetilde F_N((\mathrm{Id}+tv)^{\oplus N}X_N,
                (\mathrm{Id}+tv)_\#\mu)
 =\frac12\widetilde I_{N,s}[v].
\end{equation}
Consequently, our modulated energy satisfies
\[
 \left.\frac{\dd}{\dd t}\right|_{t=0}
 F_N((\mathrm{Id}+tv)^{\oplus N}X_N,
     (\mathrm{Id}+tv)_\#\mu)
 =C_N[v],
\]
which is the normalization used in the dynamical chain rule.

To make the conversion of the additive term explicit, suppose first that a cited
lower bound is written as
\[
 \widetilde F_N+B\,e_{N,s}(\mu)\ge0.
\]
Using \eqref{eq:static-normalization-conversion},
\[
 F_N+2B\,e_{N,s}(\mu)
 =2\bigl(\widetilde F_N+B\,e_{N,s}(\mu)\bigr)\ge0.
\]
Likewise, a cited first-order bound of the form
\[
 |\widetilde I_{N,s}[v]|
 \le A\,\mathcal V(v)
 \bigl(\widetilde F_N+B\,e_{N,s}(\mu)\bigr)
\]
becomes
\[
 |C_N[v]|
 \le \frac A2\,\mathcal V(v)
 \bigl(F_N+2B\,e_{N,s}(\mu)\bigr).
\]
The factor \(1/2\) changes the numerical constants in the finite-\(N\)
correction and the outer prefactor, while leaving the power of \(N\) unchanged.  Finally,
\[
 \|\mu\|_{L^\infty}\lambda_N^{d-s}
 =\|\mu\|_{L^\infty}^{s/d}N^{s/d-1},
\]
so the additive error term in the cited estimate has exactly the scale \(e_{N,s}(\mu)\) used here.  This proves
part~\textup{(i)}, and proves part~\textup{(ii)} whenever
\(\lambda_N<1\).

\emph{Removal of the scale restriction in the Coulomb and super-Coulomb estimate.}
The global estimate of Rosenzweig and Serfaty used in part~\textup{(ii)} is
stated under the condition \(\lambda_N<1\).  If this condition is not already
satisfied, choose \(L>0\) so that \(L\lambda_N<1\) and define
\[
 x_i^L=Lx_i,\qquad
 \mu_L(x)=L^{-d}\mu(x/L),\qquad
 v_L(x)=Lv(x/L).
\]
Then \(\mu_L\) has unit mass and
\[
 F_N(X_N^L,\mu_L)=L^{-s}F_N(X_N,\mu),\qquad
 C_N(X_N^L,\mu_L;v_L)=L^{-s}C_N(X_N,\mu;v),
\]
while
\[
 \|\nabla v_L\|_\infty=\|\nabla v\|_\infty,\qquad
 e_{N,s}(\mu_L)=L^{-s}e_{N,s}(\mu),\qquad
 \lambda_N(\mu_L)=L\lambda_N(\mu)<1.
\]
Applying the cited estimate to \((X_N^L,\mu_L,v_L)\) and multiplying by \(L^s\)
therefore gives exactly the estimate in part~\textup{(ii)} for the original
configuration, with constants independent of \(L\).

\end{proof}

The normalized lower bound and first-order commutator estimates used in the
applications are exactly the conclusions of
Lemma~\ref{lem:static-riesz-transfer}, after choosing the constants in $\eta_N$ large enough to dominate the additive errors in that lemma.

\subsection{Planar logarithmic normalization}
Set
\[
 g_S(x):=-\log|x|,
 \qquad
 g(x):=\frac1{2\pi}g_S(x)=-\frac1{2\pi}\log|x|,
 \qquad
 \nu_N:=\frac1N\sum_{i=1}^N\delta_{x_i}-\mu.
\]
There are two normalizations in the cited results to distinguish.  In \cite{Serfaty1}, the
logarithmic modulated energy is written with the unnormalized neutral charge
\[
 \Sigma_N:=\sum_{i=1}^N\delta_{x_i}-N\mu=N\nu_N
\]
and kernel $g_S$.  If $\mathcal F_N^S$ denotes that off-diagonal energy, then
\begin{equation}\label{eq:planar-serfaty-normalization}
 \mathcal F_N^S
 :=\iint_{(\R^2)^2\setminus\Delta_2}
 g_S(x-y)\,\dd\Sigma_N(x)\dd\Sigma_N(y),
 \qquad
 F_N=\frac1{2\pi N^2}\mathcal F_N^S.
\end{equation}
In the normalization of Rosenzweig--Serfaty \cite{Rosenzweig11222},
write
\[
 \widetilde F_{N,0}
 :=\frac12\iint_{(\R^2)^2\setminus\Delta_2}
 g_S(x-y)\,\dd\nu_N(x)\dd\nu_N(y)
\]
and denote by $\widetilde I_{N,0}[v]$ the corresponding full first-order
commutator,
\[
 \widetilde I_{N,0}[v]
 :=\iint_{(\R^2)^2\setminus\Delta_2}
 (v(x)-v(y))\cdot\nabla g_S(x-y)
 \,\dd\nu_N(x)\dd\nu_N(y).
\]
Our normalization is therefore
\begin{equation}\label{eq:planar-rs-normalization}
 F_N=\pi^{-1}\widetilde F_{N,0},
 \qquad
 C_N[v]=(2\pi)^{-1}\widetilde I_{N,0}[v].
\end{equation}
Thus the factor $1/2$ in the Rosenzweig--Serfaty energy is absent from our
$F_N$, while their commutator is already the full first-order commutator.  We use
these identities explicitly in the next two lemmas.

\subsection{Logarithmic Coulomb lower bound and commutator estimate}
We use the following two finite-$N$ estimates in the planar logarithmic
Coulomb argument.  The power-law estimates above include the
higher-dimensional Coulomb case $s=d-2$; the planar logarithmic normalization
is treated separately below.

\begin{lemma}
\label{lem:diagonal-excluded-lower-bound}
Let $\mu\in\mathcal P_2(\R^2)\cap L^\infty(\R^2)$, and let
$X_N\in(\R^2)^N\setminus\Delta_N$.  There exists a constant $C>0$ such that
\begin{equation}\label{eq:planar-normalized-lower-bound}
 F_N(X_N,\mu)
 \ge -C\frac{1+\log(1+N\|\mu\|_{L^\infty})}{N}.
\end{equation}
After choosing the $d=2$ constant in \eqref{eq:coulomb-finite-N-correction} sufficiently large,
\[
 F_N(X_N,\mu)+\eta_{N,2}(\mu)\ge0.
\]
\end{lemma}

\begin{proof}[Normalization]
Let $M:=\|\mu\|_{L^\infty}$.  Since $\mu$ has unit mass, $M>0$.
The energy used in \cite{Serfaty1} is $\mathcal F_N^S$ from
\eqref{eq:planar-serfaty-normalization}.  Set
\[
 L:=M^{1/2},\qquad x_i^L:=Lx_i,\qquad
 \mu_L(x):=L^{-2}\mu(x/L),
 \qquad \|\mu_L\|_{L^\infty}=1.
\]
Because $g_S(Lz)=g_S(z)-\log L$, the logarithmic scaling must be computed
with the diagonal exclusion kept visible.  The charge $\Sigma_N$ is neutral,
but the $N$ atomic self-interaction pairs are removed, and hence
\[
 \iint_{(\R^2)^2\setminus\Delta_2}
 1\,\dd\Sigma_N(x)\dd\Sigma_N(y)=-N.
\]
Consequently
\begin{equation}\label{eq:planar-log-energy-dilation}
 \mathcal F_N^S(X_N^L,\mu_L)
 =\mathcal F_N^S(X_N,\mu)+N\log L.
\end{equation}
Apply Corollary~3.5 of \cite{Serfaty1} in dimension two and in the logarithmic
case to $\mu_L$.  By construction its $L^\infty$ norm is one, so the corollary gives
\[
 \mathcal F_N^S(X_N^L,\mu_L)
 \ge -\frac N2\log N-CN,
\]
where $C$ is numerical because $\|\mu_L\|_{L^\infty}=1$.
Using \eqref{eq:planar-log-energy-dilation} and
$N\log L=(N/2)\log M$ yields
\begin{equation}\label{eq:planar-log-source-lower-bound-explicit}
 \mathcal F_N^S(X_N,\mu)
 \ge -\frac N2\log(NM)-CN.
\end{equation}
If $NM<1$, the logarithmic contribution is nonnegative and may be discarded.
If $NM\ge1$, it is controlled by $N\log_+(NM)$.  Thus, in both cases,
\[
 \mathcal F_N^S(X_N,\mu)
 \ge -CN\bigl(1+\log_+(NM)\bigr)
 \ge -CN\bigl(1+\log(1+NM)\bigr).
\]
Dividing by $2\pi N^2$ according to
\eqref{eq:planar-serfaty-normalization} proves
\eqref{eq:planar-normalized-lower-bound}.
\end{proof}

\begin{lemma}
\label{lem:serfaty-lipschitz-commutator}
Let $\mu\in\mathcal P_2(\R^2)\cap L^\infty(\R^2)$, let
$X_N\in(\R^2)^N\setminus\Delta_N$, and put
$\nu_N=N^{-1}\sum_i\delta_{x_i}-\mu$.  For every
$v\in W^{1,\infty}(\R^2;\R^2)$,
\begin{equation}\label{eq:planar-normalized-lipschitz-commutator}
\left|\iint_{(\R^2)^2\setminus\Delta_2}
 (v(x)-v(y))\cdot\nabla g(x-y)
 \,\dd\nu_N(x)\dd\nu_N(y)\right|
\le C\|\nabla v\|_{L^\infty}
 \bigl(F_N(X_N,\mu)+\eta_{N,2}(\mu)\bigr).
\end{equation}
The estimate holds for every $N\ge2$.
\end{lemma}

\begin{proof}[Normalization and removal of the restriction on $\lambda$]
Set $M=\|\mu\|_{L^\infty}>0$ and
$\lambda=(NM)^{-1/2}$.  First suppose $\lambda<1$.  The logarithmic case of
\cite[Theorem~1.1, estimate~(1.7)]{Rosenzweig11222} uses
$g_S=-\log|\cdot|$, the energy $\widetilde F_{N,0}$ defined with the factor $1/2$, and the full
first-order commutator $\widetilde I_{N,0}[v]$.  Its hypotheses include absolute
integrability of the logarithmic background energy, which follows from $\mu\in\mathcal P_2\cap L^\infty$ and
\eqref{eq:auto-absolute-log-energy}.  In the normalization of the cited theorem,
\[
 |\widetilde I_{N,0}[v]|
 \le C\|\nabla v\|_{L^\infty}
 \left(\widetilde F_{N,0}
 -\frac{\log\lambda}{2N}+CM\lambda^2\right).
\]
Using the identities in the planar normalization subsection gives
\[
 |C_N[v]|
 \le C\|\nabla v\|_{L^\infty}
 \left[F_N+C\left(\frac{-\log\lambda}{N}+M\lambda^2\right)\right].
\]
Since $-\log\lambda=\frac12\log(NM)$ and $M\lambda^2=N^{-1}$,
\[
 \frac{-\log\lambda}{N}+M\lambda^2
 \le C\frac{1+\log(1+NM)}{N},
\]
which proves the claim in this case after choosing the $d=2$ constant in
\eqref{eq:coulomb-finite-N-correction} sufficiently large.

It remains to consider $\lambda\ge1$, equivalently $NM\le1$.  Set
\[
 L:=\frac{\sqrt{NM}}2\in(0,1/2],\qquad
 x_i^L=Lx_i,\qquad
 \mu_L(x)=L^{-2}\mu(x/L),\qquad
 v_L(x)=Lv(x/L).
\]
Then $\|\mu_L\|_\infty=4/N$, so the scaled parameter satisfies
$\lambda_L=(N\|\mu_L\|_\infty)^{-1/2}=1/2<1$.  Moreover
\begin{equation}\label{eq:planar-commutator-dilation}
 C_N(X_N^L,\mu_L;v_L)=C_N(X_N,\mu;v),\qquad
 \|\nabla v_L\|_\infty=\|\nabla v\|_\infty,
\end{equation}
and the logarithmic energy with the diagonal removed satisfies
\begin{equation}\label{eq:planar-normalized-energy-dilation}
 F_N(X_N^L,\mu_L)
 =F_N(X_N,\mu)+\frac{\log L}{2\pi N}.
\end{equation}
Applying the already proved $\lambda_L<1$ estimate to the scaled data gives
\[
 |C_N(X_N,\mu;v)|
 \le C\|\nabla v\|_\infty
 \left(F_N(X_N,\mu)+\frac{\log L}{2\pi N}+\frac{C}{N}\right).
\]
Since $L\le1$ the logarithmic shift is nonpositive, and hence the last
right-hand side is bounded by
$C\|\nabla v\|_\infty(F_N+C/N)$.  Because $NM\le1$,
\[
 \eta_{N,2}(\mu)
 =A_2\frac{1+\log(1+NM)}{N}\ge\frac{A_2}{N}.
\]
Choosing the fixed constant $A_2$ sufficiently large proves
\eqref{eq:planar-normalized-lipschitz-commutator} also in this regime.  The lower bound $F_N+\eta_{N,2}\ge0$ ensures that the final right-hand side is nonnegative, and the estimate therefore holds for all $N\ge2$.
\end{proof}

\section*{Acknowledgments}
The authors are grateful to Zhenfu Wang for helpful discussions.

\section*{Statements and Declarations}
\noindent\emph{Funding.}
This work was supported by the National Natural Science Foundation of China
(grant no.~12371180).

\noindent\emph{Competing interests.}
The authors declare that they have no competing interests.

\noindent\emph{Data availability.}
No datasets were generated or analyzed during the current study.

\end{document}